\documentclass[a4paper,12pt]{amsart}
\usepackage{fullpage}
\usepackage{amsmath,amsfonts,amssymb,amsthm}
\usepackage{graphics}
\usepackage{epsfig}
\usepackage{esint}
\usepackage{hyperref}
\usepackage{enumerate}
\usepackage{color}
\usepackage{mathtools}
\usepackage{nicefrac}
\newtheorem{definition}{Definition}
\newtheorem{proposition}[definition]{Proposition}
\newtheorem{theorem}[definition]{Theorem}
\newtheorem{remark}[definition]{Remark}
\newtheorem{lemma}[definition]{Lemma}

\newtheorem{assumption}[definition]{Assumption}

\usepackage{tikz}

\usepackage{float}
\floatstyle{boxed} 
\restylefloat{figure}

\newcommand{\ts}{\hspace{0.5pt}}

\newcommand{\R}{\mathbb{R}\ts}

\newcommand{\bS}{\mathbb S}

\newcommand{\wto}{\rightharpoonup}

\newcommand{\ep}{\varepsilon}
\newcommand{\eps}{\ep}

\newcommand{\Om}{\Omega}

\newcommand{\supp}{\mathrm{supp}}

\newcommand{\Hmm}[1]{\leavevmode{\marginpar{\tiny%
      $\hbox to 0mm{\hspace*{-0.5mm}$\leftarrow$\hss}%
      \vcenter{\vrule depth 0.1mm height 0.1mm width \the\marginparwidth}%
      \hbox to 0mm{\hss$\rightarrow$\hspace*{-0.5mm}}$\\\relax\raggedright #1}}}

\newcommand{\cL}{\mathcal{L}}
\newcommand{\cLe}{\mathcal{L}_{\ep}}

\newcommand{\bL}{\mathbb{L}}
\newcommand{\bLe}{\mathbb{L}_{\ep}}

\newcommand{\ue}{u_\ep}
\newcommand{\Hto}{\stackrel{H}{\to}}

\renewcommand{\th}{\theta}

\newcommand{\bbL}{\mathbb{L}}

\newcommand{\rd}{\mathrm{d}}
\newcommand{\dm}{\, \rd\mu}

\newcommand{\cE}{\mathcal{E}}
\newcommand{\cEe}{\mathcal{E}_{\ep}}

\newcommand{\di}{\mbox{div}}

\newcommand{\pa}{\partial}
\newcommand{\refmani}{N_0}
\newcommand{\LimMeasure}{\hat{\mu}_0}
\newcommand{\LimMetric}{\hat{g}_0}

\newcommand{\Div}{{\mathop{\operatorname{div}}}}

\begin{document}
\title{$H$-compactness of elliptic operators on weighted Riemannian Manifolds}
\author{Helmer Hoppe}
\address[Helmer Hoppe]{Technische Universit\"at Dresden, Faculty of Mathematics, 01069 Dresden, Germany}
\email{helmer.hoppe@tu-dresden.de}
\thanks{}

\author{Jun Masamune}
\address[Jun Masamune]{Hokkaido University, Department of Mathematics, 
Kita 10, Nishi 8, Kita-Ku, Sapporo, Hokkaido, 060-0810, Japan}
\email{jmasamune@math.sci.hokudai.ac.jp}
\thanks{}

\author{Stefan Neukamm}
\address[Stefan Neukamm]{Technische Universit\"at Dresden, Faculty of Mathematics, 01069 Dresden, Germany}
\email{stefan.neukamm@tu-dresden.de}

\begin{abstract}
In this paper we study the asymptotic behavior of second-order uniformly elliptic 
operators on weighted Riemannian manifolds. They naturally emerge when studying spectral properties of the Laplace-Beltrami operator on families of manifolds with rapidly oscillating metrics. We appeal to the notion of \mbox{$H$-convergence} introduced by Murat and Tartar. 
In our main result we establish an \mbox{$H$-compactness} result that applies to elliptic operators  with measurable, uniformly elliptic coefficients on weighted Riemannian manifolds. 
We further discuss  the special case of ``locally periodic'' coefficients and study the asymptotic spectral behavior of compact submanifolds of $\R^n$ with rapidly oscillating geometry.
\end{abstract}

\date{\today}
\maketitle
\tableofcontents

\section{Introduction}

We study the asymptotic behavior of elliptic operators on families of weighted Riemannian manifolds
that might feature fast oscillations.
In this introduction we survey the results and the structure of this paper without going into detail. 
The precise definitions and statements can then be found in Section~\ref{Sec:Statements}.

\smallskip

Convergence of metric measure spaces, in particular, Riemannian manifolds, has attracted an enormous amount of attention. 
Especially, substantial effort has been devoted to establishing geometric criteria for the
convergence of spectral structures, e.g., see \cite{Fukaya1987,KasueKumura1994,KasueKumura1996,KuwaeUemura1997,KuwaeShioya2003,
LanciaMoscoVivaldi2008,KuwaeShioya2008,MoscoVivaldi2009,
Masamune2011,ChenCroydonKumagai2015,
Giglietal2015,Kasue2017}.

\smallskip

Our point of view is different. We establish a compactness result that shows that \emph{any} family of (uniformly elliptic) PDEs of the form $-\Div_{g_\ep,\mu_{\varepsilon}}(\mathbb L_\ep\nabla_{g_\ep})u=f$ defined  on a uniformly bi-Lipschitz diffeomorphic family of weighted Riemannian manifolds $(M_\ep, g_\ep,\mu_\ep)$ admits an 
\emph{$H$-convergent} subsequence. The latter notion has been introduced in the context of homogenization of elliptic PDEs on $\mathbb{R}^n$ (in divergence form and of second-order), see \cite{MT}. In particular, in our setting it yields the existence of a limiting manifold and a limiting elliptic PDE such that solutions to the elliptic PDE on $M_\ep$ converge as $\ep\downarrow 0$ to the solution of the limiting PDE.
Our approach in particular allows us to treat Riemannian manifolds which oscillate rapidly on a small length scale $0<\ep\ll 1$.
\smallskip

This should be compared with the seminal work by Kuwae and Shioya \cite{KuwaeShioya2003}, 
where spectral convergence is established for families of manifolds which are locally bi-Lipschitz diffeomorphic to a reference manifold with a 
bi-Lipschitz constant converging to $1$. In situations where the manifold features rapid oscillations, the family of diffeomorphisms between the manifolds is only uniformly bi-Lipschitz but not locally close to an isometry---and thus the approach in \cite{KuwaeShioya2003} is not applicable. In contrast, as we shall show, it is still possible to establish $H$-convergence, which in the symmetric case (e.g.~when considering the Laplace-Beltrami operator on $M_\ep$) implies Mosco-convergence of the associated energy forms, and the convergence of the associated spectrum. Moreover, our approach also applies to non-symmetric PDEs.
\smallskip

For general uniformly bi-Lipschitz diffeomorphic families of manifolds the limiting manifold and  PDE depends on the extracted subsequence. However, under geometric conditions for $(M_\ep)$, we can uniquely identify the limit by appealing to suitable homogenization formulas (see Section~\ref{S:ident}). In the flat case, a natural geometric condition is periodicity of the coefficient field. In the case of PDEs on Riemannian manifolds with a symmetry structure, or for  general manifolds that feature periodicity in local coordinates, we obtain similar identification results and  homogenization formulas.
\smallskip

The latter might be of interest for applications to  diffusion models in biomechanics, which is another motivation of our work. In this context, diffusion and reaction-diffusion processes in biological membranes and through interfaces are studied, e.g.~see \cite{Aizenbud1982, Jacobesn1983, Sbalzarini, neuss-radu}.
One observation made is that ``diffusion in biological membranes can
appear \textit{anisotropic} even though it is molecularly isotropic in all observed instances'', see \cite{Sbalzarini}. We present examples (see below) where anisotropic diffusion on surfaces emerges on large scales from isotropic diffusion on surfaces with rapidly oscillating geometry.
\medskip

\paragraph{\bf Examples}
Before stating our results in a general form, we illustrate our findings on the level of examples. 
In the following we present four examples. Each example considers a family of $2$-dimensional submanifolds $(M_\ep)$ in $\R^3$ given by an explicit formula and depending on a small parameter $\ep>0$. In the limit $\varepsilon\downarrow0$, $M_{\varepsilon}$ Hausdorff-converges (as a subset of $\mathbb{R}^3$) to a reference submanifold $M_0\subset\mathbb{R}^3$; however the spectrum of the associated Laplace-Beltrami operator on $M_{\varepsilon}$ does not converge to the spectrum of the one on $M_0$. Nevertheless, we can associate to $(M_{\ep})$ a $2$-dimensional submanifold $N_0\subset\R^3$ that captures the asymptotic spectral behavior of $(M_{\ep})$ in the limit $\varepsilon\downarrow 0$: The spectrum of the Laplace-Beltrami operator on $M_{\ep}$ converges to the spectrum of the Laplace-Beltrami operator on $N_0$ in the sense of Lemma~\ref{171017.cor} below. Proofs and further details are presented in Section~\ref{Sec:Examples}.
\medskip

\paragraph{\bf (a) A graphical surface with star-shaped corrugations.} For $R>0$ and a smooth, $2\pi$-periodic function $f\colon[0,\infty)\to\mathbb{R}$ we introduce the family $(M_{\varepsilon})$ of $2$-dimensional submanifolds of $\mathbb{R}^3$:
\begin{equation}\label{ex:a}
	M_{\varepsilon}:=\Big\{
	\begin{pmatrix}
		r\sin\theta\\
		r\cos\theta\\
		\varepsilon f(\tfrac{\theta}{\varepsilon})
	\end{pmatrix}
	;r\in(0,R),\theta\in[0,2\pi)
	\Big\}.
\end{equation} In Figure~\ref{Example1:Fig1} we present $M_{\varepsilon}$ for some values of $\varepsilon$ in the case $f=\sin^2$.
As an application of our results we show that the spectrum of the Laplace-Beltrami operator on $M_\varepsilon$ converges to the spectrum of the Laplace-Beltrami operator on the submanifold
\begin{equation}\label{Example1:Eq2}
	N_0:=\Big\{
	\begin{pmatrix}
		\rho_0(r)\sin\theta\\
		\rho_0(r)\cos\theta\\
		\int_0^r\sqrt{1-\rho_0'(t)^2}\,\mathrm{d}t
	\end{pmatrix}
	;r\in(0,R),\theta\in[0,2\pi)
	\Big\},
\end{equation}
where $\rho_0(r)=\tfrac{1}{2\pi}\int_0^{2\pi}\sqrt{f'(y)^2+r^2}\,\mathrm{d}y$, see Figure~\ref{Example1:Fig1}.
\begin{figure}[ht]
  \centering
  \begin{minipage}{0.2\textwidth}
    \centering\includegraphics[scale=0.7]{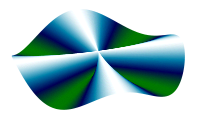}
  \end{minipage}\begin{minipage}{0.2\textwidth}
    \centering\includegraphics[scale=0.7]{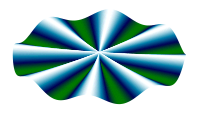}
  \end{minipage}\begin{minipage}{0.2\textwidth}
    \centering\includegraphics[scale=0.7]{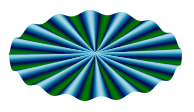}
  \end{minipage}\begin{minipage}{0.2\textwidth}
    \centering$\xrightarrow{\varepsilon\downarrow0}$
  \end{minipage}\begin{minipage}{0.2\textwidth}
    \centering\includegraphics[scale=0.7]{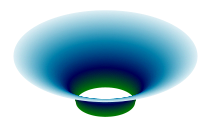}
  \end{minipage}\\\begin{minipage}{0.2\textwidth}
    \centering$\varepsilon=\tfrac{1}{2}$
  \end{minipage}\begin{minipage}{0.2\textwidth}
    \centering$\varepsilon=\tfrac{1}{4}$
  \end{minipage}\begin{minipage}{0.2\textwidth}
    \centering$\varepsilon=\tfrac{1}{8}$
  \end{minipage}\begin{minipage}{0.2\textwidth}
		\centering\mbox{}
  \end{minipage}\begin{minipage}{0.2\textwidth}
    \centering\mbox{}
  \end{minipage}
  \caption{A family of graphical surfaces with star-shaped corrugations. The three pictures on the left show $M_\ep$ defined by \eqref{ex:a} with $f=\sin^2$ and decreasing values of $\ep$. The picture on the right shows the limiting surface $N_0$ defined via \eqref{Example1:Eq2}. As $\ep\to 0$ the spectrum of the Laplace-Beltrami operator on $M_{\ep}$ converges to the spectrum of the Laplace-Beltrami operator on $N_0$. The color indicates the height component.}\label{Example1:Fig1}
\end{figure}
\medskip

\paragraph{\bf (b) A carambola-shaped sphere in $\R^3$.} We can transfer the example above from a graph over $\mathbb{R}^2$ to a sphere with oscillatory perturbation of its radius as depicted in Figure~\ref{Example1:Fig2}. More precisely, for a smooth $2\pi$-periodic function $f\colon[0,\infty)\to\mathbb{R}$ we consider the family $(M_{\varepsilon})$ of $2$-dimensional submanifolds of $\mathbb{R}^3$:
\begin{equation}\label{Example1:Eq3}
  M_\varepsilon
	:=\Big\{(1+\varepsilon f(\tfrac{\theta}{\varepsilon}))
		\begin{pmatrix}
			\sin\varphi\sin\theta\\
			\sin\varphi\cos\theta\\
			\cos\varphi
		\end{pmatrix};
		\varphi\in(0,\pi),\theta\in[0,2\pi)
	\Big\}.
\end{equation}
In that case a limiting submanifold is given by
\begin{equation}\label{Example1:Eq4}
  \refmani
	:=\Big\{\begin{pmatrix}
			\rho_0(\varphi)\sin\theta\\
			\rho_0(\varphi)\cos\theta\\
			\int_0^\varphi\sqrt{1-\rho_0'(t)^2}\,\mathrm{d}t
		\end{pmatrix};
		\varphi\in(0,\pi),\theta\in[0,2\pi)
	\Big\},
\end{equation}
where $\rho_0(\varphi)=\tfrac{1}{2\pi}\int_0^{2\pi}\sqrt{f'(y)^2+\sin^2\varphi}\,\mathrm{d}y$. See Figure~\ref{Example1:Fig2} for a visualization in the case $f=\sin^2$.
\begin{figure}[ht]
  \centering
  \begin{minipage}{0.2\textwidth}
    \centering\includegraphics[scale=0.75]{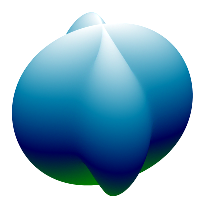}
  \end{minipage}\begin{minipage}{0.2\textwidth}
    \centering\includegraphics[scale=0.75]{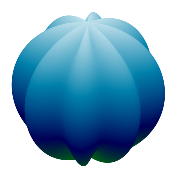}
  \end{minipage}\begin{minipage}{0.2\textwidth}
    \centering\includegraphics[scale=0.75]{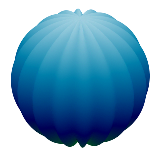}
  \end{minipage}\begin{minipage}{0.2\textwidth}
    \centering$\xrightarrow{\varepsilon\downarrow0}$
  \end{minipage}\begin{minipage}{0.2\textwidth}
    \centering\includegraphics[scale=0.75]{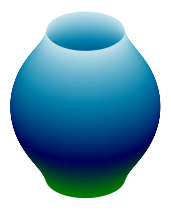}
  \end{minipage}\\\begin{minipage}{0.2\textwidth}
    \centering$\varepsilon=\tfrac{1}{2}$
  \end{minipage}\begin{minipage}{0.2\textwidth}
    \centering$\varepsilon=\tfrac{1}{4}$
  \end{minipage}\begin{minipage}{0.2\textwidth}
    \centering$\varepsilon=\tfrac{1}{8}$
  \end{minipage}\begin{minipage}{0.2\textwidth}
		\centering\mbox{}
  \end{minipage}\begin{minipage}{0.2\textwidth}
    \centering\mbox{}
  \end{minipage}
\caption{A family of spheres with radial perturbations oscillating with the longitude. The three pictures on the left show $M_\ep$ defined by \eqref{Example1:Eq3} with $f=\sin^2$ and decreasing values of $\ep$. The picture on the right shows the limiting surface $N_0$ defined via \eqref{Example1:Eq4}.}\label{Example1:Fig2}
\end{figure}
\medskip

\paragraph{\bf (c) A corrugated, rotationally symmetric submanifold in $\R^3$.} In contrast to the previous example we assume a sphere with radial perturbations with the latitude, i.e.\ for a smooth $\pi$-periodic function $f\colon[0,\infty)\to\mathbb{R}$ we consider the family $(M_{\varepsilon})$ of $2$-dimensional submanifolds of $\mathbb{R}^3$:
\begin{equation}\label{Example2:Eq3}
  M_\varepsilon
	:=\Big\{(1+\varepsilon f(\tfrac{\varphi}{\varepsilon})
		\begin{pmatrix}
			\sin\varphi\sin\theta\\
			\sin\varphi\cos\theta\\
			\cos\varphi
		\end{pmatrix};
		\varphi\in(0,\pi),\theta\in[0,2\pi)
	\Big\}.
\end{equation}
In that case a limiting submanifold is given by
\begin{equation}\label{Example2:Eq4}
  \refmani
	:=\Big\{\begin{pmatrix}
			\sin\varphi\sin\theta\\
			\sin\varphi\cos\theta\\
			\int_0^\varphi\sqrt{\tfrac{\rho_0(t)^2}{\sin^2t}-\cos^2t}\,\mathrm{d}t
		\end{pmatrix};
		\varphi\in(0,\pi),\theta\in[0,2\pi)
	\Big\},
\end{equation}
where $\rho_0(\varphi)=\tfrac{\sin\varphi}{\pi}\int_0^{\pi}\sqrt{f'(y)^2+1}\,\mathrm{d}y$. See Figure~\ref{Example2:Fig2} for the case $f=\sin^2$.
\begin{figure}[ht]
  \centering
  \begin{minipage}{0.2\textwidth}
    \centering\includegraphics[scale=0.75]{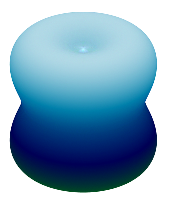}
  \end{minipage}\begin{minipage}{0.2\textwidth}
    \centering\includegraphics[scale=0.75]{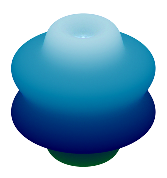}
  \end{minipage}\begin{minipage}{0.2\textwidth}
    \centering\includegraphics[scale=0.75]{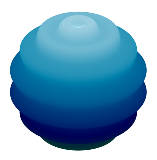}
  \end{minipage}\begin{minipage}{0.2\textwidth}
    \centering$\xrightarrow{\varepsilon\downarrow0}$
  \end{minipage}\begin{minipage}{0.2\textwidth}
    \centering\includegraphics[scale=0.75]{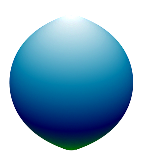}
  \end{minipage}\\\begin{minipage}{0.2\textwidth}
    \centering$\varepsilon=\tfrac{1}{2}$
  \end{minipage}\begin{minipage}{0.2\textwidth}
    \centering$\varepsilon=\tfrac{1}{4}$
  \end{minipage}\begin{minipage}{0.2\textwidth}
    \centering$\varepsilon=\tfrac{1}{8}$
  \end{minipage}\begin{minipage}{0.2\textwidth}
		\centering\mbox{}
  \end{minipage}\begin{minipage}{0.2\textwidth}
    \centering\mbox{}
  \end{minipage}
\caption{A family of spheres with radial perturbations oscillating with the latitude. The three pictures on the left show $M_\ep$ defined by \eqref{Example2:Eq3} with $f=\sin^2$ and decreasing values of $\ep$. The picture on the right shows the limiting surface $N_0$ defined via \eqref{Example2:Eq4}.}\label{Example2:Fig2}
\end{figure}
\medskip

\paragraph{\bf (d) A locally corrugated graphical surface.} Consider a relatively-compact open set $Y\subset\mathbb{R}^2$ and a set $Z\in Y$ of isolated points. For every point $z\in Z$ we use a smooth function $\psi_z\colon[0,\infty)\to[0,1]$ to define a rotationally symmetric cut-off function $\psi_z(|\cdot-z|)$ such that
\begin{equation*}
	\begin{cases}
		\text{$\psi_z(0)=1$,}&\\
		\text{$\supp\psi_z(|\cdot-z|)\cap\supp\psi_{z'}(|\cdot-z'|)=\emptyset$ for all $z'\in Z\setminus\{z\}$.}
	\end{cases}
\end{equation*}
Now we consider a smooth $T$-periodic function $f\colon[0,\infty)\to\mathbb{R}$ and the set $M_{\varepsilon}$ which is the graph of the function
\begin{equation}\label{Example2:Eq5}
  Y\setminus Z\ni x\mapsto\sum_{z\in Z}\varepsilon f\big(\tfrac{|x-z|}{\varepsilon}\big)\psi_z(|x-z|)\in\mathbb{R}^3,
\end{equation}
which we regard as a two-dimensional submanifold of $\R^3$. In that case a limiting submanifold is given by
\begin{equation}\label{Example2:Eq6}
  Y\setminus Z\ni x\mapsto\sum_{z\in Z}\int_0^{|x-z|}\sqrt{\tfrac{\rho_{0,z}(t)^2}{t^2}-1}\,\mathrm{d}t\in\mathbb{R}^3,
\end{equation}
where $\rho_{0,z}(r)=\tfrac{r}{T}\int_0^{T}\sqrt{f'(y)^2\psi_z(r)^2+1}\,\mathrm{d}y$. See Figure~\ref{Example2:Fig3} for the case $f=\sin^2$.
\begin{figure}[ht]
  \centering
  \begin{minipage}{0.2\textwidth}
    \centering\includegraphics[scale=0.65]{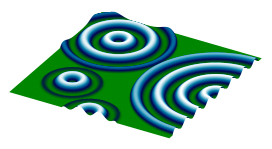}
  \end{minipage}\begin{minipage}{0.2\textwidth}
    \centering\includegraphics[scale=0.65]{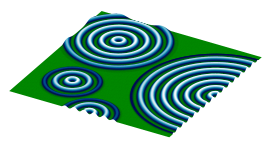}
  \end{minipage}\begin{minipage}{0.2\textwidth}
    \centering\includegraphics[scale=0.65]{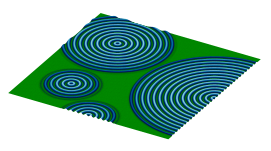}
  \end{minipage}\begin{minipage}{0.2\textwidth}
    \centering$\xrightarrow{\varepsilon\downarrow0}$
  \end{minipage}\begin{minipage}{0.2\textwidth}
    \centering\includegraphics[scale=0.65]{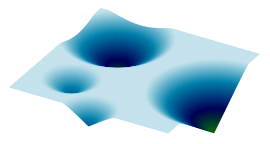}
  \end{minipage}\\\begin{minipage}{0.2\textwidth}
    \centering$\varepsilon=\tfrac{1}{2}$
  \end{minipage}\begin{minipage}{0.2\textwidth}
    \centering$\varepsilon=\tfrac{1}{4}$
  \end{minipage}\begin{minipage}{0.2\textwidth}
    \centering$\varepsilon=\tfrac{1}{8}$
  \end{minipage}\begin{minipage}{0.2\textwidth}
		\centering\mbox{}
  \end{minipage}\begin{minipage}{0.2\textwidth}
    \centering\mbox{}
  \end{minipage}
\caption{A family of locally corrugated graphical surfaces. The three pictures on the left show $M_\ep$ defined via \eqref{Example2:Eq5} with $f=\sin^2$ and decreasing values of $\ep$. The picture on the right shows the limiting surface $N_0$ defined via \eqref{Example2:Eq6}.}\label{Example2:Fig3}
\end{figure}
\bigskip

\paragraph{\bf General setting and the structure of the paper}
Throughout this paper we consider weighted Riemannian manifolds $M=(M,g,\mu)$ with metric $g$ and measure $\mu$. We \textit{always} assume that $M$ is \textit{$n$-dimensional (with $n\geq 2$), smooth, connected, without boundary}, and that $\mu$ has a smooth positive density against the Riemannian volume associated with $g$. We refer to the end of the introduction for a summary of standard notation that we use in this paper. The examples discussed above belong to the following general setting:

\begin{definition}[Uniformly bi-Lipschitz diffeomorphic families of manifolds]\label{D:biLip}
  
 A family of weighted Riemannian manifolds $(M_\ep,g_\ep,\mu_\ep)$ indexed by $0<\ep<1$ is called \emph{uniformly bi-Lipschitz diffeomorphic}, 
 if there exits a weighted Riemannian manifold $( M_0, g_0, \mu_0)$ and a constant $C$ such that for all $\varepsilon$ there exist diffeomorphisms $h_{\varepsilon}\colon  M_0\to M_{\varepsilon}$ with
  \begin{equation}\label{unifbound}
    \tfrac{1}{C}|\xi|\leq|dh_{\varepsilon}(x)\xi|\le C|\xi|
    \qquad\text{for all $x\in M_0$ and $\xi\in T_xM_0$}.
  \end{equation}
  We call $( M_0, g_0,\mu_0)$ reference manifold.
\end{definition}
In the setting of \eqref{unifbound} the Laplace-Beltrami operator on $M_{\varepsilon}$ gives rise to a second-order elliptic operator $\Div(\mathbb{L}_{\varepsilon}\nabla)$ on $ M_0$ with a uniformly elliptic coefficient field $\mathbb{L}_{\varepsilon}$, i.e.
\begin{equation*}
   g_0(\xi,\mathbb{L}_{\varepsilon}\xi)\geq\tfrac{1}{C^{n+2}}|\xi|^2,
	\qquad
	 g_0(\xi,\mathbb{L}_{\varepsilon}^{-1}\xi)\geq {C^{n+2}}|\xi|^2
	\qquad\text{for every $\xi\in TM_0$},
\end{equation*}
see Section \ref{SpecConv} for further details. It is therefore natural to consider homogenization of elliptic operators on the reference manifold with oscillating coefficients and measure. This is done in Section~\ref{Sec:Statement of the main results}, where our results are presented. 

Our main result, cf.~Theorem~\ref{T1}, is a compactness result for $H$-convergence. In the symmetric case (e.g., for the Laplace-Beltrami operator) $H$-convergence implies \textit{Mosco-convergence} of the associated energy forms, cf.~Lemma~\ref{L:HtoMo}, and the \textit{convergence of the spectrum} of the associated second-order elliptic operators $-\Div(\mathbb L_\varepsilon \nabla)$, cf. Lemma~\ref{L:HtoSpec}. In Section~\ref{S:ident} we address the problem of identifying the limiting PDE and manifold. In particular, we provide a \textit{homogenization formula} for manifolds that feature \emph{periodicity} in local coordinates. In Section~\ref{SpecConv} we discuss the application to families of parametrized manifolds that are bi-Lipschitz diffeomorphic. In particular, for such families, we establish spectral convergence (along subsequences) in Lemma~\ref{171017.cor} and discuss the special case of families of submanifolds of $\R^d$, see Lemma~\ref{L:submani}. 
In Section~\ref{Sec:Examples} we discuss concrete examples as the ones presented above. All proofs of the results in this paper are presented in Section~\ref{SecProofs}. %
\medskip

\textbf{Notation.}
For the background of the analysis on manifolds, we refer the readers to \cite{Grigoryan2009,Jost2011}.
\begin{itemize}
\item Let $\Omega\subset M$ open. We write 
$\omega \Subset \Omega$ if $\omega$ is an open set such that the closure $\overline \omega$ is compact and $\overline \omega\subset \Omega$.
\item We use $h$ for a diffeomorphism between manifolds and denote its differential by $dh$.
We use $\mathbb L$ for a measurable $(1,1)$-tensor field on a manifold. We call $\mathbb L$ also a \textit{coefficient field} on the manifold.
\item We use the notation $(\cdot,\cdot)(x)=g(\cdot,\cdot)(x)$ and $|\xi|(x)=\sqrt{g(\xi,\xi)(x)}$ to 
denote the inner product and induced norm in $T_xM$ at $x\in M$. 
We tacitly simply write $(\xi,\eta)$ and $|\xi|$ instead of $(\xi,\eta)(x)$ and 
$|\xi|(x)$ if the meaning is clear from the context. 
\item For a (sufficiently regular) function $u$ and vector field $\xi$ on $\Omega$, the gradient of $u$ is denoted by $\nabla_g u$
  and the divergence of $\xi$ is denoted by $\Div_{g,\mu}\xi$, i.e., we have $g(\nabla_g u,\xi)=\xi u=du(\xi)$ and 
$-\int_\Omega g(\Div_{g,\mu} \xi,u)\,\mathrm d\mu=-\int_\Omega g(\xi,\nabla_g u)\,\mathrm d\mu$ 
provided either $u$ or $\xi$ are compactly supported.
In particular, we write $\triangle_{g,\mu}:=\Div_{g,\mu}\nabla_g$ to denote the (weighted) 
Laplace-Beltrami operator. 
If the meaning is clear from the context, we shall simply write 
$\nabla,\Div$, and $\Delta$. 
In some situations the Riemannian manifold will be parametrized by the parameter 
$\varepsilon$; 
in that case, we may us the notation $\nabla_{\varepsilon}, 
\Div_{\varepsilon}$ and $\triangle_{\varepsilon}$. If there is no danger of confusion, we may drop the index $\varepsilon$ in the notation.
\item For $\Omega\subset M$ open we denote by $L^2(\Omega,g,\mu)$ the Hilbert space of 
square integrable functions and denote by
\begin{equation*}
  \|u\|_{L^2(\Omega,g,\mu)}^2:=\int_\Omega |u|^2\,\mathrm d\mu
\end{equation*}
the associated norm. We denote by $L^2(T\Omega)$ the space of measurable sections $\xi$ of $T\Omega$ such that $|\xi|\in L^2(\Omega,g,\mu)$.
\item We denote by $C^\infty_c(\Omega)$ the space of smooth compactly supported functions, 
and by $H^1(\Omega,g,\mu)$ the usual Sobolev space on $(\Omega,g,\mu)$, 
i.e.~the space of functions $u\in L^2(\Omega,g,\mu)$ with distributional first derivatives in $L^2(\Omega,g,\mu)$. 
Equipped with the norm 
\begin{equation*}
  \|u\|_{H^1(\Omega,g,\mu)}^2:=\int_M |u|^2+|\nabla u|^2\,\mathrm d\mu
\end{equation*}
(and the usual inner product), $H^1(\Omega,g,\mu)$ is a Hilbert space.
\item We denote by $H_0^1(\Omega,g,\mu)$ the closure of $C^\infty_c(\Omega)$ in $H^1(\Omega,g,\mu)$. 
We denote by $H^{-1}(\Omega,g,\mu)$ the dual space to $H^1_0(\Omega,g,\mu)$ and use the notation
$\langle F,u\rangle_{(\Omega,g,\mu)}$ to denote the dual pairing of $F\in H^{-1}(\Omega,g,\mu)$ and $u\in H^1_0(M,g,\mu)$.
\end{itemize}
We tacitly simply write  $\Omega$ (instead of $(\Omega,g,\mu)$), 
$L^2(\Omega)$, $H^1(\Omega)$, $\|\cdot\|_{L^2(\Omega)}$, $\|\cdot\|_{H^1(\Omega)}$, 
$\langle\cdot,\cdot\rangle$,  if the meaning is clear from the context. 
\section{Statement of the main results} \label{Sec:Statement of the main results}\label{Sec:Statements}
\subsection{H-, Mosco- and spectral convergence}\label{Sec:H-, Mosco- and spectral convergence}
We are interested in second-order elliptic operators of the form
\begin{equation*}
  -\Div(\mathbb L \nabla)\colon H^1_0(\Omega)\to H^{-1}(\Omega),\qquad \Omega\subset M\text{ open},
\end{equation*}
where $-\Div=-\Div_{g,\mu}\colon L^2(T\Omega)\to H^{-1}(\Omega)$ is 
the adjoint of $\nabla=\nabla_g\colon H^1_0(\Omega)\to L^2(T\Omega)$, 
and $\mathbb L$ denotes a uniformly elliptic coefficient field defined on $\Omega$. 
More precisely, for $0<\lambda\leq \Lambda$ and $\Omega\subset M$ open, 
we denote by $\mathcal{M}(\Omega,\lambda,\Lambda)$ the set of all measurable 
coefficient fields $\mathbb L\colon \Omega\to \operatorname{Lin}(T\Omega)$ that 
are uniformly elliptic and bounded in the sense that for $\mu$-a.e.\ $x\in\Omega$ 
and all $\xi\in T_x\Omega$
\begin{align}
\label{ellipt1}  g(\xi,\mathbb L(x)\xi)&\geq\lambda|\xi|^2,\\
\label{ellipt2}  g(\xi,(\mathbb L(x))^{-1}\xi)&\geq\tfrac{1}{\Lambda}|\xi|^2.
\end{align}
Moreover, we denote by $m_0(\Omega)$ the infimum of all $m\in\R$ such that
\begin{equation*}
  \inf\Big\{\int_\Omega m |u|^2+g(\nabla u,\nabla u)\,\mathrm{d}\mu;u\in H^1_0(\Omega)\text{ with }\|u\|_{H^1_0(\Omega)}=1\Big\}>0.
\end{equation*}
(See Remark~\ref{R:m0} below for a discussion of $m_0(\Omega)$). 
Given a family $(\mathbb L_{\ep})_{\ep>0}\subset\mathcal M(\Omega,\lambda,\Lambda)$ and 
$f\in H^{-1}(\Omega)$ we study the asymptotic behavior as $\ep\downarrow 0$ 
of the solution $u_\ep\in H^1_0(\Omega)$ to the equation
\begin{equation}
  \label{eq1}
  m u_\ep-\Div(\mathbb L_\ep \nabla u_{\ep})=f\qquad\text{ in }H^{-1}(\Omega),
\end{equation}
where $m$ denotes a fixed scalar satisfying $m>\frac{m_0(\Omega)}{\lambda}$. 
\begin{remark}\label{R:m0}
  By the Lax-Milgram lemma, \eqref{eq1} admits a unique solution $u_{\ep}\in H^1_0(\Omega)$ satisfying
  \begin{equation}
    \label{0}
    \|u_{\ep}\|_{H^1(\Omega)}\le C(\Omega,\lambda,m)\|f\|_{H^{-1}(\Omega)}.
  \end{equation}
  We briefly comment on the constant $m_0(\Omega)$, which appears in the lower bound condition for $m$ in \eqref{eq1}. 
  If $\Omega\Subset M$ is relatively-compact and connected, then Poincar\'e's inequality (for functions with zero mean) holds:
  \begin{equation*}
    \forall u\in H^1(\Omega)\,:\qquad \int_\Omega \Big|u-\tfrac{1}{\mu(\Omega)}\int_\Omega u\,\dm\Big|^2\,\dm\leq C_\Omega \int_\Omega |\nabla u|^2\,\dm.
  \end{equation*}
  In this case we have  $m_0(\Omega)\leq 0$, and in \eqref{eq1} any $m>0$ is admissible.  
  Also note that, the condition $m_0(\Omega)<0$ is equivalent to the validity of Poincare's inequality (for functions with vanishing boundary conditions):
  \begin{equation}\label{poincare}
    \forall u\in H^1_0(\Omega)\,:\qquad \int_\Omega |u|^2\,\dm\leq C'_\Omega \int_\Omega |\nabla u|^2\,\dm,
  \end{equation}
  where $C'_\Omega>0$ denotes a generic constant (only depending on $n$). Moreover, if $m_0(\Omega)<0$, then  in \eqref{eq1} we may then consider the case $m=0$.
\end{remark}

\paragraph{\bf H-compactness}
Our first main result is a compactness result concerning the homogenization limit
$\ep\downarrow 0$. 
It relies on the notion of $H$-convergence which goes back to the seminal work by 
Murat and Tartar (\cite{MT}) where the notion is introduced in the flat case $M=\R^n$. 
It is a generalization of the notion of $G$-convergence by Spagnolo and De Giorgi. 
The definition of $H$-convergence can be phrased in our setting as follows:
\begin{definition}[$H$-convergence]
  Let $\Omega\subset M$ be open. We say a sequence $(\mathbb{L}_{\ep})\subset\mathcal{M}(\Omega,\lambda,\Lambda)$ \emph{$H$-converges} 
  in $(\Omega,g,\mu)$ to $\mathbb{L}_0\in\mathcal M(\Omega,\lambda,\Lambda)$ as $\ep\to 0$, 
  if for any relatively-compact open subset $\omega\Subset\Omega$ with $m_0(\omega)<0$, and any $f\in H^{-1}(\omega)$, 
  the unique solutions $u_\ep,u_0\in H^1_0(\omega)$ to
  \begin{equation*}
    \begin{aligned}
      -\Div(\mathbb L_\ep \nabla u_{\ep})=-\Div(\mathbb L_0\nabla u_0)=f\qquad\text{in }H^{-1}(\omega)
    \end{aligned}
  \end{equation*}
  satisfy
  \begin{equation*}
    \begin{cases}
      u_{\ep}\rightharpoonup u_0&\quad\text{weakly in }H^1(\omega),\\
      \mathbb{L}_{\ep}\nabla u_{\ep}\rightharpoonup\mathbb{L}_0\nabla u_0	&\quad\text{weakly in }L^2(T\omega).
    \end{cases}
  \end{equation*}
  In that case we write $\mathbb L_\ep\Hto\mathbb L_0$ in $(\Omega,\mu,g)$ as $\ep\to 0$. 
\end{definition}
Our main result is the following $H$-compactness statement:
\begin{theorem}\label{T1}
  Let $\lambda,\Lambda>0$ and let $(\mathbb{L}_{\ep})$ denote a sequence in $\mathcal{M}(M,\lambda,\Lambda)$. 
  Then there exist a subsequence (not relabeled) and $\mathbb L_0\in\mathcal M(M,\lambda,\Lambda)$ such that the following holds:
  \begin{enumerate}[(a)]
  \item\label{th2:a}$\mathbb{L}_{\ep}\Hto\mathbb{L}_0$ in $(M,g,\mu)$.
  \item\label{th2:b}For every $\Omega\subset M$ open, every $m>\frac{m_0(\Omega)}{\lambda}$, 
  and sequences $(f_{\ep})\subset L^2(\Omega)$ and $(F_\ep)\subset L^2(T\Omega)$ with
	\begin{equation*}
		\begin{cases}
			f_{\ep}\rightharpoonup f_0&\quad\text{weakly in }L^2(\Omega),\\
			F_{\ep}\to F_0&\quad\text{in }L^2(T\Omega),
		\end{cases}
	\end{equation*}
	the solutions $u_{\ep},u_0\in H^1_0(\Omega)$ to 
        \begin{equation}\label{eq.problems}
          \begin{aligned}
            &mu_{\ep}-\Div(\mathbb{L}_{\ep}\nabla u_{\ep})=f_{\ep}+\Div F_{\ep}&&\quad\text{in }H^{-1}(\Omega),\\
            &mu_0-\Div(\mathbb{L}_0\nabla u_0)=f_0+\Div F_0&&\quad\text{in }H^{-1}(\Omega),
          \end{aligned}
	\end{equation}
	satisfy
	\begin{equation*}
		\begin{cases}
			u_{\ep}\rightharpoonup u_0&\quad\text{weakly in }H^1_0(\Omega),\\
			\mathbb{L}_{\ep}\nabla u_{\ep}\rightharpoonup\mathbb{L}_0\nabla u_0&\quad\text{weakly in }L^2(T\Omega).
		\end{cases}
	\end{equation*}
	Additionally we have $u_{\ep}\to u_0$ strongly in $L^2(\Omega)$, if either $H^1_0(\Omega)$ is compactly embedded in $L^2(\Omega)$, or 
	$m\neq 0$ and $f_{\ep}\to f_0$ strongly in $L^2(\Omega)$.
      \end{enumerate}
    \end{theorem}
    For the proof see Section~\ref{S:pT1}. The theorem is an extension of a classical result in \cite{MT} 
    where (scalar) elliptic operators of the form $-\mbox{div} (A_\ep\nabla)$ on $\R^n$ are considered. 
    It  has been extended to a large class of elliptic equations on $\R^n$ including e.g.~linear elasticity 
    \cite{Francfort1986} and monotone operators for vector valued fields (\cite{Francfort2009}). See also \cite{Waurick} for a variant that applies to non-local operators.
    \medskip
    
    In the following we briefly comment on the proof of Theorem~\ref{T1}, which is based on Murat and Tartar's method of oscillating test-functions. 
    In contrast to the classical flat case $M=\R^n$, we require a localization argument, since the tangent spaces $T_xM$ 
    change when $x$ varies in $M$. We therefore first establish $H$-compactness restricted to sufficiently small balls $B$ (see Proposition~\ref{P1} below) and then argue by covering $M$ with countably many of such balls. 
\begin{proposition}[$H$-compactness on small balls]\label{P1}
Let $(\mathbb{L}_{\ep})\subset\mathcal{M}(M,\lambda,\Lambda)$ and let $B\Subset M$ denote an open ball with radius smaller than the injectivity radius at its center. Then there exits $\mathbb{L}_0\in\mathcal{M}(\tfrac{1}{2}B,\lambda,\Lambda)$ and a (not relabeled) subsequence of $(\mathbb{L}_{\ep})$ such that $\mathbb{L}_{\ep}\Hto\mathbb{L}_0$ in $\frac{1}{2}B$, which is the open ball with the same center point and half the radius of $B$.
\end{proposition}
To lift Proposition~\ref{P1} from small balls to the whole manifold we cover $M$ by a countable collection of sufficiently small balls and pass to a diagonal sequence that features $H$-convergence on each of these balls. In order to guarantee that the $H$-limits associated with these balls are identical on the intersections of the balls, we appeal to the following lemma, which in particular establishes the uniqueness and locality property of $H$-convergence:
\begin{lemma}[Uniqueness, locality, invariance w.r.t.~transposition]\label{L:local}
Let $\Omega\subset M$ be open and consider a sequences $(\mathbb{L}_{\ep})\subset\mathcal{M}(\Omega,\lambda,\Lambda)$ that $H$-converges to some $\mathbb L_0$ in $\Omega$.
\begin{enumerate}[(a)]
\item Let $(\widetilde{\mathbb{L}}_{\ep})\subset\mathcal{M}(\Omega,\lambda,\Lambda)$ denote another sequence that $H$-converges to some $\widetilde{\mathbb L}_0$ in $\Omega$. Suppose that for some open $\omega\Subset\Omega$ we have $\mathbb{L}_{\ep}=\widetilde{\mathbb{L}}_{\ep}$ in $\omega$ for all $\ep$. Then $\mathbb{L}_0=\widetilde{\mathbb{L}}_0$ $\mu$-a.e.~in $\omega$.
\item Consider the coefficient field  $\mathbb L_\ep^*$ defined by the identity
  \begin{equation*}
    g(\mathbb L_\ep^*\xi,\eta)=    g(\xi,\mathbb L_\ep\eta)\qquad\text{for all }\xi,\eta\in T\Omega,
  \end{equation*}
  i.e., the adjoint of $\mathbb L_\ep$. Then $(\mathbb L_\ep^*)$ $H$-converges in $\Omega$ to $\mathbb L_0^*$ (the adjoint of $\mathbb L_0$).
\end{enumerate}
\end{lemma}
Finally, to prove that $H$-convergence on the individual balls yields $H$-convergence on the entire manifold, and in order to treat the varying right-hand sides in part (b) of Theorem~\ref{T1}, we apply the following lemma.
\begin{lemma}\label{lemB}
Let $\Omega\subset M$ be open and $\mathbb{L}_{\ep}\Hto\mathbb{L}_0$ in $\Omega$. Let $\omega\Subset\Omega$ with $m_0(\omega)<0$. Then for every $f_{\ep},f_0\in L^2(\omega)$ and $G_{\ep},F_{\ep},G_0,F_0\in L^2(T\omega)$ with
\begin{equation*}
	\begin{cases}
		f_{\ep}\rightharpoonup f_0&\quad\text{weakly in }L^2(\omega),\\
		G_{\ep}\to G_0&\quad\text{in }L^2(T\omega),\\
		F_{\ep}\to F_0&\quad\text{in }L^2(T\omega),
	\end{cases}
\end{equation*}
the unique solutions $u_{\ep},u_0\in H^1_0(\omega)$ to 
\begin{equation*}
	\begin{aligned}
		&-\Div(\mathbb L_\ep \nabla u_{\ep})=f_{\ep}-\Div(\mathbb{L}_{\ep}G_{\ep})-\Div F_{\ep}&&\quad\text{in }H^{-1}(\omega),\\
		&-\Div(\mathbb L_0 \nabla u_0)=f_0-\Div(\mathbb{L}_0G_0)-\Div F_{\ep}&&\quad\text{in }H^{-1}(\omega)
	\end{aligned}
\end{equation*}
satisfy
\begin{equation*}
	\begin{cases}
		u_{\ep}\rightharpoonup u_0&\quad\text{weakly in }H^1_0(\omega),\\
		\mathbb{L}_{\ep}\nabla u_{\ep}\rightharpoonup\mathbb{L}_0\nabla u_0&\quad\text{weakly in }L^2(T\omega).
	\end{cases}
\end{equation*}
\end{lemma}

\bigskip

\paragraph{\bf Mosco-convergence and convergence of the spectrum}
If we restrict to the symmetric case, i.e.\ $\mathbb L_\ep$ satisfies
\begin{equation*}
  g(\mathbb L_{\ep}\xi,\eta)=  g(\xi,\mathbb L_{\ep}\eta)\qquad\text{for all }\xi,\eta\in TM,
\end{equation*}
the solutions to \eqref{eq.problems} can be characterized as the unique
minimizers in $H^1_0(\Omega)$ to the strictly convex and coercive functional
\begin{equation*}
  H^1_0(\Omega)\ni u\mapsto \mathcal E_{m,\ep}(u)-\int_M f_\ep u+g(F_\ep,\nabla u)\,\dm,
\end{equation*}
where
\begin{equation*}
  \mathcal E_{m,\ep}(u):=\tfrac{1}{2}\int_\Omega m|u|^2+g(\mathbb L_\ep \nabla u,\nabla u)\,\dm.
\end{equation*}
In this symmetric situation we can appeal to variational notions of convergence, 
in particular   $\Gamma$-convergence and Mosco-convergence. 
The latter is extensively used to study the convergence properties of the associated evolution 
(i.e.~the semigroup generated by $-\Div(\mathbb L_{\ep}\nabla)$), e.g.~see 
\cite{KuwaeShioya2003,KuwaeUemura1997,Kolesnikov2008,Masamune2011,Lobus2015}. See a work by Hino (\cite{Hino1998}) for a non-symmetric generalization of Mosco-convergence.
A simple argument (that we outline for the reader's convenience---together with the definition of Mosco-convergence---in the appendix) 
shows that $H$-convergence implies Mosco-convergence (resp.\ Resolvent convergence):
\begin{lemma}[$H$-convergence implies Mosco-convergence]\label{L:HtoMo}
  Let $\mathbb L_{\ep}\in\mathcal M(M,\lambda,\Lambda)$ be symmetric. 
  Suppose $\mathbb L_{\ep}\Hto\mathbb L_0$, then the functional $\cEe\colon L^2(M)\to\R\cup\{+\infty\}$, 
  \[\cEe (u) = 
  \begin{cases}
    \int_M (\bLe \nabla u,\nabla u)\dm &\quad u \in H^1_0(M), \\
    \infty &\quad\mbox{otherwise}
  \end{cases}
  \]
  Mosco-converges to $\cE_0\colon L^2(M)\to\R\cup\{+\infty\}$, 
  \[\cE_0 (u) = 
  \begin{cases}
    \int_M (\bL_0 \nabla u,\nabla u)\dm &\quad u \in H^1_0(M), \\
    \infty & \quad\mbox{otherwise}.
  \end{cases}
  \]
\end{lemma}
\begin{remark}
The notion of Mosco-convergence only directly yields strong convergence of 
$(u_\ep)$ in $L^2(M)$ (and weak convergence in $H^1(M)$). 
The notion of $H$-convergence is a bit stronger, since it also yields convergence
of the fluxes $\mathbb L_\ep \nabla u_\ep$. 
In contrast, Mosco-convergence
in conjunction with the Div-Curl Lemma, see Lemma~\ref{lemC} below, 
only yields convergence of the $L^2$-projection of $\mathbb L_{\ep}\nabla u_\ep$ onto the orthogonal complement of $\{\nabla\phi\,:\,\phi\in
H^1_0(M)\}\subset L^2(T\Omega)$.
\end{remark}
Another consequence of $H$-convergence is convergence of the spectrum. In the following we consider an open, relatively-compact subset $\Omega\Subset M$ and suppose that $m_0(\Omega)<0$, so that Poincar\'e's inequality \eqref{poincare} is available and the embedding $H^1_0(\Omega)\subset L^2(\Omega)$ is compact. Moreover, we consider a symmetric, uniformly elliptic coefficient field $\mathbb L_\ep\in \mathcal M(M,\lambda,\Lambda)$. Then the spectral theorem for compact, self-adjoint operators applied to the operator  $-\Div(\mathbb L_\ep\nabla):H^1_0(\Omega)\subset L^2(\Omega)\to L^2(\Omega)$ implies that $L^2(\Omega)$ decomposes into countably many, finite dimensional, orthogonal eigenspaces associated with strictly positive eigenvalues. The following statement shows that if $\mathbb L_\ep$ is $H$-convergent, then the eigenspaces and eigenvalues converge. The statement is a direct consequence of \cite[Lemma~11.3 and Theorem~11.5, see also Theorem~11.6]{Zhikov-book} combined with Theorem~\ref{T1}:
\begin{lemma}[$H$-convergence implies spectral convergence]\label{L:HtoSpec}
  Let $(\mathbb L_{\ep})$ be a sequence of symmetric coefficient fields in $\mathcal M(M,\lambda,\Lambda)$ and suppose that $\mathbb L_{\ep}\Hto\mathbb L_0$. 
  Consider an open, relatively-compact set $\Omega\subset M$ with $m_0(\Omega)<0$. For $\ep\geq 0$ we consider the unbounded operator
  \begin{equation*}
    -\Div(\mathbb{L}_{\varepsilon}\nabla):H^1_0(\Omega)\subset L^2(\Omega)\to L^2(\Omega),
  \end{equation*}
  and let
  \begin{equation*}
    0<\lambda_{\varepsilon,1}\le\lambda_{\varepsilon,2}\le\lambda_{\varepsilon,3}\leq \cdots
  \end{equation*}
  denote the list of increasingly ordered eigenvalues, where eigenvalues are repeated according to their multiplicity. Let $u_{\ep,1},u_{\ep,2},u_{\ep,3},\ldots$ denote associated eigenfunctions. Then for all $k\in\mathbb N$, 
	\begin{equation*}
		\lambda_{\varepsilon,k}\to\lambda_{0,k},
	\end{equation*}
        and if $s\in\mathbb{N}$ denotes the multiplicity of $\lambda_{0,k}$, i.e.
	\begin{equation*}
		\lambda_{0,k-1}<\lambda_{0,k}=\cdots=\lambda_{0,k+s-1}<\lambda_{0,k+s}\qquad\text{(with the convention $\lambda_{0,0}=0$)},
	\end{equation*}
	then there exists a sequence $\bar{u}_{\varepsilon,k}$ of linear combinations of $u_{\varepsilon,k},\ldots,u_{\varepsilon,k+s-1}$ such that
	\begin{equation*}
		\bar{u}_{\varepsilon,k}\to u_{0,k}\qquad\text{strongly in $L^2(\Omega)$.}
	\end{equation*}
\end{lemma}

\subsection{Identification of the limit via local coordinate charts}\label{S:ident}
For a general sequence of coefficient fields $(\mathbb L_\ep)$ 
the $H$-limit $\mathbb L_0$ obtained by Theorem~\ref{T1} depends on the choice of the subsequence. 
In contrast, if the coefficient field features a special structure, 
then the $H$-limit is unique, the convergence holds for the entire sequence and one might even have a homogenization formula for $\mathbb L_0$. 
In the flat case $M=\R^n$ such results are classical. 
The simplest (non-trivial) example is periodic homogenization when $\mathbb L_\ep(x)=\mathbb L(\frac{x}{\ep})$ 
where $\mathbb L$ is periodic, i.e.~$\mathbb L(\cdot+k)=\mathbb L(\cdot)$ a.e.~in $\R^n$ for all $k\in\mathbb Z^n$; 
another example is stochastic homogenization, when $\mathbb L_{\ep}(x)=\mathbb L(\frac{x}{\ep})$ and $\mathbb L$ 
is sampled from a stationary and ergodic ensemble, 
see the seminal papers \cite{Papanicolaou1979} or \cite{neukamm-lecture-notes} 
for a self-contained introduction to periodic and stochastic homogenization. 
In the flat case these results rely on the fact that we can define an ergodic  group action on the manifold $M$. 
For general manifolds this is not possible. 
In this section we first make the simple observation that a coefficient field locally $H$-converges 
if and only if the coefficient field expressed in local coordinates $H$-converges, 
and secondly, obtain $H$-convergence and a homogenization formula for \textit{locally} periodic coefficient fields on general manifolds.
\medskip

For this purpose we fix $(\Om, \Psi; x^1,x^2,\dotsc,x^n)$ a local coordinate chart of $M$, a relatively-compact set $U\Subset\Psi(\Om) \subset \mathbb R^n$, and set $\omega:=\Psi^{-1}(U)\subset\Omega$. 
We will suppress $\Psi$ when the meaning is clear from the context. In particular, for the representation of a function $u$ on $\Omega$ in local coordinates we shall simply write $u$ instead of $u\circ\Psi^{-1}$.
We associate to $\mathbb L\in\mathcal M(\omega,\lambda,\Lambda)$ a density $\rho$ and a coefficient field $A\colon U\to\mathbb R^{n\times n}$ with components
\begin{equation}\label{def:A}
  A_{ij}:=\rho\,g(\mathbb{L}\nabla_gx^i,\nabla_gx^j)\quad\text{for all $i,j=1,\dotsc,n$},\qquad   \rho = \sigma \sqrt{\det g},
\end{equation}
where $\sigma$ is the density of $\mu$ against the Riemannian volume measure.
\begin{lemma}\label{190514.01}
  Let $\mathbb L\in\mathcal M(\omega,\lambda,\Lambda)$ and let $A:U\to\R^{n\times n}$ be defined by \eqref{def:A}. Then there exist $0<\lambda'\leq\Lambda'<\infty$ (only depending on $\Psi$, $U$, $\lambda$, and $\Lambda$) such that we have
  \begin{equation*}
    \forall \xi\in\mathbb R^n\,:\qquad A\xi\cdot \xi\geq \lambda'|\xi|^2\qquad\text{and}\qquad A^{-1}\xi\cdot \xi\geq\tfrac{1}{\Lambda'}|\xi|^2\qquad\text{a.e. in }U,
  \end{equation*}
  where $``\cdot "$ denotes the scalar product in $\mathbb R^n$.

\end{lemma}
Next we express the elliptic equation in local coordinates. 
For $f \in L^2({\omega})$ and $\xi \in L^2(T{\omega})$ let  $u \in H^1_0({\omega})$ be the  unique solution to
\[
  -\Div_{g,\mu}(\bbL \nabla_g u) = f -\Div_{g,\mu}\xi\qquad\text{in }H^{-1}({\omega}),
\]
that is
\[
  \int_{\omega} g(\bL \nabla_gu, \nabla_g\varphi)\dm = \int_{\omega} f\varphi\dm + \int_{\omega} g(\xi,\nabla_g\varphi)\dm\qquad \text{for all } \varphi \in  H^1_0({\omega}).
\]
Let $F \in L^2(TU)\cong L^2(U;\mathbb R^n)$ be the vector field
on $U$ with the components $F^i=dx^i(\xi)$. Then 
\begin{equation} \label{eq250109}
-\mbox{div} (A \nabla u) = \rho f -\mbox{div} (\rho F)\qquad\text{in }H^{-1}(U),
\end{equation}
that is, for any $\psi \in C^\infty_c(U)$
\[
\int_U A \nabla u \cdot \nabla \psi\, \mathrm d x = \int_U \rho f\psi\, \rd x + \int_U \rho F \cdot \nabla \psi\, \rd x,
\]
where $``\cdot"$ stands for the scalar product in $\mathbb R^n$. 

\medskip

With help of this transformation we can make the following observation:

\begin{lemma}\label{local-chart}
Let $\bL_\ep,\bL_0\in\mathcal M(\omega,\lambda,\Lambda)$ and denote by $A_\ep,A_0$ be defined by \eqref{def:A}. 
Then the following assertions are equivalent.
\begin{itemize}
\item[\rm(1)] $(\bL_\ep)$  $H$-converges to $\bL_0$ on $({\omega},g,\mu)$.
\item[\rm(2)] $(A_\ep)$  $H$-converges to $A_0$ on $U$ equipped with the standard Euclidean metric and measure.
\end{itemize}
\end{lemma}


On the level of $A_\ep$ (which is defined on the ``flat'' open subset  $U\subset\R^n$), we can naturally consider periodic homogenization. In the following we denote by $Y:=[0,1)^n$ the reference cell of periodicity and by $H^1_{\#}(Y)$ the Hilbert-space of $Y$-periodic functions $\phi\in H^1(Y)$ with zero average, i.e.\ $\int_Y\phi=0$. We denote by $\mathcal{M}_{\text{per}}(\lambda,\Lambda)$ the class of $Y$-periodic coefficient fields $A\colon\mathbb{R}^n\times\mathbb{R}^n\to\mathbb{R}^{n\times n}$ with ellipticity constants $0<\lambda\le\Lambda<\infty$, that is
\begin{align}
	&A(\cdot,y)\text{ is continuous for a.e.~$y\in\mathbb{R}^n$,}\\
	&A(x,\cdot)\text{ is measurable and $Y$-periodic for each $x\in\mathbb{R}^n$,}\\
	&\begin{aligned}
		&A(x,y)\xi\cdot\xi\ge\lambda|\xi|^2\text{ and }A(x,y)^{-1}\xi\cdot\xi\ge\tfrac{1}{\Lambda}|\xi|^2
		\text{for each $x\in\mathbb{R}^n$, a.e.~$y\in\mathbb{R}^n$}\\&\text{and all $\xi\in\mathbb{R}^n$}.
	\end{aligned}
\end{align}
It is a classical result (see e.g.~\cite[Theorem~2.2]{allaire1992}) that for $A\in\mathcal{M}_{\text{per}}(\lambda,\Lambda)$ the sequence $A_{\varepsilon}(x):=A(x,\tfrac{x}{\varepsilon})$ $H$-converges to a homogenized coefficient field $A_{\hom}$ which is characterized as follows:
\begin{equation}\label{perhom}
  A_{\hom}(x)e_j=\int_Y A(x,y)(\nabla_y\phi_j(x,y)+e_j)\,\mathrm dy,
\end{equation}
where $(e_j)$ is the standard basis in $\mathbb{R}^n$, and $\phi_j(x,\cdot)\in H^1_{\#}(Y)$ denotes the unique weak solution to 
\begin{equation}\label{perhom.b}
  \int_YA(x,y)(\nabla_y\phi_j(x,y)+e_j)\cdot \nabla_y\psi(y)\,\mathrm dy=0\qquad\text{for all }\psi\in H^1_{\#}(Y).
\end{equation}
For our purpose we require a small variant of this classical result which includes an additional shift in the definition of $A_{\varepsilon}$:
\begin{lemma}\label{ex:appendix}
  Let $A\in\mathcal{M}_{\mathrm{per}}(\lambda,\Lambda)$ and $r\in\R$. The sequence $A_{\eps}(x):=A(x,\frac{x+r}{\eps})$ $H$-converges on $\R^n$ to $A_{\hom}$ as defined in \eqref{perhom}.
\end{lemma}
Since we could not find a suitable reference in the literature we give the argument in the appendix. By appealing to periodic homogenization, we can make the following observation:
\begin{lemma}[Homogenization formula]\label{L:periodic} 
Let $\bL_\ep,\bL_0\in\mathcal M(M,\lambda,\Lambda)$ and suppose that $(\bLe)$ $H$-converges to $\bL_0$ on $M$. Fix a local coordinate chart $(\Omega,\Psi;x^1,x^2,\ldots,x^n)$ and let $A_\ep,A_0$ be the coefficient fields on $U\Subset\Psi(\Omega)$ associated with $\bL_\ep$ and $\bL_0$ defined by \eqref{def:A}. Suppose local periodicity in the sense that there exists a $Y:=[0,1)^n$-periodic coefficient field $L\colon\R^n\to\R^{n\times n}$ such that
\[
  g(\mathbb{L}_\ep(x)\tfrac{\partial}{\partial x^i},\tfrac{\partial}{\partial x^j})=L_{ij}(x,\tfrac{x}{\ep})\qquad\text{for a.e. }x\in\Omega.
\]
Then $\mathbb L_0$ on ${\omega}=\Psi^{-1}(U)\subset\Omega$ in local coordinates takes the form
\begin{equation*}
  (A_{\hom})_{ij}=\rho g(\mathbb{L}_0\nabla_gx^i,\nabla_gx^j)\qquad\text{a.e.~in }U,
\end{equation*}
where $A_{\hom}\colon U\to\R^{d\times d}$ is defined by \eqref{perhom} with $A(x,y):=\rho(x)L(y)$.
\end{lemma}

\subsection{Asymptotic behavior of the Laplace-Beltrami on parametrized manifolds}\label{SpecConv}
In this section we consider weighted Riemannian manifolds $(M_{\varepsilon},g_{\varepsilon},\mu_{\varepsilon})$ that are bi-Lipschitz diffeomorphic to a reference manifold $(M_0,{g_0},{\mu_0})$ in the sense of Definition~\ref{D:biLip}. In particular, below we shall consider the special case of submanifolds of $\R^d$ and study the asymptotic behavior of the associated Laplace-Beltrami operator. In our approach we pull the Laplace-Beltrami operator on $M_\ep$, $\Delta_{g_\ep,\mu_\ep}$, back to the reference manifold $M_0$ by appealing to the diffeomorphism $h_\ep$ from Definition~\ref{D:biLip}. In this way we obtain a family of elliptic operators on $M_0$ with coefficients $\mathbb L_\ep$. By appealing to our result on $H$-compactness, cf.~Theorem~\ref{T1}, we may extract a subsequence along which the elliptic operators $H$-converge to a limiting operator of the form $\Div(\mathbb L_0\nabla)$. In the symmetric case, we may combine this with our results with Lemma~\ref{L:HtoMo} and Lemma~\ref{L:HtoSpec} to deduce Mosco-convergence and convergence of the spectrum.
\smallskip

We start with a transformation rule. It invokes the following notation: If $(M,g,\mu)$ and $(M_0,g_0,\mu_0)$ are Riemannian manifolds, and $h:M_0\to M$ a diffeomorphism, then for every function $f$ on $M$ we denote by $\overline{f}:=f\circ h$ the pullback of $f$ along $h$. 
Moreover, we denote by $(d h^{-1})^* : TM_0 \to TM$ the adjoint of the differential $d h^{-1} : TM \to TM_0$  of $h^{-1}$ given by 
\begin{equation*}
  g((dh^{-1})^*\xi,\eta)(h(x))=g_0(\xi,dh^{-1}\eta)(x)\qquad\text{for all $\xi\in T_xM_0$, $\eta\in T_{h(x)}M$}.
\end{equation*}
\begin{lemma}[Transformation lemma]\label{L:trafo}
  Let $(M,g,\mu)$ and $(M_0,{g_0},{\mu_0})$ be weighted Riemannian manifolds and assume that there exists a bi-Lipschitz diffeomorphism $h:M_0\to M$ satisfying \eqref{unifbound}. Let $\sigma$ and $\sigma_0$ denote the densities of $\mu$ and $\mu_0$ w.r.t.~the Riemannian volume measures associated with $g$ and $g_0$, respectively. We use the notation $\overline{f}:=f\circ h$ and $\overline{u}:=u\circ h$ for the pullback along $h$. We define a density function $\rho$ and a coefficient field $\mathbb{L}$ on $M_0$ by the identities
  \begin{equation*}
    \rho:=\tfrac{\overline{\sigma}}{\sigma_0}\sqrt{\tfrac{\det \overline{g}}{\det g_0}}
    \qquad\text{and}\qquad
    g_0(\mathbb L\xi,\eta)=\rho\, \overline{g}((dh^{-1})^*\xi,(dh^{-1})^*\eta),
  \end{equation*}
	where $\overline{\sigma}:=\sigma\circ h$ and $\overline{g}:=g\circ h$ denote the pulled back quantities.
              Moreover we consider the metric $\LimMetric $ and the measure $\LimMeasure $ on $M_0$ given by
              \begin{equation*}
                d\LimMeasure :=\rho d\mu_0
                \qquad\text{and}\qquad
                \LimMetric (\mathbb{L}\xi,\eta):=\rho\, g_0(\xi,\eta),
	\end{equation*}
        Then the following are equivalent:
        \begin{enumerate}[(a)]
        \item $u\in H^1(M)$ is a solution to
          \begin{equation*}
            (m-\Delta_{g,\mu})u=f\qquad\text{in $H^{-1}(M,g,\mu)$};
    \end{equation*}
  \item $\overline{u}\in H^1(M_0)$ is a solution to
    \begin{equation*}
      (m\rho-\Div_{g_0,\mu_0}(\mathbb{L}\nabla_{g_0}))\overline{u}=\rho\overline{f}\qquad\text{in $H^{-1}(M_0,g_0,\mu_0)$};
    \end{equation*}
	\item $\overline{u}\in H^1(M_0)$ is a solution to
    \begin{equation*}
      (m-\Delta_{\LimMetric ,\LimMeasure })\overline{u}=\overline{f}\qquad\text{in $H^{-1}(M_0,\LimMetric ,\LimMeasure )$}.
    \end{equation*}
  \end{enumerate}
\end{lemma}
In the rest of this section, we consider the following setting:
\begin{assumption}[Family of uniformly bi-Lipschitz diffeomorphic manifolds]\label{ass}
  We denote by $(M_\ep,g_\ep,\mu_\ep)$ a family of weighted Riemannian manifolds that are bi-Lipschitz diffeomorphic to a reference manifold $(M_0,g_0,\mu_0)$ in the sense of Definition~\ref{D:biLip}. We assume that $H^1(M_0,g_0,\mu_0)$ is compactly embedded in $L^2(M_0,g_0,\mu_0)$. We denote by $\sigma_{\varepsilon}$ and $\sigma_0$ the densities of $\mu_{\varepsilon}$ and $\mu_0$ w.r.t.~the Riemannian volume measures associated with $g_{\varepsilon}$ and $g_0$, respectively. Moreover, we define $\rho_\ep$ and $\mathbb L_\ep$ by the identities
  \begin{equation}\label{ass:eq1}
    \rho_\ep:=\tfrac{\overline{\sigma}_\ep}{\sigma_0}\sqrt{\tfrac{\det \overline{g}_{\ep}}{\det g_0}}
    \qquad\text{and}\qquad
    g_0(\mathbb L_\ep\xi,\eta)=\rho_\ep\, \overline{g}_\ep((dh_{\varepsilon}^{-1})^*\xi,(dh_{\varepsilon}^{-1})^*\eta)
  \end{equation}
	with $\overline{\sigma}_{\varepsilon}:=\sigma_{\varepsilon}\circ h_{\varepsilon}$ and $\overline{g}_{\varepsilon}:=g_{\varepsilon}\circ h_{\varepsilon}$.
\end{assumption}
We introduce the following notion of strong $L^2$-convergence for functions defined on the variable spaces $L^2(M_\ep,g_\ep,\mu_\ep)$:
\begin{definition}
In the setting of Assumption~\ref{ass}. Let $f_\ep\in L^2(M_{\varepsilon},g_\ep,\mu_\ep)$ and $f_0\in L^2(M_0,\LimMetric ,\LimMeasure )$. We say $(f_\ep)$ strongly converges to $f_0$ in $L^2$, if
\begin{equation}\label{strong-conv}
	\begin{aligned}
		&\int_{M_{\varepsilon}} f_\ep(\psi\circ h_{\varepsilon}^{-1})\,\dm_\ep\to \int_{M_0} f_0\psi\,\mathrm{d}\LimMeasure \qquad\text{for all }\psi\in C^\infty_c(M_0),\qquad\text{and}\\
		&\int_{M_{\varepsilon}}|f_{\varepsilon}|^2\,\mathrm{d}\mu_{\varepsilon}\to\int_{M_0}|f_0|^2\,\mathrm{d}\LimMeasure .
	\end{aligned}
\end{equation}
\end{definition}

\begin{lemma}[$H$-Compactness of bi-Lipschitz diffeomorphic manifolds]\label{prop33108}
  Consider the setting of Assumption~\ref{ass}. Then there exists a subsequence for $\ep\to 0$ (not relabeled) such that the following holds:
  \begin{enumerate}[(a)]
  \item There exist a density $\rho_0$ and a uniformly elliptic coefficient field $\mathbb L_0$ on $M_0$ such that 
    $(\rho_{\ep})$ converges to $\rho_0$ weak-$*$ in $L^\infty(M_0)$, and $(\mathbb L_\ep)$ $H$-converges to $\mathbb L_0$ in $(M_0,g_0,\mu_0)$.
  \item Define a measure $\LimMeasure $ and a metric $\LimMetric $ on $M_0$ via the identities
    \begin{equation*}
      \mathrm d\LimMeasure :=\rho_0\mathrm d\mu_0
      \qquad\text{and}\qquad
      \LimMetric (\mathbb{L}_0\xi,\eta)=\rho_0\, g_0(\xi,\eta).
    \end{equation*}
    Let $m>m_0(M_0,g_0,\mu_0)$ and let $u_{\varepsilon}\in H^1(M_{\varepsilon})$ and $u_0\in H^1(M_0)$ denote the unique solutions to
    \begin{subequations}
      \begin{align}\label{C:bilip:1}
        (m-\Delta_{g_{\varepsilon},\mu_{\varepsilon}}) u_\ep&=f_\ep	&&\text{in $H^{-1}(M_{\varepsilon},g_{\varepsilon},\mu_{\varepsilon})$},\\\label{C:bilip:2}
        (m-\Delta_{\LimMetric ,\LimMeasure })u_0&=f_0 &&\text{in
          $H^{-1}(M_0,\LimMetric ,\LimMeasure )$},
      \end{align}
    \end{subequations}
    and suppose that
    \begin{equation*}
      f_{\varepsilon}\to f_0\qquad\text{strongly in $L^2$ in the sense of \eqref{strong-conv}}.
    \end{equation*}
    Then        
    \begin{equation*}
      u_{\varepsilon}\to u_0\qquad\text{strongly in $L^2$ in the sense of \eqref{strong-conv}.}
    \end{equation*}
  \end{enumerate}
\end{lemma}
The coefficient field $\mathbb{L}_{\varepsilon}$ in Lemma~\ref{prop33108} is symmetric and uniformly elliptic (with respect to $g_0$) by construction. Therefore, similarly to Lemma~\ref{L:HtoSpec} we may deduce convergence of the spectrum of the Laplace-Beltrami operators. To that end, we additionally suppose that $M_0$ is compact and $m_0(M_0)<0$. Thanks to \eqref{unifbound}, the weighted Riemannian manifolds $M_{\ep}$ satisfy the same properties, and thus the spectrum of $-\Delta_{g_\varepsilon,\mu_\ep}$ consists only of the real point spectrum with  strictly positive eigenvalues.
\begin{lemma}[Spectral convergence of bi-Lipschitz diffeomorphic manifolds]\label{171017.cor}
  Suppose that $M_0$ is compact and $m_0(M_0)<0$. Consider the setting of Assumption~\ref{ass}, and let $\overline g_0$, $\overline \mu_0$ be defined as Lemma~\ref{prop33108} (b). For $\ep\geq 0$ consider the operator
  \begin{equation*}
    \left\{\begin{aligned}
      &-\Delta_{g_{\varepsilon},\mu_{\varepsilon}}:H^1_0(M_\ep,g_\ep,\mu_\ep)\subset L^2(M_\ep,g_\ep,\mu_\ep)\to L^2(M_\ep,g_\ep,\mu_\ep)&&\text{for }\ep>0,\\
      &-\Delta_{\LimMetric,\LimMeasure}:H^1_0(M_0,\LimMetric,\LimMeasure)\subset L^2(M_0,\LimMetric,\LimMeasure)\to
        L^2(M_0,\LimMetric,\LimMeasure)&&\text{for }\ep=0,
      \end{aligned}\right.
  \end{equation*}
  and let
  \begin{equation*}
    0<\lambda_{\varepsilon,1}\le\lambda_{\varepsilon,2}\le\lambda_{\varepsilon,3}\leq \cdots,
  \end{equation*}

denote the increasingly ordered eigenvalues with eigenvalues being repeated according to their multiplicity. Let $u_{\ep,1},u_{\ep,2},u_{\ep,3},\ldots$ denote associated orthonormal eigenfunctions.  
Then for all $k\in\mathbb N$,
  \begin{equation*}
    \lambda_{\varepsilon,k}\to\lambda_{0,k},
  \end{equation*}
  and if $s\in\mathbb{N}$ is the multiplicity of $\lambda_{0,k}$, i.e.
  \begin{equation*}
    \lambda_{0,k-1}<\lambda_{0,k}=\cdots=\lambda_{0,k+s-1}<\lambda_{0,k+s} \qquad\text{(with the convention $\lambda_{0,0}=0$)},
  \end{equation*}
  then there exists a sequence $(\bar{u}_{\varepsilon,k})_{\varepsilon}$ of linear combinations of $u_{\varepsilon,k},\ldots,u_{\varepsilon,k+s-1}$ such that
  \begin{equation}\label{T:spectral:conv}
    \bar{u}_{\varepsilon,k}\to u_{0,k}\qquad\text{strongly in $L^2$ in the sense of \eqref{strong-conv}.}
  \end{equation}
\end{lemma}

We finally discuss the special case of submanifolds of $\R^d$. In the following lemma we collect (without proof) some consequences that directly follow from Lemma~\ref{L:trafo}, Lemma~\ref{prop33108}, and Lemma~\ref{171017.cor} applied to the special case. 
\begin{lemma}\label{L:submani}
  Consider the setting of Assumption~\ref{ass}, and assume that
  \begin{itemize}
  \item $M_{\varepsilon}$ are $n$-dimensional submanifolds of the Euclidean space
    $\mathbb{R}^d$ with $g_\ep$ and $\mu_\ep$ induced by the standard metric and measure of  $\mathbb{R}^d$;
  \item the reference manifold $M_0$ is a subset of the Euclidean space $\R^n$, i.e., $M_0\subset\R^n$, $g_0(\xi,\eta):=\xi\cdot\eta$, and $\mathrm d\mu_0=\mathrm dx$.
  \end{itemize}
  Then:
  \begin{enumerate}[(a)]
  \item The formulas in \eqref{ass:eq1} turn into
    \begin{equation*}
      \rho_{\varepsilon}=\sqrt{\det(dh_{\varepsilon}^{\sf T}dh_{\varepsilon})}
      \qquad\text{and}\qquad
      \mathbb L_{\varepsilon}=\rho_{\varepsilon}(dh_{\varepsilon}^{\sf T}dh_{\varepsilon})^{-1},
    \end{equation*}
    where $dh_{\varepsilon}$ denotes the Jacobian of $h_{\varepsilon}$.
  \item An application of Lemma~\ref{prop33108} yields the existence of a density $\rho_0$ and a coefficient field $\mathbb{L}_0\in\mathcal{M}(M_0,\frac{1}{C_0},C_0)$ (with $C_0>0$ only depending on $n$, $\lambda$, $\Lambda$ and the constant $C$ in \eqref{unifbound}) such that
    \begin{equation*}
      \begin{aligned}
        &\rho_{\varepsilon}=\sqrt{\det(dh_{\varepsilon}^{\sf T}dh_{\varepsilon})}&&\stackrel{*}{\rightharpoonup}\rho_0&&\text{weakly-$*$ in $L^{\infty}(M_0)$},\\
        &\mathbb{L}_{\varepsilon}=\rho_{\varepsilon}(dh_{\varepsilon}^{\sf T}dh_{\varepsilon})^{-1}&&\overset{H}{\to}\mathbb{L}_0&&\text{on $M_0$},
      \end{aligned}
    \end{equation*}
    for a subsequence (not relabeled), and the limiting Riemannian manifold $(M_0,\LimMetric ,\LimMeasure )$ is then given by
    \begin{equation*}
      \mathrm{d}\LimMeasure =\rho_0\mathrm{d}x
      \qquad\text{and}\qquad
      \LimMetric (\xi,\eta)=\rho_0\mathbb{L}_0^{-1}\xi\cdot\eta.
    \end{equation*}
  \item If additionally $M_0$ is open and bounded and has a Lipschitz boundary, then the conclusion of Lemma~\ref{171017.cor} on spectral convergence holds.
\end{enumerate}
\end{lemma}

\begin{remark}[Realizability of $(M_0,\LimMetric,\LimMeasure)$]\label{SubmanifoldRemark}
  If the limiting metric $\LimMetric$ is smooth, then it is realizable in $\R^m$ with $m$ large enough, i.e., there exists an isometry $h_0:(M_0,\LimMetric,\LimMeasure)\to\R^{m}$ such that $N_0:=h_0(M_0)$ is a $n$-dimensional submanifold of $\mathbb{R}^m$ (with induced metric and measure from $\R^m$). Such an embedding is characterized by the identity
  \begin{equation}\label{eq:embedding}
    dh_0^{\sf T}dh_0=\rho_0\mathbb{L}_0^{-1}.
  \end{equation}
  Indeed, this follows by the Nash embedding theorem provided the dimension of the ambient space $m$ is large enough. However, in the general case, we cannot necessarily give an explicit definition of the immersion $h_0$. In the examples that we discuss in Section~\ref{Sec:Examples} below, we study parametrized, $n=2$-dimensional submanifolds of $\R^3$ that converge to a limiting manifold that is realizable as a $2$-dimensional submanifold of $\R^3$ and given by an explicit formula.
\end{remark}

\section{Examples}\label{Sec:Examples}

In the following we consider two examples of laminate-like coefficient fields. We study each of them by appealing to homogenization in the flat case via local charts.  
Note that the coefficient fields in the following examples are intrinsic objects that could be considered without using charts, and so the respective $H$-limit, even though it is studied and expressed in local coordinates, is not bound to charts.

\subsection{Laminate-like coefficient fields on spherically symmetric manifolds}\label{s:Laminate-like-emxample}\label{Example1}

Let $0< R \le \infty$ and $s \in C^\infty([0,R))$ such that $s(r)>0$ if $r>0$, $s(0)=0$, and $s'(0)=1$.
We consider the $2$-dimensional spherically symmetric manifold 
$M=\{(x_1,x_2)=(r,\theta)\in [0,R) \times \bS^1\}$ equipped 
with the Riemannian metric
\[g=dr^2 + s^2(r) d\th^2\]
in the polar coordinates $(r,\th)$ (see e.g.\ \cite{Grigoryan2009}). 
For example, 
\begin{itemize}
\item $\R^2$ is a model with $R=\infty$ and $s(r)=r$;
\item $\bS^2$ without pole is a model with $R=\pi$ and $s(r)=\sin r$; 
\item $\mathbb H^2$ is a model with $R=\infty$ and $s(r)=\sinh r$.
\end{itemize}
For the sake of simplicity we normalize $\bS^1$ to have circumference 1. 
Consider $\bLe\in \mathcal M (M, \lambda, \Lambda)$ of the form
\[\bLe (r,\theta)=\bL_{\#}\big(r,\theta,\tfrac{\th}{\ep}\big)
\qquad\mbox{a.e.\  in $M$}
\]
and assume that $M\ni (r,\theta)\mapsto \bL_{\#}(r,\theta,y)$ is continuous for a.e.~$y\in\R$ 
and $y\mapsto \bL_{\#}(r,\theta,y)$ is measurable and $1$-periodic for all $(r,\theta)\in M$.
Denoting by $\{\phi(t)\}$ the one-parameter group
\[\phi(t):x \mapsto \exp_x \left(t {\frac{\pa}{\pa \th}} \right), \qquad x \in M \setminus \mbox{pole(s)},\ t \in \R,\]
the coefficient field $\mathbb L_{\ep}$ oscillates (on scale $\ep$) 
along $\phi$, while it is slowly varying in the radius direction. 
We therefore call $\mathbb L_{\ep}$ a \textit{laminate-like coefficient field} on $M$, see Figure~\ref{Example1:Fig0}.

\begin{figure}[ht]
	\begin{minipage}{0.3\textwidth}
		\centering\includegraphics[width=0.75\textwidth]{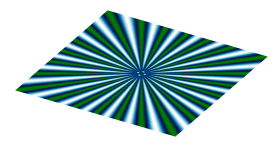}
	\end{minipage}\begin{minipage}{0.3\textwidth}
		\centering\includegraphics[width=0.75\textwidth]{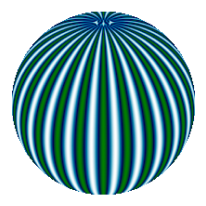}
	\end{minipage}\begin{minipage}{0.3\textwidth}
		\centering\includegraphics[width=0.75\textwidth]{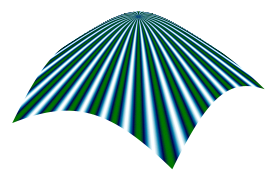}
	\end{minipage}
	\caption{Illustrations of the laminate-like structure of the coefficient field on $\mathbb{R}^2$, $\mathbb{S}^2$ and $\mathbb{H}^2$.}\label{Example1:Fig0}
\end{figure}

We make the following observations:
\begin{enumerate}[(a)]
\item\label{Example1:part1}By Theorem~\ref{T1} we have $\mathbb L_{\ep}\Hto\mathbb L_0$ for a subsequence and some 
coefficient field $\mathbb L_0$. As we shall see below, the limit $\mathbb L_0$ can be expressed by 
a ``homogenization formula'' that uniquely determines $\mathbb L_0$ in terms of $\mathbb L_{\#}$. 
Hence, $\mathbb L_0$ is independent of the chosen subsequence and we conclude that $\mathbb L_\ep\Hto\mathbb L_0$ for all sequences $\ep\downarrow 0$. 
\item\label{Example1:part2}Consider the special case
\begin{equation}\label{example:special}
  \bL_{\#}(r,\theta,y):=
  \begin{pmatrix}
    a_{\#}(y)&0\\0&b_{\#}(y)
  \end{pmatrix}
\end{equation}
with $a_\#,b_{\#}:\R\to (\lambda,\Lambda)$ measurable and $1$-periodic. 
Above, we tacitly expressed $\mathbb L_{\#}$ w.r.t.~polar coordinates, i.e.~$(\mathbb L_{\#})_{ij}:=(\tfrac{\partial}{\partial x^i},\mathbb L_{\#}\tfrac{\partial}{\partial x^j})$ where $x=(x^1,x^2)=(r,\theta)$. 
In this case we may represent $\mathbb L_0$ with help of the arithmetic and harmonic mean of $a_{\#}$ and $b_{\#}$ to express the diffusivity orthogonal to 
the flow $\phi$ and aligned to the flow $\phi$, respectively:
\begin{equation}
  \mathbb L_0=\label{example:special_hom}
  \begin{pmatrix}
    \int_0^1a_{\#}&0\\
    0&(\int_0^1b_{\#}^{-1})^{-1}
  \end{pmatrix}.
\end{equation}
\end{enumerate}
\medskip

In order to prove these claims it suffices to identify $\mathbb L_0$ locally. 
Consider an open, bounded set $\omega\Subset M$. 
We may assume without loss of generality that $\overline \omega$ does not intersect the curve 
$\{(r,\theta)\,:\,\theta=0\}$.
Denote the chart of polar coordinates by $\Psi$ and define $U\subset\mathbb R^2$ by $U:=\Psi(\omega)$. 
According to \eqref{def:A} we associate to $\mathbb L_{\ep}$ a coefficient field $A_\ep$ on $U$. 
It can be written in the form $A_{\ep}(r,\theta)=A_{\#}(r,\theta,\tfrac{\theta}{\ep})$ with
\begin{equation*}
  A_\#(r,\theta,y)=
  \begin{pmatrix}
    s(r)&0\\
    0&s^{-1}(r)
  \end{pmatrix}\mathbb L_{\#}(r,\theta,y),
  \end{equation*}
where we identified $\mathbb L_{\#}(r,\theta,y)$ with the corresponding coefficient matrix in polar coordinates.
Since $\mathbb L_\ep\Hto\mathbb L_0$ on $\omega$, 
we have $A_\ep\Hto A_0$ on $U$ by Lemma \ref{local-chart}. 
On the other hand, since $A_{\ep}$ is a coefficient field of the form $A_{\#}(r,\theta,\frac{\theta}{\ep})$ 
with $A_{\#}$ being continuous in the first two components and periodic in the third component, 
the periodic homogenization formula 
\eqref{perhom} applies 
and we deduce that $A_0$ only depends on $\mathbb L_{\#}$ and 
the metric $g$ (but not on the extracted subsequence). 
Hence, $\mathbb L_0$ is uniquely determined by $\mathbb L_{\#}$ and the metric, and thus $H$-convergence holds for the entire sequence. This proves \eqref{Example1:part1}

Next, we discuss the special case \eqref{example:special} for which we obtain
\begin{equation*}
  A_\#(r,\th,y)=
  \begin{pmatrix}
    s(r)a_\# \left( {\frac{\th}{\ep}} \right)       &  0 \\
    0      &  s^{-1}(r)b_\# \left( {\frac{\th}{\ep}} \right) 
  \end{pmatrix}
\end{equation*}
and 
\begin{equation*}
  A_0(r,\th)=
  \begin{pmatrix}
    s(r)\int_0^1a_\#      &  0 \\
    0      &  s^{-1}(r)(\int_0^1b_\#^{-1})^{-1} 
  \end{pmatrix}.
\end{equation*}
The above identities can be seen by evaluating \eqref{perhom}, which in the case of laminates can be done by hand. This proves \eqref{Example1:part2}.
\medskip

\paragraph{\textbf{Example 1: A graphical surface with star-shaped corrugations}}
In the spirit of Definition~\ref{D:biLip} we start with the reference manifold
\begin{equation*}
	M_0=\{(r,\theta);r\in(0,R),\theta\in[0,2\pi)\}
\end{equation*}
for some $R>0$. Note that $M_0$ does not include the origin. Now we define a family $M_{\varepsilon}=h_{\varepsilon}(M_0)$ of $2$-dimensional submanifolds of $\mathbb{R}^3$ (with standard metric and measure induced from $\mathbb{R}^3$) using uniform bi-Lipschitz immersions $h_{\varepsilon}\colon M_0\to\mathbb{R}^3$,
\begin{equation*}
	h_{\varepsilon}(r,\theta)=
	\begin{pmatrix}
		r\sin\theta\\
		r\cos\theta\\
		\varepsilon f(r,\tfrac{\theta}{\varepsilon})
	\end{pmatrix},
\end{equation*}
where $f\colon(0,\infty)\times[0,\infty)\to\mathbb{R}$ is smooth and $2\pi$-periodic in the second argument. In Figure~\ref{Example1:Fig1} in the Introduction we choose $f(r,y)=\sin^2(y)$ to present $M_{\varepsilon}$ for some values of $\varepsilon$.

We follow the path described in Lemma~\ref{L:submani} and calculate first
\begin{equation*}
	dh_{\varepsilon}^{\sf T}dh_{\varepsilon}=
	\begingroup
	\renewcommand*{\arraystretch}{1.5}
	\begin{pmatrix}
		1+\bigl(\varepsilon\partial_1f(r,\tfrac{\theta}{\varepsilon})\bigr)^2&\varepsilon\partial_1f(r,\tfrac{\theta}{\varepsilon})\partial_2f(r,\tfrac{\theta}{\varepsilon})\\
			\varepsilon\partial_1f(r,\tfrac{\theta}{\varepsilon})\partial_2f(r,\tfrac{\theta}{\varepsilon})&r^2+\bigl(\partial_2f(r,\tfrac{\theta}{\varepsilon})\bigr)^2
	\end{pmatrix},
	\endgroup
\end{equation*}
to get the density
\begin{equation*}
	\rho_{\varepsilon}
	=\sqrt{\det(dh_{\varepsilon}^{\sf T}dh_{\varepsilon})}
	=\sqrt{r^2+r^2\bigl(\varepsilon\partial_1f(r,\tfrac{\theta}{\varepsilon})\bigr)^2+\bigl(\partial_2f(r,\tfrac{\theta}{\varepsilon})\bigr)^2},
\end{equation*}
and the coefficient field
\begin{align*}
	\mathbb{L}_{\varepsilon}
	&=\rho_{\varepsilon}(dh_{\varepsilon}^{\sf T}dh_{\varepsilon})^{-1}\\
	&=1/\rho_{\varepsilon}
	\begingroup
	\renewcommand*{\arraystretch}{1.5}
	\begin{pmatrix}
		r^2+\bigl(\partial_2f(r,\tfrac{\theta}{\varepsilon})\bigr)^2&-\varepsilon\partial_1f(r,\tfrac{\theta}{\varepsilon})\partial_2f(r,\tfrac{\theta}{\varepsilon})\\
		-\varepsilon\partial_1f(r,\tfrac{\theta}{\varepsilon})\partial_2f(r,\tfrac{\theta}{\varepsilon})&1+\bigl(\varepsilon\partial_1f(r,\tfrac{\theta}{\varepsilon})\bigr)^2
	\end{pmatrix}.
	\endgroup
\end{align*}
It turns out that $\rho_{\varepsilon}\stackrel{*}{\rightharpoonup}\rho_0$ weakly-$*$ in $L^{\infty}(M_0)$ with
\begin{equation*}
	\rho_0(r)
	=\tfrac{1}{2\pi}\int_0^{2\pi}\sqrt{\bigl(\partial_2f(r,y)\bigr)^2+r^2}\,\mathrm{d}y,
\end{equation*}
and using \eqref{example:special_hom} we see $\mathbb{L}_{\varepsilon}\overset{H}{\to}\mathbb{L}_0$ with
\begin{align*}
	\mathbb{L}_0
	&=
	\begin{pmatrix}
		\tfrac{1}{2\pi}\int_0^{2\pi}\sqrt{\bigl(\partial_2f(r,y)\bigr)^2+r^2}\,\mathrm{d}y&0\\
		0&\Bigl(\tfrac{1}{2\pi}\int_0^{2\pi}\sqrt{\bigl(\partial_2f(r,y)\bigr)^2+r^2}\,\mathrm{d}y\Bigr)^{-1}
	\end{pmatrix}\\
	&=
	\begin{pmatrix}
		\rho_0(r)&0\\
		0&\tfrac{1}{\rho_0(r)}
	\end{pmatrix}.
\end{align*}
Thus the limiting metric on $M_0$ is given by
\begin{equation*}
	\LimMetric (\xi,\eta)
	=\rho_0\mathbb{L}_0^{-1}\xi\cdot\eta
	=\begin{pmatrix}
			1&0\\
			0&\rho_0^2
		\end{pmatrix}
		\xi\cdot\eta.
\end{equation*}
In this situation we finally can find a bi-Lipschitz immersion $h_0\colon M_0\to\mathbb{R}^3$ such that $dh_0^{\sf T}dh_0=\rho_0\mathbb{L}_0^{-1}$, namely
\begin{equation*}
	h_0(r,\theta)=
	\begin{pmatrix}
		\rho_0(r)\sin\theta\\
		\rho_0(r)\cos\theta\\
		\int_0^r\sqrt{1-\rho_0'(t)^2}\,\mathrm{d}t
	\end{pmatrix}.
\end{equation*}
That means, by Remark~\ref{SubmanifoldRemark}, the (rotationally symmetric) submanifold $\refmani:=h_0(M_0)$ of $\mathbb{R}^3$ (with the standard measure and metric induced from $\mathbb{R}^3$), which for the case $f(r,y)=\sin^2(y)$ is pictured in Figure~\ref{Example1:Fig1}, is the spectral limit of $(M_{\varepsilon})$. Note that the excluded origin in the reference manifold coincides now with a circle of radius $\lim_{r\downarrow0}\rho_0(r)$, which for $f(r,y)=\sin^2(y)$ is $\tfrac{\pi}{2}$.
\medskip

\paragraph{\textbf{Example 2: Sphere with radial perturbations oscillating with the longitude}}
Instead of a graph over $\mathbb{R}^2$ as in the example above we can treat a radially perturbed sphere in the same way. We take an analogous underlying reference manifold
\begin{equation*}
	M_0=\{(\varphi,\theta);\varphi\in(0,\pi),\theta\in[0,2\pi)\}
\end{equation*}
and define the family $M_{\varepsilon}:=h_{\varepsilon}(M_0)$ of $2$-dimensional submanifolds of $\mathbb{R}^3$ via bi-Lipschitz immersions $h_{\varepsilon}\colon M_0\to M_{\varepsilon}$,
\begin{equation*}
	h_{\varepsilon}(\varphi,\theta)
	=\bigl(1+\varepsilon f(\varphi,\tfrac{\theta}{\varepsilon})\bigr)
	\begin{pmatrix}
		\sin\varphi\sin\theta\\
		\sin\varphi\cos\theta\\
		\cos\varphi
	\end{pmatrix},
\end{equation*}
where $f\colon(0,\pi)\times[0,\infty)\to\mathbb{R}$ is differentiable and $2\pi$-periodic in the second argument. In Figure~\ref{Example1:Fig2} in the Introduction we choose $f(r,y)=\sin^2(y)$ to picture $M_{\varepsilon}$ for some values of $\varepsilon$. As in the previous example we obtain the following formulas for the limiting density
\begin{equation*}
	\rho_0(\varphi)
	=\tfrac{1}{2\pi}\int_0^{2\pi}\sqrt{(\partial_2f(\varphi,y))^2+\sin^2\varphi}\,\mathrm{d}y,
\end{equation*}
and the limiting metric
\begin{equation*}
	\LimMetric (\xi,\eta)
	=\tfrac{1}{\rho_0}\mathbb{L}_0=
	\begin{pmatrix}
		1&0\\
		0&\rho_0^2
	\end{pmatrix}
	\xi\cdot\eta.
\end{equation*}
Again we can find a bi-Lipschitz immersion $h_0\colon M_0\to\mathbb{R}^3$ such that $dh_0^{\sf T}dh_0=\rho_0\mathbb{L}_0^{-1}$, namely
\begin{equation*}
	h_0(\varphi,\theta)=
	\begin{pmatrix}
		\rho_0(\varphi)\sin\theta\\
		\rho_0(\varphi)\cos\theta\\
		\int_0^\varphi\sqrt{1-\rho_0'(t)^2}\,\mathrm{d}t
	\end{pmatrix}.
\end{equation*}
Thus the (rotationally symmetric) submanifold $\refmani:=h_0(M_0)$ of $\mathbb{R}^3$, which for the case $f(r,y)=\sin^2(y)$ is pictured in Figure~\ref{Example1:Fig2}, is the spectral limit of the sequence $(M_{\varepsilon})$.


\subsection{Concentric laminate-like coefficient fields on Voronoi tesselated manifolds}\label{Example2}

Let $(M,g,\mu)$ be a $n$-dimensional manifold and $Z\subset M$ a countable closed subset. For $z\in Z$ we denote by $M_z$ the associated Voronoi cell, that is
\begin{equation*}
	M_z:=\{x\in M;d(x,z)<d(x,Z\setminus\{z\})\},
\end{equation*}
where $d(\cdot,\cdot)$ is the geodesic distance on $M$. We assume the Voronoi tessellation to be fine enough to ensure that for $\mu$-a.e.~point $x_0\in M$ there are $z\in Z$ and $\varrho>0$ such that
\begin{equation}\label{ex3.cond}
	\text{for all }x\in B_{\varrho}(x_0)\subset M_z\text{ exists exactly one shortest path }\gamma_x\text{ from }x\text{ to }z.
\end{equation}

We consider a sequence $(\mathbb{L}_{\varepsilon})$ in $\mathcal{M}(M,\lambda,\Lambda)$ of rapidly oscillating coefficient fields of the form $\mathbb{L}_{\varepsilon}(x)=\mathbb{L}(\tfrac{d(x,Z)}{\varepsilon})$, where $\mathbb{L}(r)$ is $1$-periodic in $r\in\R$, see Figure~\ref{Example2:Fig0}.

\begin{figure}[ht]
	\begin{minipage}{0.3\textwidth}
		\centering\includegraphics[width=0.75\textwidth]{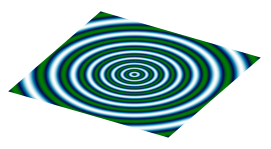}
	\end{minipage}\begin{minipage}{0.3\textwidth}
		\centering\includegraphics[width=0.75\textwidth]{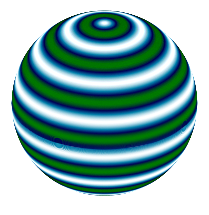}
	\end{minipage}\begin{minipage}{0.3\textwidth}
		\centering\includegraphics[width=0.75\textwidth]{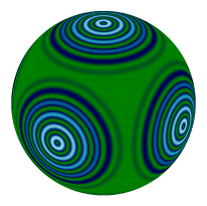}
	\end{minipage}
	\caption{Illustration of coefficient fields with laminate-like structure.}\label{Example2:Fig0}
\end{figure}

By Theorem~\ref{T1} $(\mathbb{L}_{\varepsilon})$ $H$-converges (up to a subsequence) to some $\mathbb{L}_0\in\mathcal{M}(M,\lambda,\Lambda)$. We are going to show that $\mathbb{L}_0$ coincides $\mu$-a.e.~on $M$ with some constant coefficient field which is uniquely determined by $\mathbb{L}$. In particular the whole sequence $(\mathbb{L}_{\varepsilon})$ $H$-converges to $\mathbb{L}_0$.

In order to prove this, it suffices to identify $\mathbb{L}_0$ locally, i.e.~for $\mu$-a.e.~$x_0\in M$. As a first step we construct curvilinear coordinates such that in these coordinates the coefficients locally turn into a laminate up to a small perturbation that vanishes at $x_0$. In particular we claim that local coordinates $(B_{\varrho}(x_0),\Psi;x^1,\dotsc,x^n)$ exist such that
\begin{subequations}
  \begin{align}
    \label{ex3.cond0}&\Psi(x_0)=0,\\
    \label{ex3.cond1}&x^1=d(\cdot,z)-d(x_0,z),\\
    \label{ex3.cond2}&g(\tfrac{\partial}{\partial x^1},\tfrac{\partial}{\partial x^j})=0\text{ for $j=2,\dotsc,n$},\\
    \label{ex3.cond3}&\lim\limits_{x\to x_0}\rho(x)g(\tfrac{\partial}{\partial x^i},\tfrac{\partial}{\partial x^j})(x)=\delta_{ij}.
  \end{align}
\end{subequations}
Indeed, note that by \eqref{ex3.cond1} geodesics through $z$ are mapped to straight lines parallel to the $x^1$-axis.

Therefore,  we fix $x_0\in M$, $z\in Z$ and $\varrho>0$ satisfying \eqref{ex3.cond}. As in \eqref{ex3.cond1} we set for $x\in B_{\varrho}(x_0)$
\begin{equation*}
	x^1(x):=d(x,z)-d(x_0,z).
\end{equation*}
Thanks to \eqref{ex3.cond} $x^1$ is differentiable and the level set $U_{x_0}:=\{x\in B_{\rho}(x_0);x^1(x)=0\}$ is a $n-1$-dimensional submanifold of $M_z$ including $x_0$ and for any point $x\in U_{x_0}$ the tangent space $T_xU_{x_0}$ is orthogonal to $d\gamma_x(0)$, which gives \eqref{ex3.cond2}. Assume $\varrho>0$ to be small enough such that we can choose local normal coordinates $x^2,\dotsc,x^n$ of $U_{x_0}$ with $x^j(x_0)=0$ ($j=2,\dotsc,n$). By the differentiability of geodesics we can extend these coordinate functions to curvilinear coordinates $x^1,\dotsc,x^n$ on $B_{\varrho}(x_0)$ (with a probably smaller $\varrho$) in the way that $x^2,\dotsc,x^n$ are constant on $\gamma_x$ for every $x\in B_{\varrho}(x_0)$. Then we have
\begin{equation}\label{eq:ex3.1}
  \lim\limits_{x\to x_0}g(\tfrac{\partial}{\partial x^i},\tfrac{\partial}{\partial x^j})(x)=
	\begin{cases}
		1,&i=j,\\
		0,&i\neq j,
	\end{cases}
\end{equation}
which yields \eqref{ex3.cond3}.

\begin{figure}[ht]
	\centering
	\begin{tikzpicture}
		\draw[fill] (6,0) circle (1pt) node at (6.2,0.2) {$z$};
		\draw[fill] (0,0) circle (1pt) node at (0.4,-0.1) {$x_0$};
		\draw (0,0) circle (2) node at (1.2,2.1) {$B_{\varrho}(x_0)$};
		\draw (-2.5,-2) to[out=50,in=205] (0,0) to[out=25,in=160] (6,0) to[out=-20,in=140] (6.5,-0.3) node at (3,0.4) {$\gamma_{x_0}$};
		\draw[fill] (-1,0.9) circle (1pt) node at (-1,1.1) {$x$};
		\draw (-2.5,0.5) to[out=20,in=190] (-1,0.9) to[out=10,in=185] (-0.3,1) to[out=5,in=150] (6,0) to[out=-30,in=120] (6.4,-0.4) node at (3,1.3) {$\gamma_x$};
		\draw (2,-2) to[out=145,in=-65] (0,0) to[out=115,in=-85] (-0.3,1) to[out=95,in=260] (-0.2,2.5) node at (2.4,-1.5) {$U_{x_0}$};
		\draw[->] (0,0) to (-0.9063078,-0.4226183) node at (-1,-0.2) {$x^1$};
		\draw[->] (0,0) to (-0.4226183,0.9063078) node at (-0.65,0.7) {$x^i$};
	\end{tikzpicture}
	\caption{Construction of the local coordinates}
\end{figure}
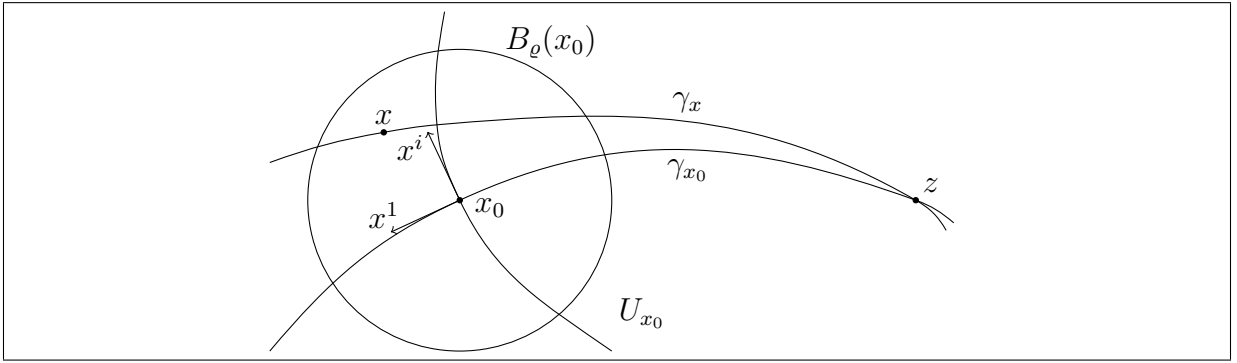

In these coordinates the associated coefficient field at $y\in U:=\Psi(B_{\varrho}(x_0))$ can be written as
\begin{equation*}
	A_{\varepsilon}(y)=A(y,\tfrac{y_1+d(x_0,z)}{\varepsilon})
\end{equation*}
for some $A\colon U\times\mathbb{R}$ continuous in the first, and measurable and $1$-periodic in the second argument. This can be seen by considering \eqref{def:A}: The coefficient field $A_{\varepsilon}$ on $U$ associated to $\mathbb{L}_{\varepsilon}$ takes the form
\begin{equation*}
	(A_{\varepsilon})_{ij}=\overline{\rho}\, \overline{g}(\mathbb{L}_{\varepsilon}\nabla_gx^i,\nabla_gx^j),
\end{equation*}
where $\overline{\rho}:=\rho\circ\Psi^{-1}$ and $\overline{g}:=g\circ\Psi^{-1}$ denote the representation of the quantities in local coordinates.
By the definitions of $\mathbb{L}_{\varepsilon}$ and $x^1$ we see that
\begin{equation*}
	g(\mathbb{L}_{\varepsilon}(x)\tfrac{\partial}{\partial x^i},\tfrac{\partial}{\partial x^j})
	=g(\mathbb{L}(\tfrac{d(x,Z)}{\varepsilon})\tfrac{\partial}{\partial x^i},\tfrac{\partial}{\partial x^j})
	=g(\mathbb{L}(\tfrac{x^1(x)+d(x_0,Z)}{\varepsilon})\tfrac{\partial}{\partial x^i},\tfrac{\partial}{\partial x^j})
\end{equation*}
is only depending on $x^1(x)=y_1$, and $A_{\varepsilon}$ has the desired form with
\begin{equation}\label{eq:ex3.2}
	A_{ij}(y,r):=\overline{\rho}\, \overline{g}(\mathbb{L}(r)\nabla_gx^i,\nabla_gx^j)(y),
\end{equation}
which is continuous in $y\in U$, and measurable and $1$-periodic in $r\in\mathbb{R}$.

For $\varepsilon\to0$ the homogenized matrix $A_{\text{hom}}$ associated with $A_{\varepsilon}$ is given by the homogenization formula \eqref{perhom} for $A$ defined in \eqref{eq:ex3.2}. Therefore $A_{\text{hom}}$ continuously depends on $y\in U$. Moreover the matrix $A_{\text{hom}}(0)$ is independent on the initial choice of $x_0$ and is given by the following weak-$*$ limits in $L^{\infty}(U)$:
\begin{equation*}
	\begin{aligned}
		&\frac{1}{A_{11}(0,\tfrac{\cdot}{\varepsilon})}\rightharpoonup\frac{1}{(A_{\text{hom}})_{11}(0)},\\
		&\frac{A_{i1}(0,\tfrac{\cdot}{\varepsilon})}{A_{11}(0,\tfrac{\cdot}{\varepsilon})}\rightharpoonup\frac{(A_{\text{hom}})_{i1}(0)}{(A_{\text{hom}})_{11}(0)},\quad i=2,\dotsc,n,\\
		&\frac{A_{1j}(0,\tfrac{\cdot}{\varepsilon})}{A_{11}(0,\tfrac{\cdot}{\varepsilon})}\rightharpoonup\frac{(A_{\text{hom}})_{1j}(0)}{(A_{\text{hom}})_{11}(0)},\quad j=2,\dotsc,n,\\
		&A_{ij}(0,\tfrac{\cdot}{\varepsilon})-\frac{A_{i1}(0,\tfrac{\cdot}{\varepsilon})A_{1j}(0,\tfrac{\cdot}{\varepsilon})}{A_{11}(0,\tfrac{\cdot}{\varepsilon})}\rightharpoonup (A_{\text{hom}})_{ij}(0)-\frac{(A_{\text{hom}})_{i1}(0)(A_{\text{hom}})_{1j}(0)}{(A_{\text{hom}})_{11}(0)},\quad i,j=2,\dotsc,n.\\
	\end{aligned}
\end{equation*}

By Lemma \ref{L:periodic}, we have
\begin{equation*}
	(A_{\text{hom}})_{ij}=\overline{\rho}\, \overline{g}(\mathbb{L}_0\nabla_gx^i,\nabla_gx^j).
	\qquad\text{a.e.~in }U,
\end{equation*}
We conclude that $\mathbb{L}_0$ is continuous ($\mu$-a.e.)\ on $B_{\varrho}(x_0)$ and thus (using \eqref{eq:ex3.1}) $g(\mathbb{L}_0(x_0)\tfrac{\partial}{\partial x^i},\tfrac{\partial}{\partial x^j})(x_0)=(A_{\text{hom}})_{ij}(0)$ for $\mu$-a.e.~$x_0\in M$.

As in the previous example we could consider the special case of a diagonal matrix
\begin{equation*}
	\mathbb{L}(r)\tfrac{\partial}{\partial x^i}=a_i(r)\tfrac{\partial}{\partial x^i}\qquad\text{for $i=1,\dotsc,n$}.
\end{equation*}
Then $\mathbb{L}_0(x_0)$ is a diagonal matrix, too, and we have
\begin{equation}\label{Example2:special}
	\begin{aligned}
		&g(\mathbb{L}_0(x_0)\tfrac{\partial}{\partial x^1},\tfrac{\partial}{\partial x^1})(x_0)=\Big(\int_0^1a_1^{-1}\Big)^{-1}\qquad\text{and}\\
		&g(\mathbb{L}_0(x_0)\tfrac{\partial}{\partial x^i},\tfrac{\partial}{\partial x^i})(x_0)=\int_0^1a_i\qquad\text{for $i=2,\dotsc,n$}.
	\end{aligned}
\end{equation}
\medskip

\paragraph{\textbf{Example 3: A radially symmetric corrugated graphical surface}}
We consider the reference manifold
\begin{equation*}
	M_0=\{(r,\theta);r\in(0,R),\theta\in[0,2\pi)\}
\end{equation*}
for some $R>0$, and define a family $M_{\varepsilon}=h_{\varepsilon}(M_0)$ of $2$-dimensional submanifolds of $\mathbb{R}^3$ using uniform bi-Lipschitz immersions $h_{\varepsilon}\colon M_0\to\mathbb{R}^3$,
\begin{equation}\label{Example2:Eq1}
	h_{\varepsilon}(r,\theta)=
	\begin{pmatrix}
		r\sin\theta\\
		r\cos\theta\\
		\varepsilon f(r,\tfrac{r}{\varepsilon})
	\end{pmatrix},
\end{equation}
where $f(0,\infty)\times[0,\infty)\to\mathbb{R}$ is differentiable and $T$-periodic in the second argument. In Figure~\ref{Example2:Fig1} we took $f(r,y)=\sin^2(y)$ to illustrate $M_{\varepsilon}$ for some values of $\varepsilon$.

\begin{figure}[ht]
  \centering
  \begin{minipage}{0.2\textwidth}
    \centering\includegraphics[scale=0.7]{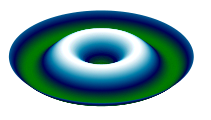}
  \end{minipage}\begin{minipage}{0.2\textwidth}
    \centering\includegraphics[scale=0.7]{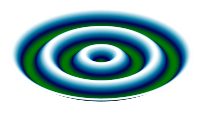}
  \end{minipage}\begin{minipage}{0.2\textwidth}
    \centering\includegraphics[scale=0.7]{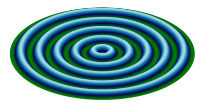}
  \end{minipage}\begin{minipage}{0.2\textwidth}
    \centering$\xrightarrow{\varepsilon\downarrow0}$
  \end{minipage}\begin{minipage}{0.2\textwidth}
    \centering\includegraphics[scale=0.7]{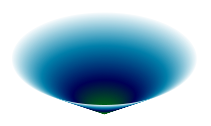}
  \end{minipage}\\\begin{minipage}{0.2\textwidth}
    \centering$\varepsilon=\tfrac{1}{2}$
  \end{minipage}\begin{minipage}{0.2\textwidth}
    \centering$\varepsilon=\tfrac{1}{4}$
  \end{minipage}\begin{minipage}{0.2\textwidth}
    \centering$\varepsilon=\tfrac{1}{8}$
  \end{minipage}\begin{minipage}{0.2\textwidth}
		\centering\mbox{}
  \end{minipage}\begin{minipage}{0.2\textwidth}
    \centering\mbox{}
  \end{minipage}
\caption{A family of rotationally symmetric corrugated graphical surfaces. The three pictures on the left show $M_\ep$ defined via \eqref{Example2:Eq1} with $f=\sin^2$ and decreasing values of $\ep$. The picture on the right shows the limiting surface $N_0$ defined via \eqref{Example2:Eq2}. As $\ep\to 0$ the spectrum of the Laplace-Beltrami operator on $M_{\ep}$ converges to the spectrum of the Laplace-Beltrami operator on $N_0$.}\label{Example2:Fig1}
\end{figure}

Following Lemma~\ref{L:submani} we compute the density
\begin{equation*}
	\rho_{\varepsilon}=\sqrt{\det(dh_{\varepsilon}^{\sf T}dh_{\varepsilon})}
	=\sqrt{r^2+r^2\bigl(\varepsilon\partial_1f(r,\tfrac{r}{\varepsilon})+\partial_2f(r,\tfrac{r}{\varepsilon})\bigr)^2}.
\end{equation*}
and the coefficient field
\begin{align*}
	\mathbb{L}_{\varepsilon}
	&=\rho_{\varepsilon}(dh_{\varepsilon}^{\sf T}dh_{\varepsilon})^{-1}\\
	&=1/\rho_{\varepsilon}
	\begin{pmatrix}
		r^2&0\\
		0&1+\bigl(\varepsilon\partial_1f(r,\tfrac{r}{\varepsilon})+\partial_2f(r,\tfrac{r}{\varepsilon})\bigr)^2
	\end{pmatrix}.
\end{align*}
We find $\rho_{\varepsilon}\stackrel{*}{\rightharpoonup}\rho_0$ weakly-$*$ in $L^{\infty}(M_0)$ with
\begin{equation*}
		\rho_0(r)
		=\tfrac{r}{T}\int_0^T\sqrt{\bigl(\partial_2f(r,y)\bigr)^2+1}\,\mathrm{d}y,
	\end{equation*}
and using \eqref{Example2:special} we see $\mathbb{L}_{\varepsilon}\overset{H}{\to}\mathbb{L}_0$ with
\begin{align*}
	\mathbb{L}_0
	&=
	\begin{pmatrix}
		\Bigl(\tfrac{1}{rT}\int_0^T\sqrt{\bigl(\partial_2f(r,y)\bigr)^2+1}\,\mathrm{d}y\Bigr)^{-1}&0\\
		0&\tfrac{1}{rT}\int_0^T\sqrt{\bigl(\partial_2f(r,y)\bigr)^2+1}\,\mathrm{d}y
	\end{pmatrix}\\
	&=
	\begin{pmatrix}
		\tfrac{r^2}{\rho_0(r)}&0\\
		0&\tfrac{\rho_0(r)}{r^2}
	\end{pmatrix}.
\end{align*}
and get the limiting metric on $M_0$:
\begin{equation*}
	\LimMetric (\xi,\eta)
	=\rho_0\mathbb{L}_0^{-1}\xi\cdot\eta
	=\begin{pmatrix}
			\tfrac{\rho_0(r)^2}{r^2}&0\\
			0&r^2
		\end{pmatrix}
		\xi\cdot\eta.
\end{equation*}
We finally find a bi-Lipschitz immersion $h_0\colon M_0\to\mathbb{R}^3$ such that $dh_0^{\sf T}dh_0=\rho_0\mathbb{L}_0^{-1}$, namely
\begin{equation}\label{Example2:Eq2}
	h_0(r,\theta)=
	\begin{pmatrix}
		r\sin\theta\\
		r\cos\theta\\
		\int_0^r\sqrt{\tfrac{\rho_0(t)^2}{t^2}-1}\,\mathrm{d}t
	\end{pmatrix}.
\end{equation}
By Remark~\ref{SubmanifoldRemark}, the submanifold $\refmani:=h_0(M_0)$ of $\mathbb{R}^3$, which for the case $f(r,y)=\sin^2(y)$ is shown in Figure~\ref{Example2:Fig1}, is the spectral limit of $(M_{\varepsilon})$.
\medskip

\paragraph{\textbf{Example 4: Sphere with radial perturbations oscillating with the latitude}.}
In the same way as in the previous example we can handle the case of a radially perturbed sphere. Again we start with the reference manifold
\begin{equation*}
	M_0=\{(\varphi,\theta);\varphi\in(0,\pi),\theta\in[0,2\pi)\}
\end{equation*}
and define the family $M_{\varepsilon}:=h_{\varepsilon}(M_0)$ of $2$-dimensional submanifolds of $\mathbb{R}^3$ via bi-Lipschitz immersions $h_{\varepsilon}\colon\overline{M}\to M_{\varepsilon}$,
\begin{equation*}
	h_{\varepsilon}(\varphi,\theta)
	=\bigl(1+\varepsilon f(\varphi,\tfrac{\varphi}{\varepsilon})\bigr)
	\begin{pmatrix}
		\sin\varphi\sin\theta\\
		\sin\varphi\cos\theta\\
		\cos\varphi
	\end{pmatrix},
\end{equation*}
where $f\colon(0,\pi)\times[0,\infty)\to\mathbb{R}$ is differentiable and $2\pi$-periodic in the second argument. In Figure~\ref{Example2:Fig2} in the Introduction we choose $f(r,y)=\sin^2(y)$ to picture $M_{\varepsilon}$ for some values of $\varepsilon$.

Doing the same calculations as in the example above we end up with the density
\begin{equation*}
	\rho_0(\varphi)
	=\tfrac{\sin\varphi}{\pi}\int_0^{\pi}\sqrt{(\partial_2f(\varphi,y))^2+1}\,\mathrm{d}y,
\end{equation*}
and the metric
\begin{equation*}
	\tfrac{1}{\rho_0}\mathbb{L}_0=
	\begin{pmatrix}
	\tfrac{\sin^2\varphi}{\rho_0(\varphi)^2}&0\\
		0&\tfrac{1}{\sin^2\varphi}
	\end{pmatrix}.
\end{equation*}
and again we find a bi-Lipschitz immersion $h_0\colon M_0\to\mathbb{R}^3$ such that $dh_0^{\sf T}dh_0=\rho_0\mathbb{L}_0^{-1}$, namely
\begin{equation*}
	h_0(\varphi,\theta)=
	\begin{pmatrix}
		\sin\varphi\sin\theta\\
		\sin\varphi\cos\theta\\
		\int_0^\varphi\sqrt{\tfrac{\rho_0(t)^2}{\sin^2t}-\cos^2t}\,\mathrm{d}t
	\end{pmatrix}.
\end{equation*}
Thus the submanifold $\refmani:=h_0(M_0)$ of $\mathbb{R}^3$, which for the case $f(r,y)=\sin^2(y)$ is pictured in Figure~\ref{Example2:Fig2}, is the spectral limit of the sequence $(M_{\varepsilon})$.

\paragraph{\textbf{Example 5: A locally corrugated graphical surface}}

We finally want to discuss an example with oscillations in several Voronoi cells which can be treated locally.

Let $Y\subset\mathbb{R}^2$ be relatively-compact and open. Consider a set $Z\in Y$ of isolated points. For every point $z\in Z$ we use a smooth function $\psi_z\colon[0,\infty)\to[0,1]$ to define a rotationally symmetric cut-off function $\psi_z(|\cdot-z|)$ such that
\begin{equation*}
	\begin{cases}
		\text{$\psi_z(0)=1$,}&\\
		\text{$\supp\psi_z(|\cdot-z|)\cap\supp\psi_{z'}(|\cdot-z'|)=\emptyset$ for all $z'\in Z\setminus\{z\}$.}
	\end{cases}
\end{equation*}
Now we consider a smooth $T$-periodic function $f\colon[0,\infty)\to\mathbb{R}$ and define $M_{\varepsilon}$ as the graph of the function $h_{\varepsilon}\colon M_0:=Y\setminus Z\to\mathbb{R}$,
\begin{equation*}
  h_{\varepsilon}(x):=\sum_{z\in Z}\varepsilon f\big(\tfrac{|x-z|}{\varepsilon}\big)\psi_z(|x-z|)\in\mathbb{R}^3,
\end{equation*}
which we regard as a two-dimensional submanifold of $\R^3$. In Figure~\ref{Example2:Fig3} in the Introduction we took $f(y)=\sin^2(y)$ to show $M_\varepsilon$ for some values of $\varepsilon$.

Doing the same calculations as in the previous examples locally in each Voronoi cell we get a function $h_0\colon M_0\to\mathbb{R}$,
\begin{equation*}
  h_0(x):=x\mapsto\sum_{z\in Z}\int_0^{|x-z|}\sqrt{\tfrac{\rho_{0,z}(t)^2}{t^2}-1}\,\mathrm{d}t\in\mathbb{R}^3,
\end{equation*}
where $\rho_{0,z}(r)=\tfrac{r}{T}\int_0^{T}\sqrt{f'(y)^2\psi_z(r)^2+1}\,\mathrm{d}y$, such that the graph of $h_0$, which is shown in Figure~\ref{Example2:Fig3} for $f(y)=\sin^2(y)$, is the spectral limit of $(M_{\varepsilon})$.



\section{Proofs}\label{SecProofs}

\subsection{Proof of Proposition~\ref{P1}, Lemma~\ref{L:local}  and Lemma~\ref{lemB}}\label{S:pP1}
The argument consists of two parts. In the first part we identify the limiting tensor field $\mathbb L_0$. For this purpose, we consider the operators 
\begin{equation}\label{P1:eq1}
 \cL_{\ep}^*\colon H^1_0(B)\to H^{-1}(B),\qquad  \cL_{\ep}^*u:=-\Div(\mathbb{L}_{\ep}^*\nabla),
\end{equation}
where $\mathbb L_\ep^*$ denotes the adjoint of $\mathbb L_\ep$ and is defined by the identity $(\mathbb L_\ep^*\xi,\eta)=(\xi,\mathbb L_\ep\eta)$ for all vector fields $\xi,\eta$. Since the operator is uniformly elliptic (with constants independent of $\varepsilon$) we can deduce the existence of a linear isomorphism $\cL_0^*$, whose inverse is the limit of $(\cL_\ep^*)^{-1}$ in the weak operator topology. Indeed, this follows from the following standard compactness result:
\begin{lemma}\label{lem3}
Let $V$ be a reflexive separable Banach space and $(T_{\ep})$ be a sequence of linear operators $T_{\ep}\colon V\to V^{\prime}$ that is uniformly bounded and coercive, i.e. there exists $C>0$ (independent of $\ep$) such that the operator norm of $T_{\ep}$ is bounded by $C$ and
\begin{equation} \label{eq6.0708}
	\langle T_{\ep}v,v\rangle_{V^{\prime},V}\ge\tfrac{1}{C}\|v\|_V^2\qquad\text{for all }v\in V.
\end{equation}
Then there exists a linear bounded operator $T_0\colon V\to V'$ satisfying $\eqref{eq6.0708}$ and  for a subsequence (not relabeled) we have $T_{\ep}^{-1}\wto T_0^{-1}$ in the weak operator topology, that is for all $f\in V^{\prime}$ we have
\begin{equation*}
	T_{\ep}^{-1}f\rightharpoonup T_0^{-1}f\qquad\text{weakly in }V.
\end{equation*}
\end{lemma}
(For a proof, e.g., see~\cite[Proposition~4]{MT}). %
We then show that $\cL_0^*$ can in fact be written in divergence form: $\cL_0^*=-\Div(\mathbb L_0^*\nabla)$ with an appropriate $(1,1)$-tensor field $\mathbb L_0^*$. In order to define $\mathbb L_0^*$ with help of $\cL_0^*$, we introduce auxiliary functions whose gradients span the tangent space. More precisely, we recall the following fact:
\begin{remark}\label{R:vs}
  Let $B\Subset M$ denote an open ball with radius smaller than the injectivity radius at its center. Then there exist $v_1,\dotsc, v_n\in C^\infty_c(B)$ such that $T(\frac12B)$ is spanned by the vector fields $\nabla v_1,\dotsc,\nabla v_n$, i.e.
  \begin{equation}\label{L:local:eq1}
    \forall y\in\tfrac{1}{2}B\,:\qquad T_y(\tfrac{1}{2}B)=\operatorname{span}\{\nabla v_1(y),\dotsc,\nabla v_n(y)\}.
  \end{equation}
\end{remark}
Following ideas of Tartar and Murat, we associate with $v_1,\ldots, v_n$ oscillating test-functions $v_{1,\ep},\ldots, v_{n,\ep}$ that allow to pass to the limit in products of weakly convergent sequences of the form $(\mathbb L_\ep \nabla u_\ep,\nabla v_{i,\ep})$. The argument invokes the following variant of the Div-Curl Lemma for manifolds:
\begin{lemma}[Div-Curl Lemma]\label{lemC}
  Let $\Omega\subset M$ be open and let $(\xi_{\ep})\subset L^2(T\Omega)$, $(v_{\ep})\subset H^1(\Omega)$ denote sequences such that
  \begin{equation*}
    \begin{cases}
      \xi_{\ep}\rightharpoonup\xi&\text{weakly in }L^2(T\Omega),\\
      \Div\xi_{\ep}\to \Div\xi&\text{in }H^{-1}(\Omega),\\
    \end{cases}
    \qquad\text{and}\qquad
    v_{\ep}\rightharpoonup v
    \quad\text{weakly in }H^1(\Omega).
  \end{equation*}
  Then
  \begin{equation*}
    \int_{\Omega}(\xi_{\ep},\nabla v_{\ep})\varphi\,\mathrm{d}\mu\to\int_{\Omega}(\xi,\nabla v)\varphi\,\mathrm{d}\mu\qquad\text{for all }\varphi\in C^\infty_c(\Omega).
  \end{equation*}
  Moreover, if $v_\ep,v\in H^1_0(\Omega)$, then
  \begin{equation*}
    \int_{\Omega}(\xi_{\ep},\nabla v_{\ep})\,\mathrm{d}\mu\to\int_{\Omega}(\xi,\nabla v)\,\mathrm{d}\mu.
  \end{equation*}
\end{lemma}
We present the short proof for the reader's convenience:
\begin{proof}[Proof of Lemma~\ref{lemC}]
  In the case $v_\ep\in H^1_0(\Omega)$ the statement follows by an integration by parts. In the general case, for $\varphi\in C_c^{\infty}(\Omega)$ we have
  \begin{equation}\label{eq6_0909}
    \int_{\Omega}(\xi_{\ep},\nabla v_{\ep})\varphi
	=\int_{\Omega}(\xi_{\ep},\nabla (v_{\ep}\varphi))-\int_{\Omega}(\xi_{\ep},v_{\ep}\nabla\varphi)
	=-\langle \Div\xi_{\ep},v_{\ep}\varphi\rangle-\int_{\Omega}(\xi_{\ep},v_{\ep}\nabla\varphi).
\end{equation}
Regarding the first term of the right-hand side of \eqref{eq6_0909},
\begin{equation*}
  -\langle \Div\xi_{\ep},v_{\ep}\varphi\rangle
	\to-\langle \Div\xi,v\varphi\rangle
	=\int_{\Omega}(\xi,v\nabla \varphi)+\int_{\Omega}(\xi,\varphi \nabla v).
\end{equation*}
For the second term of the right-hand side of \eqref{eq6_0909}, since $v_{\ep}\rightharpoonup v$ in $H^1(\Omega)$, 
for any relatively compact open set $\Omega^{\prime}\subset M$, 
there exists a subsequence of $(v_{\ep})$ converging to $v$ in $L^2(\Omega^{\prime})$ by Rellich's theorem; 
in particular, $v_{\ep}\nabla\varphi\to v\nabla\varphi$ in $L^2(TM)$ and thus $\int_{\Omega}(\xi_{\ep},v_{\ep}\nabla\varphi)\to\int_{\Omega}(\xi,v\nabla\varphi)$. 
Hence, the right-hand side of \eqref{eq6_0909} converges to $\int_{\Omega}(\xi,\nabla v)\varphi$.
\end{proof}

In a second step, we then show that $\mathbb L_0$ (the adjoint of $\mathbb L_0^*$) is an $H$-limit of $(\mathbb L_\ep)$. To that end we need to consider for (arbitrary but fixed) subdomains $\omega\Subset\Omega$ the localized operators 
\begin{equation}\label{P1:eq2}
  \cL_{\ep}\colon H^1_0(\omega)\to
  H^{-1}(\omega),\qquad \cL_\ep u:=-\Div(\mathbb{L}_{\ep}\nabla u),
\end{equation}
and show that $\cL_\ep^{-1}\to\cL_0^{-1}$ in the weak operator topology.
\begin{proof}[Proof of Proposition~\ref{P1}]
  In the proof we pass to various subsequences and it turns out to be necessary to keep track of them. For a lean notation we denote by $E\subset (0,\infty)$ the set of $\ep$'s of the given sequence $(\mathbb L_\ep)=(\mathbb L_\ep)_{\ep\in E}$. We represent subsequences by means of subsets $E',E'',\dotsc\subset E$ that have a cluster point at $0$. We follow the convention to write
  \begin{equation*}
    c_{\ep}\to c_0\qquad (\ep\in E'),
  \end{equation*}
  if and only if for any sequence $(\ep_j)_{j\in\mathbb N}\subset E'$ with $\ep_j\to 0$ we have $c_{\ep_j}\to c_0$.
  \smallskip

  \textit{Step 1.} Choice of the subsequence and definition of $\mathbb L_0$.
\smallskip

Let $\cL_\ep^*$ be defined by \eqref{P1:eq1} and fix $v_1,\dotsc, v_n\in C^\infty_c(B)$ according to Remark~\ref{R:vs}.  We claim that there exits a measurable $(1,1)$-tensor field
$\mathbb{L}_0\colon\tfrac{1}{2}B\to\operatorname{Lin}(T(\tfrac{1}{2}B))$,
a subsequence $E'\subset E$, and functions
  $(v_{1,\ep}),\dotsc,(v_{k,\ep})\subset H^1_0(B)$ (the so called
  oscillating test functions) such that for $k=1,\dotsc,n$ and $\ep\in
  E'$ we have
  \begin{equation}\label{eq110922}
    \begin{cases}
      v_{k,\ep}\rightharpoonup v_k&\quad\mbox{weakly in }H^1_0(B),\\
      v_{k,\ep}\to v_k&\quad\mbox{in }L^2(B),\\
      (\cL_{\ep}^*v_{k,\ep})&\quad\mbox{strongly converges in }H^{-1}(B),\\
      \mathbb{L}^*_{\ep}\nabla v_{k,\ep}\wto\mathbb{L}^*_0\nabla v_k&\quad\mbox{weakly
        in }L^2(T(\tfrac{1}{2}B)).
    \end{cases}
  \end{equation}
  For the argument note that by uniform ellipticity of
  $\mathbb{L}_{\ep}^*$ and the boundedness of $B$, there exists
  $C=C(B,\lambda)>0$ such that
  \begin{equation*}
    \langle\cLe^*u,u\rangle=\int_B(\mathbb{L}_{\ep}^*\nabla u,\nabla u)\ge C\|u\|^2_{H^1(B)}, 
  \end{equation*}
  and thus by Lemma~\ref{lem3} there is $\cL_0^*\colon H^1_0(B)\to
  H^{-1}(B)$ and a subsequence $E''\subset E$ such that for all $f\in
  H^{-1}(B)$ and $\ep\in E''$
  \begin{equation*}
    (\mathcal{L}_{\ep}^{*})^{-1}f\rightharpoonup(\mathcal{L}_0^*)^{-1}f\qquad\mbox{weakly in }H^1_0(B).
  \end{equation*}
  For $k=1,\dotsc,n$ define
  \begin{equation*}
    v_{k,\ep}:=(\mathcal{L}_{\ep}^*)^{-1}\mathcal{L}_0^*v_k,
  \end{equation*}
  which by uniform ellipticity of $\mathbb{L}_{\ep}^*$ and
  Poincar{\'{e}}'s inequality in $H^1_0(B)$ are bounded uniformly in
  $\ep$. Hence there exits vector fields $\ell_1,\dotsc,\ell_n\in L^2(TB)$ and
  another subsequence $E'\subset E''$ such that we have for $\ep\in
  E'$
  \begin{equation*}
    \begin{cases}
      v_{k,\ep}\rightharpoonup v_k&\quad\mbox{weakly in }H^1_0(B),\\
      v_{k,\ep}\to v_k&\quad\mbox{in }L^2(B),\\
      \mathbb{L}_{\ep}^*\nabla v_{k,\ep}\wto\ell_k&\quad\mbox{weakly in
      }L^2(TB).
    \end{cases}
  \end{equation*}
  Next, we define the tensor field $\mathbb L_0^*$ by the identity
  \begin{equation*}
    \forall k\in\{1,\ldots,n\}\,:\qquad \mathbb{L}_0^*\nabla v_k=\ell_k\qquad\mu\text{-a.e.~in }\tfrac{1}{2}B.
  \end{equation*}
  Indeed, since $\nabla v_1,\dotsc,\nabla v_n$ span $T(\tfrac{1}{2}B)$ the above identity defines $\mathbb{L}_0^*$ uniquely and the last identity in
  \eqref{eq110922} is satisfied by construction. It remains to check the strong convergence of
  $(\mathcal{L}^*_{\ep}v_{k,\ep})$. In fact the stronger statement
  $\mathcal{L}^*_{\ep}v_{k,\ep}=\mathcal{L}_0^*v_k$ is valid, which is
  a direct consequence of the definition of $v_{k,\ep}$.  \medskip

  \textit{Step 2.} $H$-convergence of $\mathbb L_\ep$ to $\mathbb{L}_0$ in $\frac{1}{2}B$.
\smallskip

  Let the subsequence $E'$, the tensor field $\mathbb{L}_0$, and $(v_{k,\ep})$ be defined as in Step 1. We
  claim that $(\mathbb{L}_{\ep})$ $H$-converges to $\mathbb{L}_0$ in
  $\tfrac{1}{2}B$ for $\ep\in E'$. To that end let
  $\omega\Subset\tfrac{1}{2}B$ and let $\mathcal{L}_{\ep}$ be defined by \eqref{P1:eq2}. Arguing as in the previous step, we can find
  another subsequence $E''\subset E'$ and a bounded linear, coercive
  operator $\mathcal{L}_0\colon H^1_0(\omega)\to H^{-1}(\omega)$ such
  that
  \begin{equation}\label{P1:1.2:eq1}
    \mathcal{L}_{\ep}^{-1}\wto\mathcal{L}_0^{-1}\qquad\text{in the weak operator topology for }\ep\in E''.
  \end{equation}
  We only need to show that 
  \begin{equation}\label{P1:1.2:eq2}
    \mathcal{L}_0 u_0=-\Div(\mathbb{L}_0\nabla u_0),
  \end{equation}
  for arbitrary $u_0\in H^1_0(\omega)$. For the argument  set
  $u_{\ep}:=\mathcal{L}_{\ep}^{-1}\mathcal{L}_0u_0$ so that by
  \eqref{P1:1.2:eq1},
  \begin{equation}\label{P1:1.2:eq5}
    u_\ep\wto u_0\qquad\text{weakly in }H^1_0(\omega)\text{ and strongly in $L^2(\omega)$}\text{ for }\ep\in E''.
  \end{equation}
  Consider $J_{\ep}:=\mathbb{L}_{\ep}\nabla u_{\ep}$. By uniform ellipticity
  of $\mathbb{L}_{\ep}$ the sequences $(J_{\ep})$ is bounded in
  $L^2(T\omega)$. Hence, there exits $J_0\in L^2(T\omega)$ and
  another subsequence $E'''\subset E''$ such that
  \begin{equation}\label{P1:1.2:eq4}
    J_{\ep}=\mathbb{L}_{\ep}\nabla u_{\ep}\wto J_0\qquad\text{weakly in }L^2(T\omega)\text{ for }\ep\in E'''.
  \end{equation}
  Combined with the identity $-\Div J_{\ep}=\mathcal{L}_0u_0$ (which
  follows from the definition of $u_{\ep}$) we find that
  \begin{equation}\label{P1:1.2:eq3}
    -\Div J_0=\mathcal{L}_0u_0.
  \end{equation}
  Hence, for any test function $\varphi\in C^{\infty}_c(\omega)$, the
  convergence properties of $(v_{k,\ep})$ yield
  \begin{align*}
    \int_{\omega}(J_{\ep},\varphi \nabla v_{k,\ep})
    &=\int_{\omega}(J_{\ep},\nabla(\varphi v_{k,\ep}))-\int_{\omega}(J_{\ep},v_{k,\ep}\nabla\varphi)\\
    &=\langle\mathcal{L}_0u_0,\varphi v_{k,\ep}\rangle-\int_{\omega}(J_{\ep},v_{k,\ep}\nabla\varphi)\\
    &\to\langle\mathcal{L}_0u_0,\varphi v_k\rangle-\int_{\omega}(J_0,v_k\nabla\varphi)\\
    &=\int_{\omega}(J_0,\varphi \nabla v_k).
  \end{align*}
  On the other hand, since $\mathbb{L}^*_{\ep}\nabla v_{k,\ep}\rightharpoonup\mathbb{L}^*_0\nabla v_k$
  weakly in $L^2(T\tfrac{1}{2}B)$ and
  $(-\Div(\mathbb{L}^*_{\ep}\nabla v_{k,\ep}))$ strongly converges in
  $H^{-1}(\frac12 B)$ by \eqref{eq110922}, the Div-Curl Lemma
  (Lemma~\ref{lemC}) yields
  \begin{equation*}
    \int_\omega (J_\ep ,\varphi \nabla v_{k,\ep})
    =\int_\omega (\varphi \nabla u_{\ep},\mathbb L_\ep^*\nabla v_{k,\ep})
    \to\int_\omega (\varphi \nabla  u_0, \mathbb{L}^*_0\nabla v_{k})
    =\int_\omega (\mathbb L_0\nabla u_0, \varphi \nabla v_{k}).
  \end{equation*}
  Hence, by combining the previous two identities we conclude that
  \begin{equation*}
    \int_\omega (\mathbb L_0\nabla  u_0, \varphi \nabla v_{k})
    =\int_\omega (J_0, \varphi \nabla v_{k}).
  \end{equation*}
  Since $\varphi\in C^\infty_c(\omega)$ is arbitrary and since
  $\nabla v_1,\dotsc, \nabla v_n$ spans $T\omega$, we get
  $J_0=\mathbb{L}_0\nabla u_0$ $\mu$-a.e.\ in $\omega$. Thus
  \eqref{P1:1.2:eq2} follows from \eqref{P1:1.2:eq3}. Moreover, since
  $J_0$ and $\mathcal L_0$ are uniquely determined by $\mathbb{L}_0$,
  the convergence in \eqref{P1:1.2:eq1}, \eqref{P1:1.2:eq5}, and
  \eqref{P1:1.2:eq4} holds for the entire sequence $E'$ (which in
  particular is independent of $\omega$).

  Next we argue that $\mathbb{L}_0\in\mathcal{M}(\omega,\lambda,\Lambda)$. Indeed, from
  \eqref{P1:1.2:eq5} and \eqref{P1:1.2:eq4} and the Div-Curl Lemma (Lemma~\ref{lemC}) we
  learn that for any non-negative $\varphi\in C^\infty_c(\omega)$ we have
  \begin{equation*}
    \int_{\omega}(\mathbb{L}_{\ep}\nabla u_{\ep},\nabla u_{\ep})\varphi\to\int_{\omega}(\mathbb{L}_0\nabla u_0,\nabla u_0)\varphi.
  \end{equation*}
  By uniform ellipticity of $\mathbb{L}_{\ep}$ in form of
  \eqref{ellipt1}, we have
  $\int_{\omega}(\mathbb{L}_{\ep}\nabla u_{\ep},\nabla u_{\ep})\rho\geq\lambda\int_{\omega}|\nabla u_{\ep}|^2\rho$,
  and thus
  \begin{equation*}
    \int_{\omega}(\mathbb{L}_0\nabla u_0,\nabla u_0)\varphi\geq\lambda\int_{\omega}|\nabla u_0|^2\varphi.
  \end{equation*}
  Since this is true for all $u_0$ and $\varphi$, we conclude that
  $\mathbb{L}_0$ satisfies the lower ellipticity condition,
  cf.~\eqref{ellipt1} $\mu$-a.e.\ in $\omega$. On the other hand
  \eqref{ellipt2} implies
  \begin{equation*}
    \int_{\omega}(\mathbb{L}_{\ep}\nabla u_{\ep},\nabla u_{\ep})\varphi=\int_{\omega}(\mathbb{L}_{\ep}\nabla u_{\ep},\mathbb{L}_{\ep}^{-1}\mathbb{L}_{\ep}\nabla u_{\ep})\varphi\geq\Lambda\int_{\omega}|\mathbb{L}_{\ep}\nabla u_0|^2\varphi,
  \end{equation*}
  and thus by the same reasoning as before, we get for $\mu$-a.e.\
  $x\in\omega$ and all $\xi\in T_x\omega$
  \begin{equation*}
    \Lambda|\mathbb{L}_0(x)\xi|^2\leq(\mathbb{L}_0(x)\xi,\xi).
  \end{equation*}
  Substituting $\xi=\mathbb{L}_0^{-1}(x)\xi'$ yields the boundedness
  condition, cf.~\eqref{ellipt2}.

  Since the above arguments hold for arbitrary
  $\omega\Subset\tfrac{1}{2}B$ we deduce that
  $\mathbb{L}_0\in\mathcal{M}(\frac{1}{2}B,\lambda,\Lambda)$ and that
  $(\mathbb{L}_{\ep})$ $H$-converges to $\mathbb{L}_0$ in
  $\frac{1}{2}B$ for $\ep\in E'$.
\end{proof}

Next we present the proof of the auxiliary statements Lemma~\ref{L:local} and Lemma~\ref{lemB}.
\begin{proof}[Proof of Lemma~\ref{L:local}]
\textit{Step 1: Proof of part (a).}
\smallskip

Let $x\in\omega$ and denote by $B\Subset\omega$ an open ball centered at $x$ and with a radius 
that is smaller than the injectivity radius of $\Omega$ at $x$. Fix $v_1,\dotsc,v_n\in C^\infty_c(B)$ according to Remark~\ref{R:vs}. For $k\in\{1,\dotsc,n\}$ set $f\in H^{-1}(B)$ by $f:=-\Div(\mathbb{L}_0\nabla v_k)$ and define $v_{\ep}\in H^1_0(B)$ as the unique solutions to $-\Div(\mathbb L_\ep\nabla v_\ep)=f$ in $H^{-1}(B)$. By $H$-convergence of $(\mathbb{L}_{\ep})$ and the definition of $f$ we have $v_{\ep}\wto v_k$ weakly in $H^1_0(B)$ and $\mathbb{L}_{\ep}\nabla v_{\ep}\wto\mathbb{L}_0\nabla v_k$ weakly in $L^2(B)$. 
Likewise, by $H$-convergence of $(\widetilde{\mathbb{L}}_{\ep})$ to $\widetilde{\mathbb{L}}_0$ and since $\widetilde{\mathbb{L}}_{\ep}=\mathbb{L}_{\ep}$ on $B$, 
we find that $\mathbb{L}_{\ep}\nabla v_{\ep}\wto\widetilde{\mathbb{L}}_0\nabla v_k$ weakly in $L^2(B)$, and thus $(\widetilde{\mathbb{L}}_0-\mathbb{L}_0)\nabla v_k=0$ $\mu$-a.e.\ in $B$. Since $k$ was arbitrary, the last identity holds for all $k=1,\ldots,n$. Hence \eqref{L:local:eq1} yields $\mathbb{L}_0=\widetilde{\mathbb{L}}_0$ $\mu$-a.e.~in $\tfrac{1}{2}B$. Since $x$ is arbitrary, the last identity holds $\mu$-a.e.\ in $\omega$.
\medskip

\textit{Step 2: Proof of (b).}
\smallskip

Let $\omega\Subset\Omega$. We define $\cL_\ep$ and $\cL_0$ according to \eqref{P1:eq2} and denote the adjoint operators by $\cL_\ep^*$, $\cL_0^*$, i.e.,
\begin{align*}
  &\cL_{\ep}^*\colon H^1_0(\omega)\to
  H^{-1}(\omega),\qquad \cL_\ep^*:=-\Div(\mathbb{L}_{\ep}^*\nabla),\\
  &\cL_{0}^*\colon H^1_0(\omega)\to
  H^{-1}(\omega),\qquad \cL_0^*:=-\Div(\mathbb{L}_{0}^*\nabla).
\end{align*}
Fix $f\in H^{-1}(\omega)$ and let $u_{\ep},u_0\in H^1_0(\omega)$ be the unique solutions to $\cL_\ep^*u_\ep=f$ and $\cL_0^*u_0=f$. It suffice to show that $u_{\ep}\wto u_0$ weakly in $H^1_0(\omega)$  and $\mathbb L_\ep^* \nabla u_\ep\wto \mathbb L_0^* \nabla u_0$ weakly in $L^2(T\omega)$. Since the limiting equation uniquely determines $u_0$, it suffices to prove the statements up to a subsequence. By a standard energy estimate and the uniform boundedness of $(\mathbb L_\ep^*)$ the sequences $(u_\ep)$ and $(\mathbb L_\ep^* \nabla u_\ep)$ are bounded in $H^1_0(\omega)$ and $L^2(T\omega)$, respectively. Hence, there exits $\tilde u_0\in H^1_0(\omega)$ and $J_0\in L^2(T\omega)$ such that for a subsequence (not relabeled),
\begin{equation*}
  \begin{cases}
    u_{\ep}\rightharpoonup \tilde u_0&\quad\text{weakly in }H^1_0(\omega),\\
    \mathbb{L}_{\ep}^*\nabla u_{\ep}\rightharpoonup J_0&\quad\text{weakly in }L^2(T\omega).
  \end{cases}
\end{equation*}
In the next two substeps we complete the argument by showing $\tilde u_0=u_0$ and $J_0=\mathbb L_0^*\nabla u_0$.
\smallskip

\textit{Substep 2.1. Argument for $\tilde u_0=u_0$:} Let $v_0\in H^1_0(\omega)$ and consider  $v_{\ep}:=(\mathcal{L}_{\ep})^{-1}\mathcal{L}_0v_0$. Thanks to $\mathbb{L}_{\ep}\Hto\mathbb{L}_0$ we have
\begin{equation*}
	\begin{cases}
		v_{\ep}\rightharpoonup v_0&\quad\text{weakly in }H^1_0(\omega)\text{ and strongly in }L^2(\omega),\\
		\mathbb{L}_{\ep}\nabla v_{\ep}\rightharpoonup\mathbb{L}_0\nabla v_0&\quad\text{weakly in }L^2(T\omega).
	\end{cases}
\end{equation*}
The Div-Curl Lemma (Lemma \ref{lemC}) thus yields
\begin{eqnarray*}
  \int_\omega (\mathbb{L}_{\ep}^*\nabla u_{\ep},\nabla v_{\ep})&=&\int_\omega (\nabla u_{\ep},\mathbb{L}_{\ep}\nabla v_{\ep})\to\int_\omega(\nabla \tilde u_0,\mathbb{L}_0\nabla v_0)=\int_\omega(\mathbb{L}_0^*\nabla \tilde u_0,\nabla v_0)\\
  &=&  \langle\cL^*_0 \tilde u_0,v_0\rangle.
\end{eqnarray*}
Since, on the other hand we have $\int_\omega (\mathbb{L}_{\ep}^*\nabla u_{\ep},\nabla v_{\ep})=\langle f,v_{\ep}\rangle\to\langle f,v_0\rangle$, and since $v_0\in H^1_0(\omega)$ is arbitrary, we conclude $\mathcal{L}_0^*\tilde u_0=f$ in $H^{-1}_0(\omega)$. Since the kernel of $\cL_0^*$ is trivial, we deduce that $\tilde u_0=u_0$.
\medskip

\textit{Substep 2.2:  Argument for $J_0=\mathbb L_0^*\nabla u_0$}. Let $B\Subset\omega$ be an open ball with radius less than the injectivity radius at its center and fix $v_1,\dotsc,v_n\in C^{\infty}_c(B)\subset C^\infty_c(\omega)$ according to Remark~\ref{R:vs}. Consider $v_{\ep}:=(\mathcal{L}_{\ep})^{-1}\mathcal{L}_0v_j$ and note that $\mathbb{L}_{\ep}\Hto \mathbb{L}_0$ yields
\begin{equation*}
  \begin{cases}
    v_{\ep}\rightharpoonup v_j&\quad\text{weakly in }H^1_0(\omega)\text{ and strongly in }L^2(\omega),\\
    \mathbb{L}_{\ep}\nabla v_{\ep}\rightharpoonup\mathbb{L}_0\nabla v_j&\quad\text{weakly in }L^2(T\omega),
  \end{cases}
\end{equation*}
Thus for any $\varphi\in C^\infty_c(\omega)$ the Div-Curl Lemma (Lemma \ref{lemC}) yields 
\begin{equation*}
  \int_\omega(\mathbb{L}_{\ep}^*\nabla u_{\ep},\nabla v_{\ep})\varphi\to\int_\omega(J_0,\nabla v_j)\varphi,
\end{equation*}
and thus
\begin{equation*}
  \int_\omega (\mathbb{L}_{\ep}^*\nabla u_{\ep},\nabla v_{\ep})\varphi=\int_\omega(\nabla u_{\ep},\mathbb{L}_{\ep}\nabla v_{\ep})\varphi\to\int_\omega(\nabla u_0,\mathbb{L}_0\nabla v_j)\varphi=\int_\omega(\mathbb{L}_0^*\nabla u_0,\nabla v_j)\varphi.
\end{equation*}
Since $\varphi\in C^\infty_c(\omega)$ is arbitrary because of \eqref{L:local:eq1}, we get $J_0=\mathbb{L}_0^*\nabla u_0$.
\end{proof}

\begin{proof}[Proof of Lemma \ref{lemB}]
Let $\cL_\ep$ and $\cL_0$ be defined by \eqref{P1:eq2} and denote by $\cL_\ep^*$ and $\cL_0^*$ the adjoint operators. Note that $u_0$ is uniquely determined by
\begin{equation}\label{lemB:eq1}
  \mathcal L_0 u_0=f_0-\Div(\mathbb{L}_0G_0)-\Div F_0\qquad\text{in }H^{-1}(\omega).
\end{equation}
We first note that (up to a subsequence) $(u_{\ep})$ converges weakly in $H^1_0(\omega)$ to some $\tilde u_0\in H^1_0(\omega)$, and $(\mathbb{L}_{\ep}\nabla u_{\ep})$ converges weakly in $L^2(T\omega)$ to some $J_0\in L^2(T\omega)$. We first claim that $\tilde u_0$ solves \eqref{lemB:eq1} (which by uniqueness of the solution implies that $\tilde u_0=u_0$). For the argument let $v_0\in H^1_0(\omega)$ and consider the oscillating test-function $v_{\ep}:=(\mathcal{L}_{\ep}^*)^{-1}\mathcal{L}_0^*v_0\in H^1_0(\omega)$.
Since $\mathbb L_{\ep}^*\Hto\mathbb L_0^*$ by Lemma~\ref{L:local}, and $\mathcal{L}_{\ep}^*v_{\ep}=\mathcal{L}_0^*v_0$, we deduce that
\begin{equation*}
  \begin{cases}
    v_{\ep}\rightharpoonup v_0&\quad\text{weakly in }H^1_0(\omega)\text{ and strongly in }L^2(\omega),\\
  \mathbb{L}_{\ep}^*\nabla v_{\ep}\rightharpoonup\mathbb{L}_0^*\nabla v_0&\quad\text{weakly in }L^2(T\omega).
\end{cases}
\end{equation*}
Thanks to $u_{\ep}\rightharpoonup \tilde u_0$ weakly in $H^1_0(\omega)$ and the Div-Curl Lemma (Lemma \ref{lemC}) we get on the one hand
\begin{eqnarray*}
  \langle \cL_\ep u_\ep,v_\ep\rangle&=&\int_{\omega}(\mathbb{L}_{\ep}\nabla u_{\ep},\nabla v_{\ep})=\langle f_{\ep},v_{\ep}\rangle+\int_{\omega}(G_{\ep},\mathbb{L}_{\ep}^*\nabla v_\ep)+(F_{\ep},\nabla v_{\ep})\\
  &\to&\int_{\omega} f_0v_0+\int_{\omega}(G_0,\mathbb L_0^*\nabla v_0)+(F_0,\nabla v_0)\\
  &=&\int_{\omega} f_0v_0+\int_{\omega}(\mathbb L_0G_0+F_0,\nabla v_0),
\end{eqnarray*}
and on the other hand
\begin{eqnarray*}
\langle \cL_\ep u_\ep,v_\ep\rangle&=&\langle \cL_\ep^* v_\ep,u_\ep\rangle=\int_{\omega}(\nabla u_{\ep},\mathbb{L}_{\ep}^*\nabla v_{\ep})\to\int_{\omega}(\nabla\tilde u_0,\mathbb{L}_0^*\nabla v_0)=\int_{\omega}(\mathbb L_0\nabla \tilde u_0,\nabla v_0)\\
&=&\langle \cL_0\nabla\tilde u_0,\nabla v_0\rangle.
\end{eqnarray*}
Since $v_0\in H^1_0(\omega)$ is arbitrary, we conclude that $\tilde u_0$ solves \eqref{lemB:eq1} and thus $\tilde u_0=u_0$. Moreover, by the argument of Substep~2.1 in the proof of Lemma~\ref{L:local} (b), we deduce that $J_0=\mathbb{L}_0\nabla u_0$, which completes the argument.\end{proof}


\subsection{Proof of Theorem~\ref{T1}}\label{S:pT1}

The proof is structured as follows: In Step~1 we pass to a subsequence and define the $H$-limit $\mathbb L_0$ by appealing to a covering of $M$ by balls, Proposition~\ref{P1}, and Lemma~\ref{L:local}; (at this point we only have $H$-convergence on balls). In Step~2 we show part (b) of the theorem and recover (a) as a special case.
\medskip

\textit{Step 1.} Choice of the subsequence and definition of $\mathbb{L}_0$.\\
Let $(B_j)$ denote a countable covering of $M$ by open balls with $4B_j\Subset M$
such that the radius of $B_j$ is smaller than a quarter of the injectivity radius of $M$ at the center of $B_j$. 
For every $j\in\mathbb{N}$ Proposition \ref{P1} provides a subsequence of $(\mathbb{L}_{\ep})$ 
$H$-converging to some $\mathbb{L}_{j,0}\in\mathcal{M}(2B_j,\lambda,\Lambda)$ in $2B_j$. 
Thus (by a diagonal subsequence argument) we can choose a subsequence 
$E'\subset E$ such that $(\mathbb L_\ep)$ $H$-converges to $\mathbb{L}_{j,0}$ in $2B_j$ for all $j\in\mathbb{N}$. 
By Lemma~\ref{L:local}~(a) we have $\mathbb{L}_{j,0}=\mathbb{L}_{k,0}$ $\mu$-a.e.~in $B_j\cap B_k$,
and thus we can choose a coefficient field $\mathbb{L}_0\in\mathcal{M}(M,\lambda,\Lambda)$ 
with $\mathbb{L}_0(x)=\mathbb{L}_{j,0}(x)$ for $\mu$-a.e.~$x\in B_j$, $j\in\mathbb{N}$. 

\bigskip
\textit{Step 2.} Proof of (b).\\
Fix $\Omega\subset M$ open, $m>\tfrac{m_0(\Omega)}{\lambda}$, and take sequences $(f_{\ep})\subset L^2(\Omega)$ and $(F_{\ep})\subset L^2(T\Omega)$ with $f_{\ep}\rightharpoonup f_0$ weakly in $L^2(\Omega)$ and $F_{\ep} \to F_0$ in $L^2(T\Omega)$. Let $u_{\ep}\in H^1_0(\Omega)$ be the solution to
\begin{equation*}
  m u_{\ep}-\Div (\mathbb{L}_{\ep}\nabla u_{\ep})=f_{\ep}-\Div F_{\ep}
	\qquad\text{in }H^{-1}(\Omega).
\end{equation*}
We extract a subsequence $E^{\prime\prime}\subset E^{\prime}$ such that
\begin{equation}\label{eq:convueps}
	\begin{cases}
		u_{\ep}\rightharpoonup u_0&\quad\text{in }H^1_0(\Omega),\\
		\mathbb{L}_{\ep}\nabla u_{\ep}\rightharpoonup J_0&\quad\text{in }L^2(T\Omega)
	\end{cases}
\end{equation}
for some $u_0\in H^1(\Omega)$ and $J_0\in L^2(T\Omega)$. We now claim that $u_0$ is the (unique) solution in $H^1_0(\Omega)$ to
\begin{equation}\label{limit:eq}
	m u_0-\Div (\mathbb{L}_0\nabla u_0)=f_0-\Div F_0
	\qquad\text{in }H^{-1}(\Omega)
\end{equation}
and that $J_0=\mathbb{L}_0\nabla u_0$. For the argument we use the covering $(B_j)$ of $M$ described in Step~1. 
Let $\varphi_j\in C_c^{\infty}(M)$ denote a partition of unity subordinate to $(B_j)$, in the sense that 
$\operatorname{supp}\varphi_j\Subset B_j$ and $\sum_{j=1}^{\infty}\varphi_j=1$. Then for every $\varphi\in H^1_0(\Omega)$ and every $j\in\mathbb{N}$
\begin{equation}\label{eq:localized}
  \begin{aligned}
    \int_{\Omega}(\mathbb{L}_{\ep}\nabla(\varphi_ju_{\ep}), \nabla\varphi)
    &=\int_{\Omega}(u_{\ep}\mathbb{L}_{\ep}\nabla\varphi_j, \nabla\varphi)
    +\int_{\Omega}(\varphi_j\mathbb{L}_{\ep}\nabla u_{\ep}, \nabla\varphi)\\
    &=\int_{\Omega}(u_{\ep}\mathbb{L}_{\ep}\nabla\varphi_j, \nabla\varphi)
    +\int_{\Omega}(\mathbb{L}_{\ep}\nabla u_{\ep}, \nabla(\varphi_j\varphi))
    -\int_{\Omega}(\mathbb{L}_{\ep}\nabla u_{\ep},\varphi \nabla\varphi_j)\\
    &=\int_{\Omega}(u_{\ep}\mathbb{L}_{\ep}\nabla\varphi_j, \nabla\varphi)
    +\int_{\Omega}(f_{\ep}-mu_{\ep})\varphi_j\varphi+(F_\ep,\nabla(\varphi_j\varphi))
    \\&\phantom{{}=}-\int_{\Omega}(\mathbb{L}_{\ep}\nabla u_{\ep},\varphi \nabla\varphi_j)\\
    &=\int_{\Omega}(\mathbb{L}_{\ep}(u_{\ep}\nabla\varphi_j),\nabla\varphi)+\int_\Omega (\varphi_j F_\ep, \nabla\varphi)\\
    &\phantom{{}=}+\int_{\Omega}\Big((f_{\ep}-mu_{\ep})\varphi_j+
    ((F_\ep-\mathbb L_\ep \nabla u_\ep),\nabla\varphi_j)\Big)\varphi\\
    &=\int_{\Omega}(\mathbb{L}_{\ep} G_{j,\ep},\nabla\varphi)+\int_\Omega (F_{j,\ep}, \nabla\varphi)+\int_{\Omega} g_{j,\ep}\varphi,
  \end{aligned}
\end{equation}
where
\begin{equation*}
          g_{j,\ep}:=(f_{\ep}-mu_{\ep})\varphi_j+ ((F_\ep-\mathbb L_\ep \nabla u_\ep),\nabla\varphi_j),\qquad 
          G_{j,\ep}:=u_\ep \nabla\varphi_j,\qquad F_{j,\ep}:=\varphi_j F_\ep.
\end{equation*}
Moreover set $v_{j,\ep}:=\varphi_ju_{\ep}$ and  note that $v_{j,\ep}\in H^1_0(B_j)$. 
Since \eqref{eq:localized} holds in particular for all $\varphi\in H^1_0(B_j)$, we infer that $v_{j,\ep}$ is the unique solution in $H^1_0(B_j)$ to 
\begin{equation*}
  -\Div (\mathbb{L}_{\ep}\nabla v_{j,\ep})=g_{j,\ep}-\Div (\mathbb L_\ep G_{j,\ep})-\Div  F_{j,\ep}\qquad\text{in }H^{-1}(B_j).
\end{equation*}
By Step~1 we have $\mathbb L_{\ep}\Hto \mathbb L_0$ on $2B_j$.
Furthermore, from \eqref{eq:convueps}, the compact embedding of $H^1_0(B_j)\subset L^2(B_j)$ 
(which yields $u_\ep\to u_0$ strongly in $L^2(B_j)$), and the convergence properties of $(f_\ep)$ and $(F_\ep)$, we deduce that
\begin{equation}\label{eq:comp}
	\begin{cases}
          v_{j,\ep}\wto v_{j,0}:=\varphi_j u_0&\quad\text{weakly in }H^1(B_j),\\
          g_{j,\ep}\rightharpoonup g_{j,0}:=(f_0-m u_0)\varphi_j+((F_0-J_0),\nabla\varphi_j)&\quad\text{weakly in }L^2(B_j),\\
          G_{j,\ep}\to G_{j,0}:=u_0\nabla\varphi_j&\quad\text{strongly in }L^2(TB_j),\\
          F_{j,\ep}\to F_{j,0}:=\varphi_j F_0&\quad\text{strongly in }L^2(TB_j).
	\end{cases}
\end{equation}
Hence, Lemma~\ref{lemB} implies that $v_{j,0}\in H^1_0(B_j)$ is the weak solution to
\begin{equation*}
  -\Div (\mathbb{L}_{0}\nabla v_{j,0})=g_{j,0}-\Div (\mathbb L_0 G_{j,0})-\Div  F_{j,0}\qquad\text{in }H^{-1}(B_j),
\end{equation*}
and
\begin{equation}\label{eq:flux}
  \mathbb{L}_{\ep}\nabla v_{j,\ep}\rightharpoonup\mathbb{L}_0\nabla v_{j,0}\qquad\text{weakly in }L^2(TB_j).
\end{equation}
Since $\sum_{j=1}^\infty\varphi_j=1$ we deduce that  $\sum_{j=1}^{\infty}\nabla\varphi_j=0$, and thus
\begin{equation*}
  \sum_{j=1}^\infty v_{j,0}=u_0,\qquad 
  \sum_{j=1}^\infty F_{j,0}=F_0,\qquad \sum_{j=1}^\infty G_{j,0}=0,\qquad \sum_{j=1}^\infty g_{j,0}=(f_0-m u_0).
\end{equation*}
In particular, summation of \eqref{eq:flux} yields $\mathbb L_\ep \nabla u_{\ep}\wto J_0=\mathbb L_0\nabla u_0$ weakly in $L^2(T\Omega)$. Moreover, for any test function $\varphi\in C^\infty_c(\Omega)$ we have on the one hand
\begin{align*}
  \int_\Omega (\mathbb L_\ep \nabla u_\ep,\nabla\varphi)=\sum_{j=1}^\infty \int_\Omega (\mathbb L_\ep \nabla v_{j,\ep},\nabla\varphi)\to \sum_{j=1}^\infty \int_\Omega (\mathbb L_0 \nabla v_{j,0},\nabla\varphi)=\int_\Omega(\mathbb L_0\nabla u_0,\nabla\varphi),
\end{align*}
and on the other hand, by summation of \eqref{eq:localized}, and by \eqref{eq:comp},
\begin{align*}
  \int_\Omega (\mathbb L_\ep \nabla u_\ep,\nabla\varphi)&=\sum_{j=1}^\infty \int_\Omega (\mathbb L_\ep d v_{j,\ep},\nabla\varphi)\\
  &=\sum_{j=1}^\infty\int_{B_j}(\mathbb L_\ep G_{j,\ep},\nabla\varphi)+(F_{j,\ep},\nabla\varphi)+g_{j,\ep}\varphi\\
  &\to \sum_{j=1}^d\int_{B_j} (\mathbb L_0G_{j,0}+F_{j,0},\nabla\varphi)+g_{0,j}\varphi\\
  &=\int_{\Omega} (F_{0},\nabla\varphi)+(f_0-mu_0)\varphi.
\end{align*}
The combination of the previous two identities yields \eqref{limit:eq}. Since the latter admits a unique solution, we deduce that the convergence holds for the entire subsequence $E'$. Finally we note that if $H^1_0(\Omega)$ is compactly contained in $L^2(\Omega)$, then we even have $u_{\ep}\to u_0$ strongly in $L^2(\Omega)$. The same conclusion is true if $m\neq 0$ and $f_{\ep}\to f_0$ strongly in $L^2(\Omega)$. To see this, first note that by $\mathbb L_{\ep}\nabla u_\ep\wto \mathbb L_0\nabla u_0$ and Lemma~\ref{lemC} we have
\begin{equation*}
  \int_\Omega (\mathbb L_{\ep}\nabla u_\ep,\nabla u_\ep)\to \int_\Omega (\mathbb L_0 \nabla u_0,\nabla u_0).
\end{equation*}
Thus, since we may pass to the limit in products of weakly and strongly convergent sequences,
\begin{align*}
m \int_{\Omega} \ue^2 
&=m \int_{\Omega} \ue^2 + \int_{\Omega} (\bLe \nabla\ue,\nabla\ue) - \int_{\Omega} (\bLe \nabla\ue,\nabla\ue) \\
&=\int_{\Omega} f_\ep\ue +\int_{\Omega}(F_\ep,\nabla\ue) - \int_{\Omega} (\bLe \nabla\ue,\nabla\ue) \\
&\to\int_{\Omega} f_0 u_0+\int_{\Omega}(F_0,\nabla u_0)  - \int_{\Omega} (\bL_0 \nabla u_0,\nabla u_0)
=m \int_{\Omega} u_0^2.
\end{align*}
Since $m\neq0$, this implies $\|u_\ep\|_{L^2(\Omega)}\to\|u_0\|_{L^2(\Omega)}$, which combined with the weak convergence $u_\ep\wto u_0$ in $L^2(\Omega)$ yields the claimed strong convergence $u_\ep\to u_0$ in $L^2(\Omega)$. This completes the argument for part (b).

\bigskip

\textit{Step 3.} Proof of part (a).\\
Since $m_0(\omega)<0$, we can take $m=0$ in part (\ref{th2:b}) and $H$-convergence immediately follows.
\qed


\subsection{Proofs of Lemma~\ref{190514.01}, Lemma~\ref{local-chart} and Lemma~\ref{L:periodic}}

\begin{proof}[Proof of Lemma~\ref{190514.01}]Let $\overline{\xi}=(\overline{\xi}^1,\dotsc,\overline{\xi}^n),\overline{\eta}=(\overline{\eta}^1,\dotsc,\overline{\eta}^n) \in \mathbb R^n$ and $\xi,\eta\in T_xM$ such that
\begin{equation*}
	\begin{cases}
		\overline{\xi}^i=g(\xi,\tfrac{\partial}{\partial x^i})&\\
		\overline{\eta}^i=g(\eta,\tfrac{\partial}{\partial x^i})&
	\end{cases}
	\qquad\text{for $i=1,\dotsc,n$}.
\end{equation*}
We identify $x \in \Psi^{-1}(U)$ and the corresponding point in $U$.
Since the metric $g(\cdot,\cdot)(x)$ continuously depends on $x$, since $\Psi$ is a diffeomorphism, and because $U\Subset\Psi(\Omega)$, there exists a constant $C>0$ such that
\begin{equation*}
	\tfrac{1}{C}|\overline{\xi}|^2\le\sum_{i,j=1}^ng^{ij}(x)\overline{\xi}^i\overline{\xi}^j=g(\xi,\xi)(x)\le C|\overline{\xi}|^2\qquad\text{and}\qquad \tfrac{1}{C}\leq \rho(x)\leq C,
\end{equation*}
for all $x\in \Psi^{-1}(U)$, where $(g^{ij})$ denotes the inverse of the matrix representation $(g_{ij})$ of $g$ in local coordinates, i.e. $g_{ij}=g(\tfrac{\partial}{\partial x^i},\tfrac{\partial}{\partial x^j})$. Then the uniform ellipticity of $\bL$  yields
\[
A(x)\overline{\xi} \cdot \overline{\xi} =\rho(x) g(\bL \xi,\xi)(x) \ge \lambda \rho(x)g(\xi, \xi)(x) \ge \tfrac{1}{C'}|\overline{\xi}|^2
\]
and
\[
A(x)\overline{\xi} \cdot \overline{\eta} = \rho(x)g(\mathbb{L}\xi,\eta)(x) \le \Lambda \rho(x)|\xi(x)|_g |\eta (x)|_{g} \le C' |\overline{\xi}| |\overline{\eta}|
\]
for some $C'>0$. Thus the statement follows.
\end{proof}

\begin{proof}[Proof of Lemma~\ref{local-chart}]
We prove only $(2)\Rightarrow(1)$ as 
the opposite implication can be proved in the same way. 
Let $f \in L^2({\omega})$ and $\xi \in L^2(T{\omega})$. 
Let  $\ue \in H^1_0({\omega})$ with $\ep > 0$ be
the solution of 
\[-\Div_{g,\mu}( \bLe \nabla_g \ue) = f - \Div_{g,\mu}\xi\qquad\text{in }H^{-1}({\omega}).\]
By \eqref{eq250109}, $\ue$ is the solution to
\[-\Div (A_{\ep} \nabla \ue) = \rho f - \di (\rho F)\qquad\text{in }H^{-1}(U).\]
Since $(A_\ep)$ $H$-converges to $A_0$, 
\begin{equation} \label{eq260109}
\begin{cases}
\ue \wto u_0 &\quad\mbox{weakly in $H^1_0 (U)$},\\
A_\ep \nabla \ue \wto A_0 \nabla u_0 &\quad\mbox{weakly in $L^2 (U;\mathbb{R}^n)$},\\
\end{cases}
\end{equation}
where
\begin{equation} \label{eq270109}
-\di (A_0 \nabla u_0) = \rho f - \di (\rho F)\qquad\text{in }H^{-1}(U).
\end{equation}
By \eqref{eq260109}
\begin{equation} \label{eq280109}
\ue \wto u_0 \qquad \mbox{weakly in $H^1_0 ({\omega},g)$}.
\end{equation}
For any $\eta \in L^2(T{\omega})$ and $\overline{\eta}=(\overline{\eta}^1,\dotsc,\overline{\eta}^n) \in L^2(U;\mathbb{R}^n)$ with $\overline{\eta}^i:=g(\eta,\tfrac{\partial}{\partial x^i})$ for $i=1,\dotsc,n$ we have
\begin{align*}
\int_{\omega} g(\bLe \nabla_g \ue,\eta)\dm 
&=\int_U A_\ep(x) \nabla \ue \cdot \overline{\eta}\, \rd x \to \int_U A_0 (x) \nabla u_0 \cdot \overline{\eta}\, \rd x \qquad \mbox{(as $\ep\to0$)} \\
&=\int_{\omega} g(\bL_0 \nabla_g u_0,\eta)\dm.
\end{align*}
Hence,
\begin{equation} \label{eq290109}
  \bL_\ep \nabla_g\ue \wto \bL_0 \nabla_g u_0 \qquad \mbox{weakly in $L^2 (T{\omega})$}.
\end{equation}

Since \eqref{eq270109} is equivalent to
\[
  -\Div_{g,\mu}(\bL_0 \nabla_g u_0) = f - \Div_{g,\mu}\xi\qquad\text{in }H^{-1}({\omega}),
\]
together with \eqref{eq280109} and \eqref{eq290109}
we arrive at the conclusion.
\end{proof}

\begin{proof}[Proof of Lemma~\ref{L:periodic}]
  The proof is a direct consequence of Lemma~\ref{local-chart} and the well-known fact from periodic homogenization that 
  $A_\ep(x)=A(x,\frac{x}{\ep})$ $H$-converges to $A_{\hom}$, e.g.\ see \cite[Theorem~2.2]{allaire1992}.
\end{proof}

\subsection{Proofs of Lemma~\ref{L:trafo}, Lemma~\ref{prop33108},  and Lemma~\ref{171017.cor}}

\begin{proof}[Proof of Lemma \ref{L:trafo}]
  \textit{Step 1.} Argument for (a)$\Leftrightarrow$(b).\\
  Since $h\colon M_0\to M$ is a diffeomorphism, the integral transformation formula yields for any function $f\in L^1(M,g,\mu)$
  \begin{equation*}
    \int_Mf\,\mathrm{d}\mu=\int_{M_0}(f\circ h)\rho\,\mathrm{d}\mu_0.
  \end{equation*}
	To show the equivalence of statement (a) and (b) it only remains to show
	\begin{equation*}
          \overline{g}(\nabla_{\overline{g}} u,\nabla_{\overline{g}}\varphi)\,\rho=g_0(\mathbb{L}\nabla_{g_0}\overline{u},\nabla_{g_0}\overline{\varphi})
	\end{equation*}
	for any test function $\varphi\in C_c^{\infty}(M)$. To that end we first claim $\nabla_{\overline{g}}u=(dh^{-1})^*\nabla_{g_0}\overline{u}$ (and that the same holds for $\varphi$). Indeed, using the definition of the gradient and the adjoint, we have
	\begin{equation*}
		\overline{g}(\nabla_{\overline{g}}u,\xi)=du(\xi)=d(u\circ h)(dh^{-1}\xi)=g_0(\nabla_{g_0}\overline{u},dh^{-1}\xi)=\overline{g}((dh^{-1})^*\nabla_{g_0}\overline{u},\xi).
	\end{equation*}
	Together with the definition of $\mathbb{L}$ we conclude
	\begin{equation*}
		\overline{g}(\nabla_{\overline{g}} u,\nabla_{\overline{g}}\varphi)\,\rho=\overline{g}((dh^{-1})^*\nabla_{g_0}\overline{u},(dh^{-1})^*\nabla_{g_0}\overline{\varphi})\,\rho=g_0(\mathbb{L}\nabla_{g_0}\overline{u},\nabla_{g_0}\overline{\varphi}).
	\end{equation*}
	\medskip
        
	\textit{Step 2.} Argument for (b)$\Leftrightarrow$(c).\\
	By the definition of $\LimMeasure $ it suffices to show
	\begin{equation*}
		g_0(\mathbb{L}\nabla_{g_0}\overline{u},\nabla_{g_0}\overline{\varphi})=\LimMetric (\nabla_{\LimMetric }\overline{u},\nabla_{\LimMetric }\overline{\varphi})\,\rho.
	\end{equation*}
	We first observe $\mathbb{L}\nabla_{g_0}\overline{u}=\rho\nabla_{\LimMetric}\overline{u}$, which can be seen by the following direct computation, using the definition of $\LimMetric$ and of the gradient:
	\begin{equation*}
		\LimMetric(\mathbb{L}\nabla_{g_0}\overline{u},\xi)=\rho\,g_0(\nabla_{g_0}\overline{u},\xi)=\rho\,d\overline{u}(\xi)=\rho\,\LimMetric(\nabla_{\LimMetric}\overline{u},\xi).
	\end{equation*}
	Again with the definition of the gradient we finally get
	\begin{equation*}
		g_0(\mathbb{L}\nabla_{g_0}\overline{u},\nabla_{g_0}\overline{\varphi})=\rho\,g_0(\nabla_{\LimMetric}\overline{u},\nabla_{g_0}\overline{\varphi})=\rho\,d\overline{\varphi}(\nabla_{\LimMetric}\overline{u})=\rho\,\LimMetric (\nabla_{\LimMetric }\overline{u},\nabla_{\LimMetric }\overline{\varphi}).
		\qedhere
	\end{equation*}
\end{proof}
\medskip

\begin{proof}[Proof of Lemma~\ref{prop33108}]
By construction, there exists a constant $C_0>0$ (only depending on the constant $C$ of Definition~\ref{D:biLip} and the dimension $n$) such that $\mathbb L_\ep\in\mathcal M(M_0,\frac{1}{C_0},C_0)$ and $\frac{1}{C_0}\leq \rho_\ep\leq C_0$ a.e. in $M_0$. Therefore, by weak-$*$ compactness in $L^\infty(M_0)$ and by Theorem~\ref{T1} there exist a subsequence, a density $\rho_0\in L^\infty(M_0)$ satisfying $\frac{1}{C_0}\leq \rho_0\leq C_0$, and a coefficient field $\mathbb L_0\in\mathcal M(M_0,\frac{1}{C_0},C_0)$ such that 
$\rho_\ep\stackrel{*}{\wto}\rho_0$ weak-$*$ in $L^\infty(M_0)$ and $\mathbb L_\ep\stackrel{H}{\to}\mathbb L_0$ in $(M_0,g_0,\mu_0)$ along a subsequence that we do not relabel. This proves statement (a).
\smallskip

Next, we prove statement (b). Set $\overline u_\ep:=u_\ep\circ h_\ep$ and $\overline f_\ep:=f\circ h_\ep$. By Lemma~\ref{L:trafo} (b), \eqref{C:bilip:1} is equivalent to
\begin{equation}\label{LL:eq:1}
  \big(\overline m-\Div_{g_0,\mu_0}(\mathbb L_\ep\nabla_{g_0})\big)\overline u_\ep=\rho_\ep\overline f_\ep-(\rho_\ep m-\overline m)\overline u_\ep\qquad\text{in }H^{-1}(M_0, g_0,\mu_0),
\end{equation}
where $\overline m$ denotes a (sufficiently large) dummy constant that we introduce in order to be able to apply Theorem~\ref{T1}. By a standard energy estimate, $(\overline u_\ep)$ is bounded in $H^1(M_0,g_0,\mu_0)$ and thanks to the compact embedding of $H^1(M_0,g_0,\mu_0)$ in $L^2(M_0,g_0,\mu_0)$ in Assumption~\ref{ass}. Thus there exists $\overline u_0\in H^1_0(M_0,g_0,\mu_0)$ such that $\overline u_\ep\to \overline u_0$ strongly in $L^2(M_0,g_0,\mu_0)$ (for a further subsequence). Moreover, since $f_\ep\to f_0$ strongly in $L^2$ in the sense of \eqref{strong-conv}, $\rho_\ep\stackrel{*}{\wto}\rho_0$ weak-$*$ in $L^\infty(M_0)$, and since $\frac{1}{C_0}\leq \rho_\ep\leq C_0$, we deduce that $\rho_\ep f_\ep\wto \rho_0 f_0$ weakly in $L^2(M_0,g_0,\mu_0)$, and thus we get for the right-hand side in \eqref{LL:eq:1},
\begin{equation*}
  \rho_\ep\overline f_\ep-(\rho_\ep m-\overline m)\overline u_\ep\wto \rho_0f_0-(\rho_0 m-\overline m)\overline u_0\qquad\text{weakly in }L^2(M_0,g_0,\mu_0).
\end{equation*}
Since $\mathbb L_\ep\stackrel{H}{\to}\mathbb L_0$ we conclude with Theorem~\ref{T1} that $\overline u_0$ is a solution to 
\begin{equation}\label{LL:eq:2}
  \big(\overline m-\Div_{g_0,\mu_0}(\mathbb L_0\nabla_{g_0})\big)\overline u_0=\rho_0f_0-(\rho_0 m-\overline m)\overline u_0\qquad\text{in }H^{-1}(M_0, g_0,\mu_0).
\end{equation}
Since this PDE admits a unique solution, we conclude that $\overline u_\ep\wto\overline u_0$ weakly in $H^1(M_0,g_0,\mu_0)$, and thus strongly in $L^2(M_0,g_0,\mu_0)$, for the entire sequence.
By appealing to the equivalence of (b) and (c) in Lemma~\ref{L:trafo}, we deduce from \eqref{LL:eq:2} that $u_0:=\overline u_0$ satisfies \eqref{C:bilip:2}. It remains to argue that $u_\ep\to u_0$ in the sense of \eqref{strong-conv}. To that end let $\psi\in C^\infty_c(M_0)$. Then, since $\overline u_\ep\to u_0$ strongly and $\rho_\ep\wto\rho_0$ weakly in $L^2(M_0,g_0,\mu_0)$,
\begin{align*}
  \int_{M_\ep}u_\ep(\psi\circ h_\ep^{-1})\,\mathrm d\mu_\ep=  \int_{M_0}\overline u_\ep\psi\rho_\ep\,\mathrm d\mu_0 \to \int_{M_0}u_0\psi\rho_0\,\mathrm d\mu_0=\int_{M_0}u_0\psi\,\mathrm d\hat \mu_0.
\end{align*}
Moreover, since $\rho_\ep\stackrel{*}{\wto}\rho_0$ in $L^\infty(M_0)$ we have $\overline u_\ep\rho_\ep\wto u_0\rho_0$ weakly in $L^2(M_0,g_0,\mu_0)$, and thus
\begin{equation*}
  \int_{M_\ep}|u_\ep|^2\,\mathrm d\mu_\ep=  \int_{M_0}\overline u_\ep\, \overline u_\ep\rho_\ep\,\mathrm d\mu_0 \to \int_{M_0}u_0\, u_0\rho_0\,\mathrm d\mu_0=\int_{M_0}|u_0|^2\, d\hat \mu_0.
	\qedhere
\end{equation*}
\end{proof}

\begin{proof}[Proof of Lemma~\ref{171017.cor}]
The argument is similar to the proof of Lemma~\ref{L:HtoSpec}, which itself is based on \cite[Lemma~11.3 and Theorem~11.5]{Zhikov-book}. We only need to treat small changes that come from rewriting the eigenvalue problem on $M_\ep$ as a PDE on the reference manifold $M_0$. For the sake of brevity we only prove that eigenpairs of the Laplace-Beltrami operator on $M_\ep$ converge (up to a subsequence) to an eigenpair of the Laplace-Beltrami operator on $(M_0,\hat g_0,\hat\mu_0)$. The conclusion of the statements of the theorem then follow by appealing to \cite[Lemma~11.3 and Theorem~11.5]{Zhikov-book}.

We first note that for all $k\in\mathbb N$ the sequence $(\lambda_{\ep,k})$ is bounded from above: For the first eigenvalue, \eqref{unifbound} implies
\begin{align*}
	\lambda_{\varepsilon,1}
	&=\inf\Big\{\int_{M_{\varepsilon}}g_{\varepsilon}(\nabla_{g_\ep} u,\nabla_{g_\ep} u)\,\mathrm{d}\mu_{\varepsilon}; u\in H^1_0(M_{\ep}),\,\|u\|_{L^2(M_{\varepsilon})}=1\Big\}\\
	&=\inf\Big\{\int_{M_0}g_0(\mathbb{L}_{\varepsilon}\nabla_{g_0}(u\circ h_{\varepsilon}),\nabla_{g_0}(u\circ h_{\varepsilon}))\,\mathrm{d}\mu_0; u\in H^1_0(M_{\ep}),\,\|u\|_{L^2(M_{\varepsilon})}=1\Big\}\\
	&\le C_0\inf\Big\{\int_{M_0}g_0(\nabla_{g_0}v,\nabla_{g_0}v)\,\mathrm{d}\mu_0; v\in H^1_0(M_{0}),\,\|v\|_{L^2(M_0)}=1\Big\}\\
	&<\infty
\end{align*}
for some constant $C_0>0$ only depending on the constant $C$ in Definition~\ref{D:biLip} and the dimension $n$. The analogue statement for the other eigenvalues can be obtained by the Rayleigh-Ritz method with a similar argument.
Likewise the sequence of the first eigenvalues $(\lambda_{1,\ep})$  is bounded from below by a positive constant. Indeed, for every eigenpair $(\lambda_{\varepsilon},u_{\varepsilon})$ we deduce with Lemma~\ref{L:trafo}, \eqref{unifbound}, and assumption $m_0(M_0)<0$ that there exists constants $C_0,\overline C_0>0$ (only depending on the constant $C$ in Definition~\ref{D:biLip} and the dimension $n$) such that
\begin{align*}
	\lambda_{\varepsilon,1}
	&=\lambda_{\varepsilon}\|u_{\varepsilon,1}\|_{L^2(M_{\varepsilon})}^2=\int_{M_{\varepsilon}}g_{\varepsilon}(\nabla_{\varepsilon} u_{\varepsilon},\nabla_{\varepsilon} u_{\varepsilon})\,\mathrm{d}\mu_{\varepsilon}\\
	&=\int_{M_0}g_0(\mathbb{L}_{\varepsilon}\nabla_{g_0}(u_{\varepsilon}\circ h_{\varepsilon}),\nabla_{g_0}(u_{\varepsilon}\circ h_{\varepsilon}))\,\mathrm{d}\mu_0 \ge \frac{1}{C_0}\int_{M_0} g_0(\nabla_{g_0}(u_{\varepsilon}\circ h_{\varepsilon}),\nabla_{g_0}(u_{\varepsilon}\circ h_{\varepsilon}))\,\mathrm{d}\mu_0\\
	&\ge \frac{1}{C_0}\|u_\ep\|_{L^2(M_0)}^2\inf\Big\{\int_{M_0} g_0(\nabla_{g_0}v,\nabla_{g_0}v)\,\mathrm{d}\mu_0;\,v\in H^1_0(M_0),\,\|v\|_{L^2(M_0)}^2=1\,\Big\}\\
	&\ge \frac{1}{C_0}\|u_\ep\|_{L^2(M_0)}^2\inf\Big\{\int_{M_0} g_0(\nabla_{g_0}v,\nabla_{g_0}v)\,\mathrm{d}\mu_0;\,v\in H^1_0(M_0),\,\|v\|_{L^2(M_0)}^2=1\,\Big\}\\
	&\ge \overline C_0>0,
\end{align*}
where in the last step we in particular used that $m_0(M_0)<0$.
Now, we fix $k\in\mathbb N$ and let $(\lambda_{\ep,k}, u_{\ep,k})$ be an eigenpair, i.e.,
\begin{equation}\label{LLL:eq1}
  -\Delta_{g_\ep,\mu_\ep}u_{\ep,k}=\lambda_{\ep,k}u_{\ep,k}\qquad\text{in }H^{-1}(M_\ep,g_\ep,\mu_\ep).
\end{equation}
By passing to a subsequence we may assume that $\lambda_{\ep,k}\to\overline\lambda$ as $\ep\to 0$ for some $\overline\lambda$. Moreover, w.l.o.g. we may assume that $u_{\ep,k}$ is normalized in the sense that $\int_{M_\ep}|u_{\ep,k}|^2\mathrm{d}\mu_\ep=1$. Testing \eqref{LLL:eq1} with $u_{\ep,k}$ then shows that $\|u_{\ep,k}\|_{H^1(M_\ep)}$ is bounded by a constant independent of $\ep$. We conclude that $\overline u_{\ep,k}:=u_{\ep,k}\circ h_\ep$ is bounded in $H^1(M_0,g_0,\mu_0)$ and we thus may pass to a further subsequence with $\overline u_{\ep,k}\wto \overline u$ weakly in $H^1(M_0,g_0,\mu_0)$ and strongly in $L^2(M_0,g_0,\mu_0)$, thanks to the compact embedding of $H^1(M_0,g_0,\mu_0)$ in $L^2(M_0,g_0,\mu_0)$ in Assumption~\ref{ass}. Note that this implies also that $u_{\ep,k}\to \overline u$ strongly in $L^2$ in the sense of \eqref{strong-conv}. We conclude that the right-hand side of \eqref{LLL:eq1} is strongly convergent to $\overline \lambda\overline u$. Thus, by appealing to Lemma~\ref{prop33108} (b) we conclude that 
\begin{equation*}
  -\Delta_{\hat g_0,\hat\mu_0}\overline u=\overline \lambda\overline u \qquad\text{in }H^{-1}(M_0,\hat g_0,\hat \mu_0).
\end{equation*}
Since $\|\overline u\|_{L^2(M_0,\hat g_0,\hat\mu_0)}=1$ by construction, we conclude that $(\overline\lambda,\overline u)$ is an eigenpair of the Laplace-Beltrami operator on $(M_0,\hat g_0,\hat \mu_0)$.
\end{proof}

\appendix
\section{Proofs of auxiliary results}
\subsection{Proof of Lemma \ref{ex:appendix}}
        We refer to \cite{neukamm-stelzig} for a similar result in a nonlinear, variational setting. 
        
	{\it Step 1.} Continuity of $\nabla\phi_i$ in the first argument.\\
	Consider a sequence $(x_j)$ in $\mathbb{R}^n$ converging to some $x_0\in\mathbb{R}^n$. For simplicity we set
	\begin{equation*}
		\phi_i^j:=\phi_i(x_j,\cdot)
		\qquad\text{and}\qquad
		A^j:=A(x_j,\cdot)
	\end{equation*}
	as well as
	\begin{equation*}
		\phi_i^0:=\phi_i(x_0,\cdot)
		\qquad\text{and}\qquad
		A^0:=A(x_0,\cdot).
	\end{equation*}
	First we note that the continuity of $A$ in the first argument gives $A^j\to A^0$ a.e.~on $Y$ and by uniform ellipticity we have $|A^j|\le\Lambda$ a.e.~on $Y$. Thus we can conclude
	\begin{equation}\label{ex.app.eq}
		\int_Y|A^j-A^0|^p\to0
	\end{equation}
	for $1<p<\infty$.
	
	Now we claim the convergence of $\nabla\phi_i^j$. By \eqref{perhom.b} we have
	\begin{equation*}
		-\nabla\cdot A^j(\nabla\phi_i^j-\nabla\phi_i^0)=\nabla\cdot\big((A^j-A^0)(\nabla\phi_i^0+e_i)\big).
	\end{equation*}
	The uniform ellipticity of $A^j$ allows to estimate
	\begin{equation*}
		\int_Y|\nabla\phi_i^j-\nabla\phi_i^0|^2\le\tfrac{1}{\lambda}\int_Y\bigl|(A^j-A^0)(\nabla\phi_i^0+e_i)\bigr|^2.
	\end{equation*}
	By Meyer's estimate there is $2<q<\infty$ and $C>0$ such that $\int_Y|\nabla\phi_i^0|^q\le C\int_Y|A^0e_i|^q$ and thus, for $p=\tfrac{q}{q-2}$ we have
	\begin{equation*}
		\|\nabla\phi_i^j-\nabla\phi_i^0\|_{L^2(Y)}
		\le\tfrac{1}{\sqrt{\lambda}}\|A^j-A^0\|_{L^p(Y)}\bigl(\|\nabla\phi_i^0\|_{L^q(Y)}+1\bigr)
	\end{equation*}
	and \eqref{ex.app.eq} implies $\|\nabla\phi_i^j-\nabla\phi_i^0\|_{L^2(Y)}\to0$.
	
	{\it Step 2.} $H$-convergence to $A_{\text{hom}}$.\\
  Fix $r\in\R$. By Theorem \ref{T1} there exists a subsequence (not relabeled) s.t.~$(A_\eps)$ $H$-converges to some uniformly elliptic coefficient field $A_0$ on $\R^n$. Let $B\subset\R^n$ denote an arbitrary ball and let $u_\eps\in H^1(B)$ denote the unique weak solution to 
  \begin{equation*}
    \left\{\begin{aligned}
      -\nabla\cdot A_\eps\nabla u_\eps&=0&&\text{ in }B,\\
      u_\eps&=x_i&&\text{on }\partial B.
    \end{aligned}\right.
  \end{equation*}
  Then $A_\eps\Hto A_0$ implies that $u_\eps\wto u_0$ weakly in $H^1(B)$, where $u_0$ is the unique weak solution to
  \begin{equation*}
    \left\{\begin{aligned}
        -\nabla\cdot A_0\nabla u_0&=0&&\text{ in }B,\\
        u_0&=x_i&&\text{on }\partial B.
      \end{aligned}\right.
  \end{equation*}
  For $k\in\mathbb N$ let $\eta_k\in C^\infty_c(B)$ be a cut-off function with $\eta_k=1$ in $B_k:=\{x\in B\,:\,\mbox{dist}(x,\partial B)>\frac{1}{k}\}$ and consider
  \begin{equation*}
    v_{\eps,k}:=x_i+\eps\phi_i(x,\tfrac{x+r}{\eps})\eta_k(x).
  \end{equation*}
  Then $(v_{\eps,k})$ converges as $\eps\to0$ to $v_0(x):=x_i$ weakly in $H^1(B)$ and strongly in $L^2(B)$, and a direct computation shows that
  \begin{equation*}
    \nabla v_{\eps,k}(x)=(e_i+\nabla\phi_i(x,\tfrac{x+r}{\eps}))+(\eta_k-1)\nabla\phi_i(x,\tfrac{x+r}{\eps})+\eps\phi_i(x,\tfrac{x+r}{\eps})\nabla \eta_k(x).
  \end{equation*}
  and thus for $w_{\eps,k}:=u_\eps-v_{\eps,k}\in H^1_0(B)$ we have (by appealing to the equation for $u_\eps$ and for $\phi_i$)
  \begin{align*}
    \int_B A_\eps\nabla w_{\eps,k}\cdot\nabla w_{\eps,k}&=-\int_B A_\eps\nabla v_{\eps,k}\cdot\nabla w_{\eps,k}\\
    &=-\int_B A(x,\tfrac{x+r}{\eps})(e_i+\nabla\phi_i(x,\tfrac{x+r}{\eps}))\cdot\nabla w_{\eps,k}\,\mathrm{d}x\\
    &\phantom{{}=}-\int_BA_{\eps}\big((\eta_k-1)\nabla\phi_i(x,\tfrac{x+r}{\eps})+\eps\phi_i(x,\tfrac{x+r}{\eps})\nabla \eta_k(x)\big)\cdot\nabla w_{\eps,k}\,\mathrm{d}x\\
    &\leq C(\Lambda)\int_{S_k}\big(|\nabla\phi_i(x,\tfrac{x+r}{\eps})|+\eps|\phi_i(x,\tfrac{x+r}{\eps})|\|\nabla\eta_k\|_{L^{\infty}(S_k)}\big)|\nabla w_{\eps,k}|\,\mathrm{d}x
  \end{align*}
  for some constant $C(\Lambda)>0$, where $S_k:=B\setminus B_k$.
  The left-hand side is bounded from below by $\lambda\int_B|\nabla w_{\eps,k}|^2$, and thus (by appealing to the Cauchy-Schwarz inequality), we deduce that
  \begin{equation*}
    \int_B|\nabla w_{\eps,k}|^2\leq C(\lambda,\Lambda)\int_{S_k}|\nabla\phi_i(x,\tfrac{x+r}{\eps})|^2+\big(\eps|\phi_i(x,\tfrac{x+r}{\eps})|\|\nabla\eta_k\|_{L^{\infty}(S_k)}\big)^2\,\mathrm{d}x.
  \end{equation*}
  Since $(|\nabla\phi_i(\cdot,\tfrac{\cdot+r}{\eps})|^2)$ is equi-integrable and $|S_k|\to 0$ for $k\to\infty$, we conclude that
  \begin{equation*}
    \limsup\limits_{k\to\infty}\limsup\limits_{\eps\to 0}\int_B|\nabla w_{\eps,k}|^2=0,
  \end{equation*}
  and thus there exists a diagonal sequence $(k_\eps)$ (with $k_\eps\to \infty$ as $\eps\to 0$) such that $w_\eps:=w_{k_\eps,\eps}$ satisfies $\nabla w_{\eps}\to 0$ strongly in $L^2(B)$. Hence, with  $v_\eps:=v_{\eps,k_\eps}$, we conclude that $\nabla u_\eps-\nabla v_\eps\to 0$ in $L^2(B)$.  On the other hand, since $v_\eps\to v_0$ strongly in $L^2(B)$, we conclude that $\nabla u_0=\nabla v_0=e_i$. Moreover, the $H$-convergence of $(A_\eps)$ to $A_0$ implies $A_\eps\nabla u_\eps\wto A_0\nabla u_0=A_0e_i$ weakly in $L^2(B)$, and thus (using $\nabla u_\eps-\nabla v_\eps\to 0$) we have $A_\eps\nabla v_\eps\wto A_0e_i$ weakly in $L^2(B)$.
	
	On the other hand for any $\varphi\in C^\infty_c(B)$ and $\eps>0$ small enough, we have $\varphi(x)\nabla v_{\eps}(x)=\varphi(x)(e_i+\nabla\phi_i(x,\frac{x+r}{\eps}))$, and thus by periodicity
  \begin{align*}
    \int \varphi A_{\eps}\nabla v_\eps
		&=\int \varphi(x) A(x,\tfrac{x+r}{\eps})(e_i+\nabla\phi_i(x,\tfrac{x+r}{\eps}))\,\mathrm{d}x\\
		&=\int \varphi(x) A(x,\tfrac{x}{\eps}+r_\eps)(e_i+\nabla\phi_i(x,\tfrac{x}{\eps}+r_\eps))\,\mathrm{d}x,
	\end{align*}
	where $r_\eps\in Y$ is defined by the identity $\frac{r}{\eps}=k+r_\eps$ for some $k\in \mathbb Z^d$. We write that expression in the following way:
	\begin{align*}
		&\int \varphi(x) A(x,\tfrac{x}{\eps}+r_\eps)(e_i+\nabla\phi_i(x,\tfrac{x}{\eps}+r_\eps))\,\mathrm{d}x\\
    &\qquad=\int \varphi(x-r_\eps) A(x-r_{\varepsilon},\tfrac{x}{\eps})(e_i+\nabla\phi_i(x-r_{\varepsilon},\tfrac{x}{\eps}))\,\mathrm{d}x\\
		&\qquad=\sum_{z\in\mathbb{Z}^n}\varepsilon^n\int_Y\varphi(\varepsilon z+\varepsilon y-r_\eps) A(\varepsilon z+\varepsilon y-r_{\varepsilon},y)(e_i+\nabla\phi_i(\varepsilon z+\varepsilon y-r_{\varepsilon},y))\,\mathrm{d}y.
  \end{align*}
	Since $(r_\eps)$ is a bounded sequence in $Y\subset\R^n$ we may pass to a subsequence (not relabeled) such that $r_\eps\to r_0$ in $Y$ for some $r_0\in\overline Y$. This implies that $\varphi(\cdot+\varepsilon y-r_\eps)\to \varphi(\cdot-r_0)$ strongly in $L^2(U)$ for any $U\subset \R^n$ open and bounded and every $y\in Y$. On the other hand by Step 1 we we have $A^j\nabla\phi_i^j\to A^0\phi_i^0$ in $L^1(Y)$ and thus we get
	\begin{align*}
		&\sum_{z\in\mathbb{Z}^n}\varepsilon^n\int_Y\varphi(\varepsilon z+\varepsilon y-r_\eps) A(\varepsilon z+\varepsilon y-r_{\varepsilon},y)(e_i+\nabla\phi_i(\varepsilon z+\varepsilon y-r_{\varepsilon},y))\,\mathrm{d}y\\
		&\qquad\to\int_{\mathbb{R}^n}\varphi(x-r_0)\int_YA(x-r_0,y)(e_i+\nabla\phi_i(x-r_0,y))\,\mathrm{d}y\,\mathrm{d}x\\
		&\qquad=\int_{\mathbb{R}^n}\varphi(x-r_0)A_{\text{hom}}(x-r_0)e_i\,\mathrm{d}x\\
		&\qquad=\int_{\mathbb{R}^n}\varphi(x)A_{\text{hom}}(x)e_i\,\mathrm{d}x,
	\end{align*}
	and we conclude that $\int \varphi(A_0-A_{\hom})e_i=0$ for all $\varphi\in C^\infty_c(B)$, which gives $A_0=A_{\hom}$ a.e.~in $B$. Since $B$ is an arbitrary ball, we conclude that $A_0=A_{\hom}$ a.e.~in $\R^n$. By uniqueness, we conclude that $(A_\eps)$ $H$-convergence to $A_{\hom}$ for the entire sequence.

\subsection{Proof of Lemma~\ref{L:HtoMo}}
We first recall the definition of Mosco-convergence:
\begin{definition}[Mosco-convergence] 
We say that $(\cEe)$ Mosco-converges to $\cE_0$ as $\ep\to0$
if the following two conditions are satisfied.
\begin{itemize}
\item[\rm(i)] If $\ue \wto u_0$ weakly in $L^2(M)$, then
\[
\liminf_{\ep\to0} \cEe(\ue) \ge \cE_0(u).\]
\item[\rm(ii)] For any $v \in L^2(M)$ there exists $(v_\ep) \subset L^2(M)$ such that
\[\limsup_{\ep\to0} \cEe(v_{\varepsilon}) \le \cE_0(v).\]
\end{itemize} 
\end{definition}
For the proof of Lemma~\ref{L:HtoMo} we recall that Mosco-convergence is equivalent to resolvent convergence of the operator associated with the Dirichlet form $\cEe$. More precisely, for $\ep\geq 0$ consider $\cLe\colon H^1_0(M)\to H^{-1}(M)$, $\cLe u:=-\Div_{g,\mu}(\mathbb L_\ep \nabla_gu)$ and  denote for $\lambda>0$ by $G^\lambda_\ep:=(\lambda+\cLe)^{-1}\colon L^2(M)\to H^1_0(M)$ the associated resolvent. 

\begin{lemma}[Theorem 2.4.1  \cite{Mosco1994}]\label{lemma90309}
The following two conditions are equivalent.
\begin{itemize}
\item[\rm(i)] $(\cE_\ep)$ Mosco-converges to $\cE_0$.
\item[\rm(ii)] For any $\lambda>0$, $(G^\lambda_\ep)$ converges to $G^\lambda_0$ in 
the strong operator topology of $L^2(M)$.
\end{itemize} 
\end{lemma}

\begin{proof}[Proof of Lemma~\ref{L:HtoMo}]
We apply Lemma \ref{lemma90309}.
Let $\lambda>0$, $f_\ep\to f_0$ in $L^2(M)$, and $\ue := G^\lambda_\ep f_\ep$.
Since $(\bLe)$  $H$-converges to $\bL_0$ in $M$, Theorem~\ref{T1} implies that $u_\ep\to u_0:=G^\lambda_0f_0$ strongly in $L^2(M)$.
\end{proof}


\section*{Acknowledgments}
HH and SN acknowledge support by the DFG in the context of TU Dresden’s Institutional Strategy ``The Synergetic University''.
JM gratefully acknowledges support by Grants-in-Aid for Scientific Research 16KT0129.
A substantial part of this work was done while JM was visiting Technische Universit\"at Dresden. 
He expresses his warmest thanks to the institution.
\bibliographystyle{plain}

\begin{thebibliography}{10}

\bibitem{Aizenbud1982}
Boris M.~Aizenbud and Nahum D.~Gershon. 
\newblock Diffusion of molecules on biological membranes of nonplanar form---a theoretical study. 
\newblock {\em Biophys. J.},  38:287–293,  1982.

\bibitem{allaire1992}
Gr{\'e}goire Allaire.
\newblock Homogenization and two-scale convergence.
\newblock {\em SIAM Journal on Mathematical Analysis}, 23(6):1482--1518, 1992.


\bibitem{ChenCroydonKumagai2015}
Zhen-Qing Chen, David~A.~Croydon, and Takashi Kumagai.
\newblock Quenched invariance principles for random walks and elliptic
  diffusions in random media with boundary.
\newblock {\em Annals of Probability}, 43(4):1594--1642, 2015.

\bibitem{Francfort1986}
Gilles~A Francfort and Fran{\c{c}}ois Murat.
\newblock Homogenization and optimal bounds in linear elasticity.
\newblock {\em Archive for Rational mechanics and Analysis}, 94(4):307--334,
  1986.

\bibitem{Francfort2009}
Gilles~A. Francfort, Fran{\c{c}}ois Murat, and Luc Tartar.
\newblock Homogenization of monotone operators in divergence form with
  x-dependent multivalued graphs.
\newblock {\em Annali di Matematica Pura ed Applicata}, 188(4):631--652, 2009.

\bibitem{Fukaya1987}
Kenji Fukaya
\newblock Collapsing of Riemannian manifolds
and eigenvalues of Laplace operator.
\newblock {\em Invent. math.}, 87:517--547, 1987.

\bibitem{Grigoryan2009}
Alexander Grigor'yan.
\newblock {\em Heat kernel and analysis on manifolds}.
\newblock Number~47 in AMS/IP Studies in Advanced Mathematics. American
  Mathematical Society, 2009.
  
\bibitem{Giglietal2015}
Nicola Gigli, Andrea Mondino and Giuseppe Savar\'e.
\newblock Convergence of pointed non-compact metric measure spaces
and stability of Ricci curvature bounds and heat flows.
\newblock {\em Proc. London Math. Soc.}
111(3):1071--1129, 2015

\bibitem{Hino1998}
Masanori Hino.
\newblock Convergence of non-symmetric forms.
\newblock {\em Kyoto Journal of Mathematics}, 38(2):329--341, 1998.


\bibitem{Jacobesn1983}
Ken Jacobson, Akira Ishihara and Richard Inman.
\newblock Lateral Diffusion of Proteins in Membranes
\newblock {\em Annual Review of Physiology} 1(49):163--175, 1987

\bibitem{Zhikov-book}
Vasilii Vasil{\'e}vich Jikov, Sergei M.~Kozlov, and Olga Arsen{\'e}vna Oleinik.
\newblock Homogenization of differential operators and integral functionals.
\newblock {\em Springer Science \& Business Media}, 2012.

\bibitem{Jost2011}
J{\"u}rgen Jost.
\newblock {\em Riemannian geometry and geometric analysis}.
\newblock Universitext. Springer, 6th~edition, 2011.


\bibitem{KasueKumura1994}
Atsushi Kasue and Hironori Kumura.
\newblock Spectral convergence of Riemannian manifolds.
\newblock {\em T\^{o}hoku Math. J}, 46(2):147--179, 1994.

\bibitem{KasueKumura1996}
Atsushi Kasue and Hironori Kumura.
\newblock Spectral convergence of Riemannian manifolds II.
\newblock {\em  T\^{o}hoku Math. J}, 48(2):71--120, 1996.



\bibitem{Kasue2017}
Atsushi Kasue.
\newblock Convergence of Dirichlet forms induced on boundaries of transient
  networks.
\newblock {\em Potential Analysis}, 47(2):189--233, 2017.

\bibitem{Kolesnikov2008}
Alexander~V.~Kolesnikov.
\newblock Weak convergence of diffusion processes on Wiener space.
\newblock {\em Probability Theory and Related Fields}, 140(1--2):1--17, 2008.

\bibitem{KuwaeShioya2003}
Kazuhiro Kuwae and Takashi Shioya.
\newblock Convergence of spectral structures: a functional analytic theory and
  its applications to spectral geometry.
\newblock {\em Communications in Analysis and Geometry}, 11:599--673, 2003.


\bibitem{KuwaeShioya2008}
Kazuhiro Kuwae and Takashi Shioya.
\newblock Variational convergence over metric spaces
\newblock {\em Trans. Amer. Math. Soc.}, 360:35-75, 2008.


\bibitem{KuwaeUemura1997}
Kazuhiro Kuwae and Toshihiro Uemura.
\newblock Weak convergence of symmetric diffusion processes.
\newblock {\em Probability Theory and Related Fields}, 109(2):79--91, 1997.

\bibitem{LanciaMoscoVivaldi2008}
Maria~Rosaria Lancia, Umberto Mosco, and Maria~Agostina Vivaldi.
\newblock Homogenization for conductive thin layers of pre-fractal type.
\newblock {\em Journal of Mathematical Analysis and Applications},
  341(1):354--369, 2008.

\bibitem{Lobus2015}
J\"{o}rg-Uwe L\"{o}bus.
\newblock Mosco type convergence of bilinear forms and weak convergence of
  n-particle systems.
\newblock {\em Potential Analysis}, 43(2):241--267, 2015.

\bibitem{Masamune2011}
Jun Masamune.
\newblock Mosco-convergence and Wiener measures for conductive thin boundaries.
\newblock {\em Journal Mathematical Analysis and Applications},
  384(2):504--526, 2011.

\bibitem{Mosco1994}
Umberto Mosco.
\newblock Composite media and asymptotic Dirichlet forms.
\newblock {\em Journal of Functional Analysis}, 123(2):368--421, 1994.

\bibitem{MoscoVivaldi2009}
Umberto Mosco and Maria~Agostina Vivaldi.
\newblock Fractal reinforcement of elastic membranes.
\newblock {\em Archive for Rational Mechanics and Analysis}, 194(1):49--74,
  2009.

\bibitem{MT}
Fran{\c{c}}ois Murat and Luc Tartar.
\newblock H-convergence.
\newblock In Andrej Cherkaev and Robert Kohn, editors, {\em Topics in the
  Mathematical Modelling of Composite Materials}, volume~31 of {\em Progress in
  Nonlinear Differential Equations and Their Applications}, pages 21--43.
  Birkh{\"{a}}user, 1997.

\bibitem{neukamm-lecture-notes}
Stefan Neukamm.
\newblock An introduction to the qualitative and quantitative theory of
  homogenization.
\newblock {\em Interdisciplinary Information Sciences}, 24(1):1--48, 2018.

\bibitem{neukamm-stelzig}
Stefan Neukamm, and Philipp Emanuel Stelzig. 
\newblock On the interplay of two-scale convergence and translation.
\newblock {\em Asymptotic Analysis}, 71(3):163--183, 2011.

\bibitem{neuss-radu}
Maria Neuss-Radu, and Willi J{\"a}ger. 
\newblock Effective transmission conditions for reaction-diffusion processes in domains separated by an interface.
\newblock {\em SIAM Journal on Mathematical Analysis}, 39(3):687--720, 2007.


\bibitem{Papanicolaou1979}
George~C. Papanicolaou and S.R.~Srinivasa Varadhan.
\newblock Boundary value problems with rapidly oscillating random coefficients.
\newblock In {\em Random fields}, volume 1, number 27 of {\em Colloquia
  Mathematica Societatis J{\'{a}}nos Bolyai}, pages 835--873. Esztergom, 1979.


\bibitem{Sbalzarini}
Ivo F.~Sbalzarini, Arnold Hayer, Ari Helenius, and Petros Koumoutsakos.
\newblock Simulations of (an) isotropic diffusion on curved biological surfaces
\newblock {\em Biophysical journal}, 90(3):878--885, 2006.

\bibitem{Waurick}
Marcus Waurick.
\newblock Nonlocal H-convergence.
\newblock {\em Calculus of Variations and Partial Differential Equations}, 57(6):159, 2018.

\end{thebibliography}

\end{document}